\documentclass[11pt]{amsart}
\usepackage{amscd}
\usepackage{amssymb}
\usepackage{mathrsfs}
\usepackage{amsfonts}
\usepackage{amsmath}
\usepackage[all]{xy}
\newtheorem{theorem}{Theorem}[section]
\newtheorem{lemma}[theorem]{Lemma}
\newtheorem{proposition}{Proposition}[section]

\theoremstyle{definition}
\newtheorem{definition}[theorem]{Definition}
\newtheorem{example}[theorem]{Example}

\newtheorem{remark}[theorem]{Remark}
\numberwithin{equation}{section}

\newcommand{\blankbox}[2]

\begin{document}
\title{Schr\"odinger-Virasoro Lie $H$-pseuoalgebras}

\author{Zhixiang Wu}
\address{School of  Mathematical Sciences, Zhejiang University,
Hangzhou, 310027, P.R.China} \email{wzx@zju.edu.cn}
\thanks{The author is supported by ZJNSF(No. LY17A010015, LZ14A010001), and NNSFC (No. 11871421)}





\begin{abstract}   To solve the problem how to combine Virasoro Lie algebra and Sch\"odinger Lie algebra into a Lie algebra,
we introduce the notion of $m$-type Schr\"odinger-Virasoro Lie conformal algebras and the notion of
  Schr\"odinger-Virasoro Lie ${\bf k}[s]$-pseudoalgebras.  Then we solve the above question by classifying  $m$-type Sch\"dinger-Virasoro Lie conformal algebras and Schr\"odinger-Virasoro Lie ${\bf k}[s]$-pseudoalgebras. Meanwhile, we classify  all Leibniz ${\bf k}[s]$-pseudoalgebras of rank two.
\end{abstract}
\maketitle

\baselineskip=16pt
\section{Introduction}

Let Vect$(S^1)$ be the Lie algebra of $C^{\infty}$-vector fields on the circle $S^1$. Then  an element of Vect$(S^1)$ can be described by the vector field $f(z)\partial_z$, where $f(z)\in\mathbb{C}[z,z^{-1}]$ is a Laurent polynomial. Let $L_n=-z^{n+1}\partial_z$. Then $\{L_n|n\in\mathbb{Z}\}$ is a basis of Vect$(S^1)$ satisfying $[L_n, L_m]=(n-m)L_{n+m}$. The central extension of Vect$(S^1)$ is called the {\it Virasoro algebra} and Vect$(S^1)$ is called {\it Virasoro algebra without center}, or {\it Witt algebra}.
A {\it Schr\"odinger algebra}  is the Lie algebra of differential operators $\mathscr{D}$ on $\mathbb{R}\times \mathbb{R}^d$ of order at most one, satisfying: $(2\mathscr{M}\partial_t-\triangle_d)(\psi)=0\Rightarrow (2\mathscr{M}\partial_t-\triangle_d)(\mathscr{D}\psi)=0.$
An amalgamation Lie algebra of Vect$(S^1)$   and the Schr\"odinger algebra for $d=1$ is called {\it Schr\"odinger-Virasoro algebra}. It is
an extension of Vect$(S^1)$  by a nilpotent Lie algebra formed with a bosonic current of weight $\frac32$ and a bosonic current of weight one. In 1994, M. Henkel (\cite{H}) introduced the Schr\"odinger-Virasoro algebra while he was trying to apply the conformal
field theory to models of statistical physics which either undergo a dynamics,
whether in or out of equilibrium, or are no longer isotropic. It was shown that the Schr\"odinger-Virasoro Lie algebra is a symmetry algebra for many statistical physics models undergoing a dynamics with dynamical exponent $z=2$. In \cite{U},  J. Unterberger gave a representation of the Schr\"odinger-Virasoro algebra by using vertex algebras. In the same paper, he introduced an extension of the Schr\"odinger-Virasoro Lie algebra, called it {\it extended Schr\"odinger-Virasor Lie algebra}. The Schr\"odinger-Virasoro Lie algebra has received more attentions in recent years (cf. \cite{UR} and references therein).

Following \cite{Ka}, Y. Su and L. Yuan defined a Schr\"odinger-Virasoro Lie conformal algebra $SV$ and an extended Schr\"odinger-Virasoro Lie conformal algebra $\widetilde{SV}$ by using the Schr\"odinger-Virasoro Lie algebra and the extended Schr\"odinger-Virasoro Lie algebra defined in [U] respectively (\cite{SY}). $SV$ and $\widetilde{SV}$
are extensions of the Virasoro Lie conformal algebra introduced in [Ka].

A natural question is whether there are other Lie algebras which are extensions of Vect$(S^1)$ by using different weight bosonic currents.  This question has been received more attentions in recent years (\cite{CKW}, \cite {UR}, \cite{HY}, \cite{WXX}). In this paper, we will give a complete answer of this question in terms of pseudoalgebra language.  We call the counterpart of the extensions of the Virasoro conformal algebra {\it a Schr\"odinger-Virasoro Lie $H$-pseudoalgebra}, where $H={\bf k}[s]$. From the definition of $\widetilde{SV}$,  a Schr\"odinger-Virasoro Lie $H$-pseudoalgebra $L$ should be a free $H$-module of rank four. Suppose that $\{e_0, e_1, e_2, e_3\}$ is a basis of $L$. Then $L_0=He_0$ is the Virasoro conformal Lie algebra and $He_i$ for each $i(\geq 1)$ is a nontrivial representation of the Virasoro conformal algebra $He_0$.
The nontrivial representation $He_i$ of $He_0$ was characterized in \cite{CK} and uniquely determined by two parameters $\lambda_i$ and $\kappa_i$. Thus
Schr\"odinger-Virasoro Lie $H$-pseudoalgebra is determined by a triple $\{(\lambda_i,\kappa_i)|1\leq i\leq 3\}$.
 For example, the extended Schr\"odinger-Virasoro Lie conformal algebra $\widetilde{SV}$ defined in \cite{SY} is a $\mathbb{C}[s]$-pseudoalgebra with a triple $\{ (\frac12, 0), (0, 0), (0, 0)\}$. By the way, the parameters $\lambda_i$ in $\widetilde{SV}$ are determined by the conformal weight of generators of the extended Schr\"odinger-Virasoro Lie algebra in \cite{U}. To include algebras in \cite{HY} and \cite{SY}, we use an asymmetric element $\alpha_m'=\sum\limits_{0\leq i<\leq m}w_{ij}(s^{(i)}\otimes s^{(j)}-s^{(j)}\otimes s^{(i)})\in H^{\otimes 2}$ to replace the element $\alpha=1\otimes \partial-\partial\otimes 1$ in $\widetilde{SV}$ in the definition of the Schr\"odinger-Virasoro Lie $H$-pseudoalgebras  (We use $s$ to replace $\partial$ in this paper). For more precise description, see Definition \ref{def25} in Section 3.
 Thus the above question is equivalent to find proper parameter pairs $(\lambda_i,\kappa_i)$  and $\alpha_m'$ to make the above free $H$-module of rank four into Lie $H$-pseudoalgebras with a proper pseudobracket.  In fact, we have determined all such pairs $(\lambda_i,\kappa_i)$ and $\alpha_m'$, which are the main tasks of the present paper.

The paper is organized as follows. In Section 2, we classify all Leibniz $H$-pseudoalgebras of rank two, where $H={\bf k}[s]$. Meanwhile we determine all Leibniz $H$-pseudoalgebras, which are extension $Vect(S^1)$ by a rank one Lie $H$-pseudoalgebra. The central extension of a Lie (Leibniz) $H$-pseudoalgebra $L$  is measured by a cohomological group $H^2(L,M)$ (respectively $Hl^2(L,M)$)(see \cite{BDK} and \cite{Wu2}), where $M$ is a representation of $L$. Using the classification of rank two Leibniz $H$-pseudoalgebras, we get $Hl^2(L,H)=H^2(L,H)=H$ when $L$ is the Virasoro  Lie conformal algebra and $Hl^2(L,H)=H^{\otimes 2}\neq H^2(L,H)=\{a\in H^{\otimes 2}|a=-(12)a\}$ when $L$ is a trivial $H$-pseudoalgebra.
With the preparation of Section 2, we introduce Schr\"odinger-Virasoro Lie $H$-pseudoalgebras in Section 3. Then we  give some basic properties of these algebras. In Section 4, we determine the subalgebra of Schr\"odinger-Virasoro Lie ${\bf k}[s]$-pseudoalgebras generated by $e_0, e_1, e_2$. We call this subalgebra an {\it $m$-type Schr\"odinger-Virasoro Lie conformal algebra, or $H$-pseudoalgebra}, where $m\in\mathbb{N}$,  the set of all nonnegative integer numbers. At first glance from the definition, one may think that there are infinitely many Schr\"odinger-Virasoro Lie $H$-pseudoalgebras as $m\in\mathbb{N}$. However, we prove that $0\leq m\leq 3$ and there are only five kinds of  $m$-type Schr\"odinger-Virasoro Lie conformal algebras. By using the annihilation Lie algebras of these $m$-type Schr\"odinger-Virasoro Lie conformal algebras, one obtains a series of infinite-dimensional Lie algebras $SV_q$, which are extensions of Vect$(S^1)$ (see Example \ref{ex33} below). For example, let $SV_\rho$ be an infinite-dimensional space spanned by $\{L_n, Y_p, M_k|n\in \mathbb{Z}, p\in\rho+\mathbb{Z}, k\in 2\rho +\mathbb{Z}\}$ for any $\rho\in{\bf k}$. Then $SV_\rho$ is a Lie algebra with nonzero brackets given by

\begin{eqnarray*} &[L_n, L_{n'}]=(n-n')L_{n+n'},\qquad [Y_p, Y_{p'}]=(p-p')M_{p+p'},\\
& [L_n, Y_p]=\left(\lambda_1(n+1)-p+\rho-1\right)Y_{n+p}+\kappa_1Y_{n+p+1},\\
& [L_n, M_k]=\left((2\lambda_1-1)(n+1)-k+2\rho-1\right)M_{n+k}+2\kappa_1M_{n+k+1},\end{eqnarray*} where $\lambda_1,\kappa_1\in{\bf k}$.
The Schr\"odinger-Virasoro Lie algebra in \cite{U} is the algebra $SV_{\rho}$ in the case when $\lambda_1=\rho=\frac12$ and $\kappa_1=0$. The subalgebra of $SV_{\frac12}$ for $\kappa_1=0$ and $\lambda_1=\frac12$ generated by $\{Y_p,M_k|p\in\frac12+\mathbb{Z},k\in\mathbb{Z}\}$ is called a {\it Schr\"odinger algebra}.  Thus this Schr\"odinger algebra is also an annihilation algebra  of  a Lie $H$-pseudoalgebra. In Section 5, we completely determine all  Schr\"odinger-Virasoro Lie $H$-pseudoalgebras.  This means that we have determined all extensions of $Vect(S^1)$ by different weight bosonic currents.

Throughout this paper, ${\bf k}$  is a field of characteristic zero.  The unadorned $\otimes$ means the tensor product over ${\bf k}$. For the other undefined terms, we refer to  \cite{R}, \cite{RU} and \cite{Sw}.

\section{Classification of Leibniz $H$-pseudoalgebras of rank two}

From Section 1, we know that an $m$-type Schr\"odinger-Virasoro Lie conformal algebra is an extension of the Virasoro Lie conformal algebra by  a Lie $H$-pseudoalgebra of rank two. How many Lie  $H$-pseudoalgebras of rank two are there? Further, how many Leibniz $H$-pseudoalgebras of rank two are there? This question is answered in this section.

First of all, let us recall some conception and fix some notations.
Let $H={\bf k}[s]$ be the polynomial algebra with a variable $s$. Then $H$ is a Hopf algebra with coproduct
$\Delta(s^{(n)})=\sum\limits_{i=0}^ns^{(i)}\otimes s^{(n-i)},$ where $s^{(n)}=\frac {s^n}{n!}$ for any nonnegative integer $n$.
A left $H$-module $A$ is called an {\it $H$-pseudoalgebra } if there is a mapping $\rho: A\otimes A\to H^{\otimes 2}\otimes_H A$ such that $\rho(h_1a_1\otimes h_2a_2)=(h_1\otimes h_2\otimes_H1)\rho(a_1\otimes a_2)$ for any $h_1,h_2\in H$ and $a_1,a_2\in A$. $\rho(a\otimes b)$ is usually denoted by $a*b$ and is called {\it pseudoproduct}. There are various varieties of $H$-pseudoalgebras, for example, associative $H$-pseudoalgebras (\cite{R}), left symmetric $H$-pseudoalgebras ([W1]) and Leibniz $H$-pseudoalgebras (\cite{Wu2}) etc.. A {\it Leibniz $H$-pseudoalgebra} $A$ is an $H$-pseudoalgebra such that its pseudoproduct $\rho$ satisfies the Jacobi Identity$$\rho(\rho(a_1\otimes a_2)\otimes a_3)=\rho(a_1\otimes\rho(a_2\otimes a_3))-((12)\otimes_H1)\rho(a_2\otimes \rho(a_1\otimes a_3))$$ for any $a_1, a_2, a_3\in A$, where $((12)\otimes_H1)\sum_if_i\otimes g_i\otimes h_i\otimes_Hc_i=\sum_ig_i\otimes f_i\otimes h_i\otimes_Hc_i$ for any $f_i,g_i,h_i\in H$ and $c_i\in A$. In a Leibniz $H$-pseudoalgebra $A$, its pseudoproduct $\rho$ is also called a {\it pseudobracket}, and $\rho(a_1\otimes a_2)$ is usually denoted by $[a_1, a_2]$.
An asymmetric Leibniz $H$-pseudoalgebra is called a {\it Lie $H$-pseudoalgebra}, that is, a Leibniz $H$-pseudoalgebra $A$ is a Lie $H$-pseudoalgebra if in addition $$\rho(a_1\otimes a_2)=-((12)\otimes_H1)\rho(a_2\otimes a_1)$$ for any $a_1,a_2\in A$.

Similar to the case when $A$ is a Lie $H$-pseudoalgebra in \cite{BDK}, one can introduce the following notations. For any Leibniz $H$-pseudoalgebra $A$, let $A^{(1)}=[A,A]$ and $A^{(n)}=[A^{(n-1)}, A^{(n-1)}]$ for any $n\geq 2$. Then $A, A^{(1)}, \cdots, A^{(n)},\cdots $ is called the {\it derived series} of $A$ and $A$ is said to be {\it solvable} if there is an integer $n$ such that $A^{(n)}=0$. If $A^{(1)}=0$, then $A$ is said to be {\it abelian}.  $A$ is said to be {\it finite} if $A$ is a finitely generated left $H$-module.

Since $H={\bf k}[s]$ is a principal ideal domain, any finite Leibniz $H$-pseudoalgebra $A$ has a decomposition $A=A_0\oplus A_1$ as $H$-modules, where $A_0$ is a torsion $H$-module and $A_1$ is a free $H$-module. Let $0\neq a\in H$ such that $aA_0=0$ and $\{e_1, \cdots, e_n\}$ be a basis of $A_1$.
Suppose $x\in A_0, y \in A$ and $[x, y]=\sum\limits_if_i\otimes_Hx_i+\sum\limits_{j=1}^n g_j\otimes_He_j$ for some $x_i\in A_0$. Then
$0=[ax,y]=\sum\limits_i(a\otimes 1)f_i\otimes_Hx_i+\sum\limits_{j=1}^n (a\otimes 1)g_j\otimes_He_j.$ Hence $g_j=0$ for all $j$ and $[x,y]\in
H^{\otimes 2}\otimes _HA_0$. Similarly, we can prove that $[y,x]\in H^{\otimes2}\otimes_HA_0$ for any $x\in A_0$ and $y\in A$. This means that $A_0$ is an ideal of $A$. Since $A_0$ is torsion, $A^{(1)}_0=0$. Thus any finite Leibniz $H$-pseudoalgebra is an extension of Leibniz $H$-pseudoalgebra of a free module by an abelian Lie $H$-pseudoalgebra, which is a torsion module.  Next we only need focus on a finite Leibniz $H$-pseudoalgebra $A$, where $A$ is a free $H$-module.

Suppose the rank of a solvable Leibniz $H$-pseudoalgebra $A$ is $n$. Then $A$ is said to be a {\it solvable Leibniz $H$-pseudoalgebra with maximal derived series} if  $A^{(n-1)}\neq 0$.

\begin{lemma} \label{lem21} (1) Let $A$ be a solvable Leibniz $H$-pseudoalgebra. Suppose that $A$ is a free $H$-module of rank $n$ and $A^{(m)}=0$. Then $m\leq n$.

(2) Let $A$ be a solvable Leibniz $H$-pseudoalgebra with maximal derived series. Suppose the rank of $A$ is $n$. Then there is a basis
$\{e_1,e_2,\cdots,e_n\}$ such that $A^{(i)}\subseteq He_{i+1}+\cdots +He_n$ and the rank of $A^{(i)}$ is equal to $n-i$. \end{lemma}

\begin{proof} (1) We prove (1) by the induction on $n$ the rank of $A$. Suppose $n=1$. Then $A$ has a basis $\{e_1\}$ such that $\{a_{1}e_{1}\}$ is a basis of $A^{(m-1)}$. Since $A^{(m)}=0$, we have $0=[a_1e_1, a_1e_1]=(a_1\otimes a_1\otimes_H1)[e_1, e_1]$. Hence $[e_1, e_1]=0$ and $A^{(1)}=0$. Thus $m=1$. Next we assume that $n\geq 2$. Let $\{e_1, e_2,\cdots,e_n\}$ be a basis of $A$ such that $\{a_{r+1}e_{r+1}, \cdots, a_ne_n\}$ is a basis of $A^{(m-1)}$. Suppose $[e_i, e_j]=\sum\limits_{k=1}^nf_k\otimes _He_k$ for any $j\geq r+1$. Then $[e_i,a_je_j]=\sum\limits_{k=1}^n(1\otimes a_j)f_k\otimes_H e_k\in H^{\otimes 2}\otimes_H A^{(m-1)}$. So $f_k=0$ for $1\leq k\leq r$ and $[e_i,e_j]\in H^{\otimes2}\otimes_HJ$, where  $J=He_{r+1}+\cdots+He_n$. In the same way, $[e_j,e_i]\in H^{\otimes2}\otimes_HJ$ for any $j\geq r+1$. Thus, $J$ is an ideal of $A$. Similar to the case of $n=1$, we can prove $J^{(1)}=0$ as $A^{(m-1)}$ is abelian. Let $\bar{A}=A/J$. Then $\bar{A}$ is a solvable Leibniz $H$-pseudoalgebra and it is a free $H$-module of rank $r$. Hence $\bar{A}^{(r)}=0$ and $A^{(r)}\subseteq J$. Thus $A^{(r+1)}=0$ and $r+1\leq n$.

(2) We prove (2) by the induction again on $n$ the rank of $A$. Assume  $n\geq 2$. Then there is a basis $\{\varepsilon_1,\cdots,\varepsilon_n\}$ of $A$ such that $\{a_{r+1}\varepsilon_{r+1}, \cdots, a_n\varepsilon_n\}$ is a basis of $A^{(n-1)}$, where $r\leq n-1$. Let $J=H\varepsilon_{r+1}+\cdots+H\varepsilon_n$. Then $J$ is an ideal of $A$ and $J^{(1)}=0$. Assume that $m$ is the least integer such that $(A/J)^{(m)}=0$. Then $m\leq r$ and $m+1\geq n$ as $A^{(m+1)}=0$ and $A^{(n-1)}\neq 0$. Thus $n-1\leq m\leq r\leq n-1$ and $r=m=n-1$.
Hence $A/J$ is a solvable Leibniz $H$-pseudoalgebra with maximal derived series. By the induction assumption, $A/J$ has a basis $\{e'_1, \cdots, e'_{n-1}\}$ such that $(A/J)^{(i)}\subseteq He'_{i+1}+\cdots+He'_{n-1}$ and the rank of $(A/J)^{(i)}$ is equal to $n-i-1$. Let
$e_n=\varepsilon_n$ and $e_i\in A$ such that $e_i'=e_i+J$ for $1\leq i\leq n-1$. Then $\{e_1,\cdots,e_n\}$ is a basis of $A$ and $A^{(i)}\subseteq He_{i+1}+\cdots+He_n$. Since $n-i$ is the smallest  integer such that $(A^{(i)})^{(n-i)}=0$, \  $n-i$ is not greater than $t_i$ the rank of $A^{(i)}$. Note that the rank of $He_{i+1}+\cdots+He_n$ is $n-i$. Then  $t_i=n-i$.
\end{proof}

Let $A$ be a solvable  Leibniz $H$-pseudoalgebra $A$ of rank two with maximal derived series.  Then $A$ has a basis $\{e_1, e_2\}$ such that $\{ae_2\}$ is a basis of $A^{(1)}$ for some nonzero $a\in H$. In addition, there exist elements $\alpha', \eta_1, \eta_2\in H\otimes H$ such that   $[e_1,e_1]=\alpha'\otimes_He_2$, $[e_1,e_2]=\eta_1 \otimes_H e_2$, $[e_2,e_1]=\eta_2 \otimes_H e_2$,
 and $[e_2, e_2]=0$. Since $A$ is not abelian, $\alpha',\eta_1,\eta_2$ are not all zero. Furthermore, we have the following Lemma.

\begin{lemma} \label{lem22} Let $A$ be  a solvable Leibniz $H$-pseudoalgebra $A$ with maximal derived series. Suppose the rank of $A$ is two. Then $A$ is isomorphic to one of the  following types

(i) $A$ has a basis $\{e_1,e_2\}$ such that $[e_1,e_1]= \alpha'\otimes_He_2$, $[e_1,e_2]=[e_2,e_1]=0$, where  $\alpha'$ is an arbitrary nonzero element in $H\otimes H$.

(ii)  $A$ has a basis $\{e_1,e_2\}$ such that $[e_1,e_1]= 0$, $[e_1,e_2]=-((12)\otimes_H1)[e_2,e_1]=(A(s)\otimes1)\otimes_He_2$,  where $0\neq A(s)=\sum\limits_{i=0}^nb_is^{(i)}\in H$  for some $b_i\in{\bf k}$.

(iii) $A$ has a basis $\{e_1,e_2\}$ such that $[e_1,e_1]=[e_2,e_1]= 0$, $[e_1,e_2]=(A(s)\otimes1)\otimes_He_2$,  where $A(s)=\sum\limits_{i=0}^nb_is^{(i)}\neq0$ for some $b_i\in{\bf k}$.
\end{lemma}

\begin{proof}   Let $\{e_1,e_2\}$ be a basis of a solvable Leibniz $H$-pseudoalgebra $A$ with maximal derived series such that $[e_1,e_1]=\alpha'\otimes_He_2$, $[e_1, e_2]=\eta_1\otimes_He_2$, $[e_2,e_1]=\eta_2\otimes_He_2$ and $[e_2, e_2]=0$.
Since $0=[[e_1, e_1], e_2]=[e_1, [e_1, e_2]]-(12)[e_1, [e_1, e_2]]$,  we have $(1\otimes \eta_1\Delta)\eta_1=(12)(1\otimes \eta_1\Delta)\eta_1.$
Thus either $\eta_1=0$, or $\eta_1=\sum\limits_{i=0}^n\sum\limits_{j=0}^{p_i}a_{ij}s^{(i)}\otimes s^{(j)}\neq 0$ for some $a_{ij}\in {\bf k}$.  If $\eta_1\neq 0$, then \begin{eqnarray*}\sum\limits_{i=0,j=0}^{n,p_i}\sum\limits_{i'=0,j'=0}^{n,p_i}\sum\limits_{t=1}^ja_{ij}a_{i'j'}(s^{(t)}s^{(i')}\otimes s^{(i)}-s^{(i)}\otimes s^{(i')}s^{(t)})\otimes s^{(j')} s^{(j-t)}=0,\end{eqnarray*} which implies that $\sum\limits_{i=0,j=0}^{n,p_i}\sum\limits_{t=1}^ja_{ij}s^{(n)}s^{(t)}\otimes s^{(i)}\otimes s^{(p_n)}s^{(j-t)}=0.$
 Hence $a_{ij}=0$ for $j\geq 1$ and any $i$. So $\eta_1=\sum\limits_{i=1}^nb_is^{(i)}\otimes1$, where $b_i=a_{i0}$ for $0\leq i\leq n$.

Since $[[e_2,e_1],e_1]=-((12)\otimes_H1)[e_1,[e_2,e_1]]$, $(\eta_2\Delta\otimes 1)\eta_2=-(12)(1\otimes \eta_2\Delta)\eta_1$. Thus $\eta_2=0$ if $\eta_1=0$. Then $\alpha'$ is any nonzero element in $H^{\otimes 2}$. Next, we assume that $\eta_1=\sum\limits_{i=0}^nb_is^{(i)}\otimes 1$
and $\eta_2=\sum\limits_{i=0}^{n'}\sum\limits_{j=0}^{p_i'}b_{ij}s^{(i)}\otimes s^{(j)}\neq 0$ for some $b_{ij}\in {\bf k}$. Since $(\eta_2\Delta\otimes 1)\eta_2=-(12)(1\otimes\eta_2\Delta)\eta_1$,
$$\sum\limits_{i,k=0}^{n'}\sum\limits_{j,l=0}^{p'_i,p'_k}\sum\limits_{t=0}^ib_{ij}b_{kl}s^{(k)}s^{(t)}\otimes s^{(l)}s^{(i-t)}\otimes s^{(j)}=
-\sum\limits_{i=0}^{n'}\sum\limits_{j=0}^{p'_i}\sum\limits_{k=0}^nb_kb_{ij}s^{(i)}\otimes s^{(k)}\otimes s^{(j)}.$$
Then $n'=0$, $p_0'=n$ and $-b_lb_{0j}=b_{0l}b_{0j}$. Thus $b_{0l}=-b_l$ and $\eta_2=-(12)\eta_1$.

Let us assume that $\eta_2\neq 0$, $b_n\neq 0$ and $\alpha'=\sum\limits_{i=0}^m\sum\limits_{j=0}^{q_i}y_{ij}s^{(i)}\otimes s^{(j)}\neq 0$.
Observe that $$[[e_1, e_1], e_1]=[e_1, [e_1, e_1]]-(12)[e_1, [e_1, e_1]].$$This means \begin{eqnarray}(\alpha'\Delta\otimes 1)\eta_2=(1\otimes\alpha'\Delta)\eta_1-(12)(1\otimes\alpha'\Delta)\eta_1,\label{abc1}\end{eqnarray} or equivalently,
\begin{eqnarray} \qquad \qquad \sum\limits_{k=0}^n\sum\limits_{i=0}^m\sum\limits_{j=0}^{q_i}
b_ky_{ij}\left((s^{(k)}\otimes s^{(i)}-s^{(i)}\otimes s^{(k)})\otimes s^{(j)}+
s^{(i)}\otimes s^{(j)}\otimes s^{(k)}\right)=0.\label{eq021}\end{eqnarray}
From this, we get $\sum\limits_{k=0}^n\sum\limits_{i=0}^mb_ky_{in}(s^{(k)}\otimes s^{(i)}-s^{(i)}\otimes s^{(k)})=-b_n\sum\limits_{i=0}^m\sum\limits_{j=0}^{q_i}
y_{ij}s^{(i)}\otimes s^{(j)}.$ Let $A(s)=\sum\limits_{i=0}^nb_is^{(i)}$ and $B(x)=\sum\limits_{i=0}^m\frac{y_{in}}{b_n}s^{(i)}$. Then $\alpha'=B(s)\otimes A(s)-A(s)\otimes B(s)$.
It is easy to check that (\ref{abc1}) holds for this $\alpha'$. Now let $e_1'=e_1+B(s)e_2$. Then $[e_1',e_1']=[e_1,e_1]+((1\otimes B(s))\otimes _H1)[e_1,e_2]+((B(s)\otimes 1)\otimes_H1)[e_2,e_1]=0$, $[e_1',e_2]=-((12)\otimes_H1)[e_2,e_1']=\eta_1\otimes_He_2$.

Finally, assume that $\eta_1=\sum\limits_{i=0}^nb_is^{(i)}\otimes 1\neq 0$ and $\eta_2=0$. From $[[e_1,e_1],e_1]=[e_1,[e_1,e_1]]-((12)\otimes_H1)[e_1,[e_1,e_1]]$ we get
$(1\otimes \alpha'\Delta)\eta_1=(12)(1\otimes\alpha'\Delta)\eta_1$. Thus $\sum\limits_{k=0}^n\sum\limits_{i=0}^m\sum\limits_{j=0}^{q_i}b_ky_{ij}s^{(k)}\otimes
s^{(i)}\otimes s^{(j)}=\sum\limits_{k=0}^n\sum\limits_{i=0}^m\sum\limits_{j=0}^{q_i}b_ky_{ij}s^{(i)}\otimes s^{(k)}\otimes s^{(j)}$. Hence $b_ky_{ij}=b_iy_{kj}$ and $y_{ij}=\frac1{b_n}b_iy_{nk}$. Thus $\alpha'=\frac1{b_n}\sum\limits_{i=0}^n\sum\limits_{j=0}^{q_n}b_iy_{nj}s^{(i)}\otimes s^{(j)}.$ Let $C(s)=\sum\limits_{j=0}^{q_n}\frac{y_{nj}}{b_n}s^{(j)}$ and $m=q_n$. Then $\alpha'=(\sum\limits_{i=0}^nb_is^{(i)})\otimes(\sum\limits_{j=0}^mc_js^{(j)})=A(s)\otimes C(s)$. Let $e_1'=e_1-C(s)e_2$. Then $[e_1',e_1']=[e_1,e_1]-((C(s)\otimes1)\otimes_H1)[e_2,e_1]-((1\otimes C(s))\otimes_H1)[e_1,e_2]=0$, $[e_1',e_2]=(\eta_1\otimes1)\otimes_He_2$ and $[e_2,e_1']=0$.
\end{proof}

Suppose that $A$ is a solvable Leibniz $H$-pseudoalgebra with maximal derived series and the rank of $A$ is $n$. Then $A$ has a basis $\{e_1, e_2, \cdots, e_n\}$ such that $J_i=He_i+He_{i+1}+\cdots+He_n$ is an ideal of $A$ for each $i$ by the proof of Lemma \ref{lem21}. Similarly, one can prove that both $A/J_i$ and $J_i$ are solvable Leibniz $H$-pseudoalgebras with maximal derived series. Thus $$\left\{\begin{array}{lll}[e_i,e_i]&=&\alpha_{ii}^{i+1}\otimes _He_{i+1}+\cdots+\alpha_{ii}^n\otimes_He_n\\

[e_i,e_j]&=&\alpha_{ij}^{j}\otimes _He_{j}+\cdots+\alpha_{ij}^n\otimes_He_n,\  for\  i<j\\

[e_j,e_i]&=&\beta_{ji}^{j}\otimes _He_{j}+\cdots+\alpha_{ij}^n\otimes_He_n,\  for\  i<j
\end{array}\right.$$
Hence $\alpha_{ii+1}^{i+1}=b_{ii+1}^{i+1}\otimes 1$ for some $b_{ii+1}^{i+1}\in H$ by Lemma \ref{lem22}.

Let $L_0=He_0$ is a free left $H$-module with a basis $\{e_0\}$. Then $L_0$ is a Leibniz $H$-pseudoalgebra with $[e_0,e_0]=\alpha\otimes_He_0$ if and only if $\alpha=c(s\otimes 1-1\otimes s)$ for some $c\in{\bf k}$. In addition, $L_0$ is a Leibniz $H$-pseudoalgebra if and only if it is a Lie $H$-pseudoalgebra by Theorem 3.1 of \cite{Wu2}. If $c\neq 0$, then $[e_0',e_0']=(s\otimes 1-1\otimes s)\otimes_He_0'$ for $e_{0}'=c^{-1}e_0$. Thus, we can assume $c=1$. In the sequel, we always assume that $c=1$, that is, $\alpha=s\otimes 1-1\otimes s$. In this case, $L_0$ is called {\it Virasoro Lie conformal algebra}.

The following lemma was proved in \cite{CK}.
\begin{lemma} \label{lem23} Suppose a free $H$-module  $He$ with basis $\{e\}$ is a nontrivial representation of the Virasoro Lie conformal algebra $He_0$. Then
$$[e_0, e]=(\lambda s\otimes 1-1\otimes s+\kappa\otimes 1)\otimes_He$$
for some $\lambda,\kappa\in{\bf k}$.
This representation is irreducible if and only if $\lambda\neq 1$, and all finitely generated irreducible representations of the conformal Lie algebra $He_0$ are of this kind.\end{lemma}

For any two solvable ideals $I_1,I_2$ of a finite Leibniz $H$-pseudoalgebras $L$, $I_1+I_2$ is also a solvable ideal of $L$. Therefore $L$ has a unique maximal solvable ideal $I$. We call this ideal $I$ the solvable radical of $L$. It is easy to prove that the solvable radical of $L/I$ is zero and $L/I$ is a Lie $H$-pseudoalgebra. Suppose $L$ is a Leibniz $H$-pseudoalgebra of rank two and $I$ is its  solvable ideal.   If $I=0$, then $L$ is a direct sum of two Virasoro Lie conformal algebras by the Theorem 13.3 of \cite{BDK}, or \cite{AK}.
If the rank of $I$ is two, then $L=I$ is solvable.  It is either abelian, or a solvable Leibniz $H$-pseudoalgebra with maximal derived series, which is described by Lemma \ref{lem22}. If the rank of $I$ is equal to $1$, then $L$ has a basis $\{e_0,e_1\}$ such that $I=He_1$ and $L/I$ is isomorphic to the Virasoro Lie conformal algebras. Thus we can assume that \begin{eqnarray}\label{m1}\left\{\begin{array}{ll} [e_0,e_0]=\alpha\otimes_He_0+\alpha'\otimes_He_1,& [e_0,e_1]=\eta_1\otimes_He_1,\\

[e_1,e_0]=\eta_2\otimes_He_1, &[e_1,e_1]=0\end{array}\right.
\end{eqnarray} where $\alpha=s\otimes 1-1\otimes s$, $\alpha',\eta_1,\eta_2\in H\otimes H$ satisfying
\begin{eqnarray}\label{aa1}\begin{array}{lll}(\alpha\Delta\otimes1)\alpha'+(\alpha'\Delta\otimes1)\eta_2&=&(1\otimes \alpha\Delta)\alpha'-(12)(1\otimes\alpha\Delta)\alpha'\\
&&+(1\otimes\alpha'\Delta)\eta_1-(12)(1\otimes\alpha'\Delta)\eta_1,\end{array}\end{eqnarray}
\begin{eqnarray}\label{aa2}(\alpha\Delta\otimes1)\eta_1=(1\otimes\eta_1\Delta)\eta_1-(12)(1\otimes\eta_1\Delta)\eta_1,\end{eqnarray}
\begin{eqnarray}\label{aa3}(\eta_2\Delta\otimes1)\eta_2=(1\otimes \alpha\Delta)\eta_2-(12)(1\otimes\eta_2\Delta)\eta_1,\end{eqnarray}
\begin{eqnarray}\label{aa4}(\eta_1\Delta\otimes1)\eta_2=(1\otimes\eta_2\Delta)\eta_1-(12)(1\otimes\alpha\Delta)\eta_2.\end{eqnarray}
From (\ref{aa2}) and Lemma \ref{lem23}, we get that either $\eta_1=0$, or $\eta_1=\lambda s\otimes 1-1\otimes s+\kappa\otimes 1$ for some $\lambda,\kappa\in {\bf k}$.
From (\ref{aa3}) and (\ref{aa4}), we get $((\eta_1+(12)\eta_2)\Delta\otimes 1)\eta_2=0$. Thus either $\eta_2=0$, or $\eta_2=-(12)\eta_1$. If $\eta_2\neq 0$, then $\eta_1=\lambda s\otimes 1-1\otimes s+\kappa\otimes 1$ and $\eta_2=-\lambda\otimes s+s\otimes 1-\kappa\otimes 1$.

\begin{lemma}\label{Aaa1} Let $\alpha=s\otimes 1-1\otimes s, \eta_1=-(12)\eta_2=\lambda s\otimes 1-1\otimes s+\kappa\otimes 1$, $\alpha'=\sum\limits_{j=0}^m\sum\limits_{i=0}^{p_j}x_{ij}s^{(i)}\otimes s^{(j)}\neq 0$ for some $x_{ij}\in{\bf k}$.
Suppose $\alpha'$ is a solution of   (\ref{aa1}). Then $\alpha'=\beta+(1\otimes A(s))\eta_1+(A(s)\otimes 1)\eta_2-\alpha\Delta(A(s))$ for some $A(s)\in H$, where
$$\beta=\left\{\begin{array}{ll}x_{10}\alpha&(\kappa,\lambda)=(0,0)\\
x_{12}(s\otimes s^{(2)}-s^{(2)}\otimes s)+
x_{20}(s^{(2)}\otimes 1-1\otimes s^{(2)})&(\kappa,\lambda)=(0,-1)\\
x_{13}(s\otimes s^{(3)}-s^{(3)}\otimes s)+x_{30}(s^{(3)}\otimes 1-1\otimes s^{(3)})&(\kappa,\lambda)=(0,-2)\\
x_{34}(s^{(3)}\otimes s^{(4)}-s^{(4)}\otimes s^{(3)})&(\kappa, \lambda)=(0,-5)\\
x_{36}((s^{(3)}\otimes s^{(6)}-s^{(6)}\otimes s^{(3)}-3(s^{(4)}\otimes s^{(5)}-s^{(5)}\otimes s^{(4)}))&(\kappa,\lambda)=(0,-7)\\
0&otherwise,\end{array}\right.$$
for some $x_{ij}\in{\bf k}$\end{lemma}

\begin{proof}Let $e_0'=e_0+A(s)e_1$ for any $A(s)\in H$. Then $[e_0',e_0']=\alpha\otimes_He_0'+\alpha''\otimes_He_1$, $[e_0',e_1]=\eta_1\otimes_He_1$ and $[e_1,e_0']=\eta_2\otimes_He_1$, where $\alpha''=\alpha'+(1\otimes A(s))\eta_1+(A(s)\otimes 1)\eta_2-\alpha\Delta(A(s)).$ Thus $\beta=(1\otimes A(s))\eta_1+(A(s)\otimes 1)\eta_2-\alpha\Delta(A(s))
$ is a solution of (\ref{aa1}) for any $A(s)\in H$. In particular, $\beta=\lambda A(s)\alpha$ for any $A(s)\in{\bf k}$ is a solution of (\ref{aa1}). Thus $x_{10}\alpha$ is a solution of (\ref{aa1}) for any $x_{10}\in {\bf k}$ provided that $\lambda\neq 0$. It is easy to check that $x_{10}\alpha$ is also a solution of (\ref{aa1}) if $\lambda=0$.
Applying the functor $\varepsilon\otimes1\otimes 1$  to (\ref{aa1}), we get \begin{eqnarray}\label{La8}\kappa\alpha'= (A^*(s)s\otimes 1-1\otimes A^*(s)s)+\lambda(s\otimes A^*(s)-A^*(s)\otimes s) -\alpha\Delta(A^*(s)),\end{eqnarray} where $A^*(s)=\sum\limits_{j=0}^mx_{0j}s^{(j)}$. If $\kappa\neq 0$,
then  $\alpha'=A'(s)s\otimes1-1\otimes A'(s)s+\lambda(s\otimes A'(s)-A'(s)\otimes s)+{\kappa}(1\otimes A'(s)-A'(s)\otimes 1)-\alpha\Delta(A'(s))$ is a solution of (\ref{aa1}), where $A'(s)=\frac{1}{\kappa}A^*(s)$.
Next we assume that $\kappa=0$. Then $(\lambda+i-1)x_{0i}=0$ for $i\geq 2$ and $\sum\limits_{i=3}^{m}\sum\limits_{t=1}^{i-2}(t+1)x_{0i}(s^{(t+1)}\otimes s^{(i-t)}-s^{(i-t)}\otimes s^{(t+1)})=0$ by (\ref{La8}). Thus $x_{0i}=0$ for all $i\geq 4$. If $x_{03}\neq 0$, then $\lambda=-2$. If  $x_{02}\neq 0$, then $\lambda=-1$. If $x_{01}\neq 0$, then $\lambda=0$.
 Applying the functor $1\otimes \varepsilon\otimes 1$ to (\ref{aa1}), we get $\lambda(\sum\limits_{j=0}^mx_{0j}s\otimes s^{(j)}+\sum\limits_{i=0}^{p_0}x_{i0}s^{(i)}\otimes s)=\sum\limits_{i=0}^{p_0}x_{i0}(i+1)s^{(i+1)}\otimes 1+\sum\limits_{j=0}^mx_{0j}(j+1)s^{(j+1)}\otimes 1$. This implies that $x_{i0}=-x_{0i}$ for all $i\geq 0$, $m=p_0$ and $ x_{00}=0$. By the same way, we can obtain
 $\sum\limits_{i=0}^mx_{i0}(i+1)(s^{(i+1)}\otimes 1-1\otimes s^{(i+1)})+\lambda\sum\limits_{i=0}^mx_{i0}(s\otimes s^{(i)}-s^{(i)}\otimes s)
 -\alpha\sum\limits_{i=0}^mx_{i0}\Delta(s^{(i)})=
 (s\otimes 1)(\alpha'+(12)\alpha') $  by using the functor $1\otimes 1\otimes\varepsilon$. Thus $(s\otimes 1)(\alpha'+(12)\alpha')=0$ and $\alpha'=-(12)\alpha'$.
  Since  $\alpha'=\sum\limits_{j=0}^m\sum\limits_{i=0}^{p_j}x_{ij}s^{(i)}\otimes s^{(j)}$,
\begin{eqnarray}\label{La9}&&\sum\limits_{j=1}^m\sum\limits_{i=1}^{p_j}\sum\limits_{t=0}^{i-1}(t+1)x_{ij}(s^{(t+1)}\otimes s^{(i-t)}-s^{(i-t)}\otimes s^{(t+1)})\otimes s^{(j)}\nonumber\\
&&=\sum\limits_{j=1}^m\sum\limits_{i=1}^{p_j}\sum\limits_{l=0}^{j-1}x_{ij}(2l-j+1)(s^{(i)}\otimes s^{(l+1)}-s^{(l+1)}\otimes s^{(i)})\otimes s^{(j-l)}\\
&&+\sum\limits_{j=1}^m\sum\limits_{i=1}^{p_j}x_{ij}\lambda(s\otimes s^{(i)}-s^{(i)}\otimes s)\otimes s^{(j)}+\lambda\sum\limits_{j=1}^m\sum\limits_{i=1}^{p_j}x_{ij}s^{(i)}\otimes s^{(j)}\otimes s \nonumber \end{eqnarray}by (\ref{aa1}) and (\ref{La8}).
Comparing the terms $*\otimes *\otimes s$ in (\ref{La9}), we get
\begin{eqnarray}\label{La10}&&\sum\limits_{i=3}^{p_1}\sum\limits_{t=1}^{i-3}(t+1)x_{i1}(s^{(t+1)}\otimes s^{(i-t)}-s^{(i-t)}\otimes s^{(t+1)})
=\sum\limits_{j=2}^m\sum\limits_{i=2}^{p_j}(\lambda+i+j-2)x_{ij}s^{(i)}\otimes s^{(j)}.\nonumber\end{eqnarray}
Thus $(\lambda+u+v-2)x_{uv}=(u-v)x_{u+v-1\ 1}$ for all $u,v\geq 2$. Moreover, (\ref{La9}) implies
\begin{eqnarray}\label{La10}&&\sum\limits_{j=2}^m\sum\limits_{i=3}^{p_j}\sum\limits_{t=1}^{i-2}(t+1)x_{ij}(s^{(t+1)}\otimes s^{(i-t)}-s^{(i-t)}\otimes s^{(t+1)})\otimes s^{(j)}\nonumber\\
&&=\sum\limits_{j=3}^m\sum\limits_{i=2}^{p_j}\sum\limits_{l=1}^{j-2}x_{ij}(2l-j+1)(s^{(i)}\otimes s^{(l+1)}-s^{(l+1)}\otimes s^{(i)})\otimes s^{(j-l)}\end{eqnarray}
 If $i+j\neq 2-\lambda$, then $x_{ij}=\frac{i-j}{i+j+\lambda-2}x_{i+j-1\ 1}$ for all $i,j\geq 2$ and $(2-\lambda-j)x_{1-\lambda\ 1}=ix_{1-\lambda\ 1}=0$ for $\lambda\leq -3$. Hence
$\alpha'=\sum\limits_{i=1}^3x_{i0}(s^{(i)}\otimes 1-1\otimes s^{(i)})+\sum\limits_{i=2}^{p_1}x_{i1}(s^{(i)}\otimes s-s\otimes s^{(i)})+\sum\limits_{t=2}^{-\lambda}x_{t\ 2-\lambda-t}
s^{(t)}\otimes s^{(2-\lambda-t)}+\sum\limits_{j=2}^m\sum\limits_{i=2}^{p_j}\frac{(i-j)}{i+j+\lambda-2}x_{i+j-1\ 1}s^{(i)}\otimes s^{(j)},$ where $x_{t\ 2-\lambda-t}=-x_{2-\lambda-t, t}$ and $x_{i0}$ satisfies $(\lambda+i-1)x_{i0}=0$ for $i\geq 1$.
Let $A(s)=\sum\limits_{i=4,i\neq 1-\lambda}^{p_1}\frac{x_{i1}}{i+\lambda-1}s^{(i)}$. Then $\alpha'=\sum\limits_{i=1}^3x_{i0}(s^{(i)}\otimes 1-1\otimes s^{(i)})+\sum\limits_{i=2}^{3}x_{i1}(s^{(i)}\otimes s-s\otimes s^{(i)})+\sum\limits_{t=2}^{-\lambda}x_{t\ 2-\lambda-t}
s^{(t)}\otimes s^{(2-\lambda-t)}-((A(s)s\otimes 1-1\otimes A(s)s)+\lambda(s\otimes A(s)-A(s)\otimes s)-\alpha\Delta(A(s)))$.

If $\lambda =0$, then $\alpha'=x_{10}\alpha+(A_1(s)s\otimes 1-1\otimes A_1(s)s)-\alpha\Delta(A_1(s)))$, where $A_1(s)=x_{12}s^{(2)}+\frac{x_{13}}2s^{(3)}-A(s)$.

If $\lambda=-1$, then $\alpha'=x_{12}(s\otimes s^{(2)}-s^{(2)}\otimes s)+x_{20}(s^{(2)}\otimes 1-1\otimes s^{(2)})+(A_1(s)s\otimes 1-1\otimes A_1(s)s)-(s\otimes A_1(s)-A_1(s)\otimes s)-\alpha\Delta(A_1(s))$, where $A_1(s)=-x_{10}+x_{13}s^{(3)}-A(s)$.

If $\lambda=-2$, then $\alpha'=x_{13}(s\otimes s^{(3)}-s^{(3)}\otimes s)+x_{30}(s^{(3)}\otimes 1-1\otimes s^{(3)})+(A_1(s)s\otimes 1-1\otimes A_1(s)s)-2(s\otimes A_1(s)-A_1(s)\otimes s)-\alpha\Delta(A_1(s)))$, where $A_1(s)=-\frac{x_{10}}2-x_{12}s^{(2)}-A(s)$.

If $\lambda\leq -3$ is an integral number, then $\alpha'=\beta+(A_1(s)s\otimes 1-1\otimes A_1(s)s)+\lambda(s\otimes A_1(s)-A_1(s)\otimes s)-\alpha\Delta(A_1(s)))$, where $A_1(s)=\frac{x_{10}}{\lambda}+\frac{x_{12}}{\lambda+1}s^{(2)}+\frac{x_{13}}{\lambda+2}s^{(3)}-A(s)$ and $\beta=\sum\limits_{t=2}^{-\lambda}x_{t\ 2-\lambda-t}
s^{(t)}\otimes s^{(2-\lambda-t)}$ .

Fix an integer $n\geq 2$. Comparing the terms $s^{(i)}\otimes s^{(j)}\otimes s^{(n)}$ in (\ref{La10}),
 we can obtain that $(i-j)x_{i+j-1\ n}=(i-n)x_{i+n-1\ j}+(j-n)x_{i\ j+n-1}$ for $i\geq n$ and $j\geq n$. Let $i+j+n-1=2-\lambda$. Then  $$(2i+n+\lambda-3)x_{2-\lambda-n\ n}=(i-n)x_{i+n-1\ 3-\lambda-n-i}+(3-\lambda-i-2n)x_{i\ 2-\lambda-i}.$$
Thus\begin{eqnarray}\label{ab10}(2i+\lambda-1)x_{-\lambda\ 2}=(i-2)x_{i+1\ 1-\lambda-i}-(\lambda+i+1)x_{i\ 2-\lambda-i}\\
\label{ab11} (2i+\lambda)x_{-\lambda-1\ 3}=(i-3)x_{i+2\ -\lambda-i}-(\lambda+i+3)x_{i\ 2-\lambda-i}\end{eqnarray}
 If $i=-\lambda$, then $\lambda x_{-\lambda\ 2}=0$ by (\ref{ab10}). Hence $x_{-\lambda\ 2}=0$ and $x_{i\ 2-\lambda-i}=\frac{i-2}{\lambda+i+1}x_{i+1\ 1-\lambda-i}$. Suppose $\lambda=-2k$ is even. If $i=k$, then $x_{k\ k+2}=0$. Further $x_{k-1\ k+3}=-\frac{k-3}{k}x_{k\ k+2}=0$. Continuing this way, we get all $x_{i\ 2-\lambda-i}=0$ and $\beta=0$.
Since $(i-j)x_{i+j-1\ 2}=(i-2)x_{i+1\ j}+(j-2)x_{i\ j+1}$ for $i\geq 2$ and $j\geq 2$, (\ref{ab10}) holds for $\lambda\leq -3$.  Suppose $\lambda=-2k-1$. Then $x_{3\ 2k}=\frac1{3-2k}x_{4\ 2k-1}$, $x_{4\ 2k-1}=\frac2{4-2k}x_{5\ 2k-2}$, $x_{5\ 2k-2}=\frac3{5-2k}x_{6\ 2k-3}$ by (\ref{ab10}).
Let $\lambda=-2k-1$ and $i=4$ in (\ref{ab11}). Then $x_{3\ 2k}=\frac1{2k-7}(x_{6\ 2k-3}+(2k-6)x_{4\ 2k-1})=(\frac{1}{2k-7}+\frac{6(2k-6)}{(2k-7)(4-2k)(5-2k)})x_{6\ 2k-3}=\frac6{(3-2k)(4-2k)(5-2k)}x_{6\ 2k-3}.$  If $x_{6\ 2k-3}=0$, then $\beta=0$.
If $x_{6\ 2k-3}\neq 0$, then $k\geq 3$ and $\lambda\leq -7$.
Moreover, $(\frac{1}{2k-7}+\frac{6(2k-6)}{(2k-7)(4-2k)(5-2k)})=\frac6{(3-2k)(4-2k)(5-2k)}.$ Thus $(4k^2-1)(k-3)=0$ and $k=3$.  If $k=3$, then  $\lambda=-7$ and $(2i-8)x_{7\ 2}=(i-2)x_{i+1\ 8-i}+(6-i)x_{i\ 9-i}$. If $i=7$, then $6x_{72}=-x_{72}=0$ and $x_{72}=0$. If $i=3$, then $x_{45}=-3x_{3 6}$.  Hence $\beta=x_{36}(s^{(3)}\otimes s^{(6)}-s^{(6)}\otimes s^{(3)}-3(s^{(4)}\otimes s^{(5)}-s^{(5)}\otimes s^{(4)}))$.

If $\lambda=-3$, then $(2i+n-6)x_{5-n\ n}=(i-n)x_{i+n-1\ 6-n-i}+(6-i-2n)x_{i\ 5-i}$. Thus  $(2i-4)x_{3\ 2}=(i-2)x_{i+1\ 4-i}+(2-i)x_{i\ 5-i}$. Let $i=3$. Then $3x_{32}=x_{41}=0$ and $\beta=0$.

If $\lambda=-5$, then $(2i+n-8)x_{7-n\ n}=(i-n)x_{i+n-1\ 8-n-i}+(8-i-2n)x_{i\ 7-i}$. Thus  $(2i-6)x_{5\ 2}=(i-2)x_{i+1\ 6-i}+(4-i)x_{i\ 7-i}$. Let $i=5$. Then $3x_{52}=3x_{61}=0$ and $\beta=x_{34}(s^{(3)}\otimes s^{(4)}-s^{(4)}\otimes s^{(3)})$.\end{proof}

 Suppose  $\eta_2=0$. Then
$(\alpha\Delta\otimes 1)\alpha'=(1\otimes\alpha\Delta)\alpha'-(12)(1\otimes \alpha\Delta)\alpha'+(1\otimes\alpha'\Delta)\eta_1-(12)(1\otimes \alpha'\Delta)\eta_1$ by (\ref{aa1}).

\begin{lemma}\label{Aa1} Let $\alpha=s\otimes 1-1\otimes s, \eta_1=\lambda s\otimes 1-1\otimes s+\kappa\otimes 1$ and $\alpha'=\sum\limits_{i=0}^m\sum\limits_{j=0}^{p_i}x_{ij}s^{(i)}\otimes s^{(j)}\neq 0$ for some $x_{ij}\in{\bf k}$ satisfying  \begin{eqnarray}\label{L8}\qquad (\alpha\Delta\otimes 1)\alpha'=(1\otimes\alpha\Delta)\alpha'-(12)(1\otimes \alpha\Delta)\alpha'+(1\otimes\alpha'\Delta)\eta_1-(12)(1\otimes \alpha'\Delta)\eta_1.\end{eqnarray}
Then $\alpha'=\beta+(1\otimes A(s))\eta_1-\alpha\Delta(A(s))$ for  some $A(s)\in H$, where
$$\beta=\left\{\begin{array}{ll}x_{00}\otimes 1&(\kappa,\lambda)=(0,1)\\
x_{30}s^{(3)}\otimes 1&(\kappa,\lambda)=(0,-1)\\
x_{22}(s^{(2)}\otimes s^{(2)}+\frac32s^{(3)}\otimes s)&(\kappa,\lambda)=(0,-2)\\
x_{23}(\frac15 s^{(2)}\otimes s^{(3)}+\frac45s^{(3)}\otimes s^{(2)}+\frac25s^{(4)}\otimes s)&(\kappa,\lambda)=(0,-3)\\
0&otherwise\end{array}\right.$$
for some $x_{ij}\in{\bf k}$.
\end{lemma}
\begin{proof} Similar to the proof of Lemma \ref{Aaa1}, we can prove that  $\beta=(1\otimes A(s))\eta_1-\alpha\Delta(A(s))$ is a solution of (\ref{L8}) for any $A(s)\in H$.

Applying the functor $\varepsilon\otimes1\otimes 1$  to (\ref{L8}), we get \begin{eqnarray}\label{L9}\kappa\alpha'=(1\otimes s-s\otimes 1)\sum\limits_{j=0}^{p_0}x_{0j}\Delta(s^{(j)})+\eta_1\sum\limits_{j=0}^{p_0}x_{0j}\otimes s^{(j)}.
\end{eqnarray}
If $\kappa\neq 0$,  $\alpha'=(1\otimes A(s))\eta_1-\alpha\Delta(A(s))$, where $A(s)=\sum\limits_{j=0}^{p_0}\frac{x_{0j}}{\kappa}\otimes s^{(j)}$. Thus it is a solution of   (\ref{L8}). If $\kappa=0$, then $\alpha\sum\limits_{j=0}^{p_0}x_{0j}\Delta(s^{(j)})=\eta_1\sum\limits_{j=0}^{p_0}x_{0j}\otimes s^{(j)}$. Hence $\eta_1\sum\limits_{j=0}^{p_0}x_{0j}\otimes s^{(j)}=(s\otimes 1-\lambda\otimes s)\sum\limits_{j=0}^{p_0}x_{0j} s^{(j)}\otimes 1$. This implies that
$(\lambda-1)x_{00}=0$ and $x_{0j}=0$ for all $j\geq 1$.
Next we always assume that $\kappa=0$.
Comparing the terms $*\otimes *\otimes s^{(d)}$ in (\ref{L8}), we get
$$\begin{array}{l}\sum\limits_{i=1}^m\sum\limits_{t=0}^{i-1}x_{id}(t+1)(s^{(t+1)}\otimes s^{(i-t)}-s^{(i-t)}\otimes s^{t+1)})\\
=\sum\limits_{i=1}^m\sum\limits_{j=d}^{p_i}(j-2d+1)x_{ij}(s^{(i)}\otimes s^{(j-d+1)}-s^{(j-d+1)}\otimes s^{(i)})\\
+\sum\limits_{i=2}^m\lambda x_{id}(s\otimes s^{(i)}-s^{(i)}\otimes s).\end{array}$$Thus \begin{eqnarray}\label{L10}(i-j)x_{i+j-1\ d}=(j-d)x_{i\ j+d-1}-(i-d)x_{j\ i+d-1}\end{eqnarray} for $i\geq 2$ and $j\geq 2$,  \begin{eqnarray}\label{L11}(2-i-\lambda-d)x_{id}=(i-d)x_{1\ i+d-1}\end{eqnarray} for all $i\geq 2$.
If $i+j\neq 2-\lambda$, then $x_{ij}=\frac{i-j}{2-i-j-\lambda}x_{1\ i+j-1}$ and $(2i+\lambda-2)x_{1\ 1-\lambda}=0$ for $i\geq 2$ and $j\geq 0$. Thus $x_{1\ 1-\lambda}=0$ if $m\geq 3$.
Moreover, $\alpha'=x_{00}\otimes 1+\sum\limits_{j=0}^{p_1}x_{1j}s\otimes s^{(j)}+\sum\limits_{i=2}^m\sum\limits_{j=0}^{p_i}x_{ij}s^{(i)}\otimes s^{(j)}=
x_{00}\otimes 1+\sum\limits_{j=0}^{p_1}x_{1j}s\otimes s^{(j)}+\sum\limits_{j=2, j\neq 1-\lambda}^{p_1}\sum\limits_{t=0}^{j}\frac{2t-j-1}{j+\lambda-1}x_{1\ j}s^{(j+1-t)}\otimes s^{(t)}+\sum\limits_{i=2}^{2-\lambda}x_{i\ 2-\lambda-i}s^{(i)}\otimes s^{(2-\lambda-i)}=x_{00}\otimes 1
+(1\otimes A(s))\eta_1-\alpha\Delta(A(s))+ \sum\limits_{i=2}^{2-\lambda}x_{i\ 2-\lambda-i}s^{(i)}\otimes s^{(2-\lambda-i)}$, where $A(s)=\frac{x_{10}}{\lambda-1}-\frac12x_{20}s+\sum\limits_{i=2,i\neq 1-\lambda}^{p_1}\frac1{i+\lambda-1}x_{1i}s^{(i)}$ and $(\lambda-1)x_{00}=0$, $x_{11}=-\frac{\lambda}2x_{20}$.

If $\lambda=1$, then $\alpha'=x_{00}\otimes 1+(1\otimes A(s))\eta_1-\alpha\Delta(A(s))$.

If $\lambda =0$, then $\alpha'=(1\otimes A(s))\eta_1-\alpha\Delta(A(s))$.

If $\lambda=-1$, then $\alpha'=(x_{30}-3x_{21})s^{(3)}\otimes 1+(1\otimes A_1(s))\eta_1-\alpha\Delta(A_1(s))$, where $A_1(s)=x_{21} s^{(2)}+ A(s)$.

If $\lambda=-2$, then $\alpha'=x_{22}(s^{(2)}\otimes s^{(2)}+\frac32s^{(3)}\otimes s)+(1\otimes A_1(s))\eta_1-\alpha\Delta(A_1(s))$, where $A_1(s)=-\frac14x_{40} s^{(3)}+ A(s)$.

If $\lambda=-3$, then $\alpha'=x_{23}s^{(2)}\otimes s^{(3)}+x_{32}s^{(3)}\otimes s^{(2)}+(x_{32}-2x_{23})s^{(4)}\otimes s+(x_{32}-4x_{23})s^{(5)}\otimes 1+(1\otimes A(s))\eta_1-\alpha\Delta(A(s))=\frac25(x_{23}+x_{32})s^{(4)}\otimes s+\frac15(x_{23}+x_{32})s^{(2)}\otimes s^{(3)}+\frac45(x_{23}+x_{32})s^{(3)}\otimes s^{(2)}+(1\otimes A_1(s))\eta_1-\alpha\Delta(A_1(s))$, where $A_1(s)=-\frac15(4x_{23}-x_{32})s^{(4)}+A(s)$.

Next we assume that $\lambda\leq -4$.  Let $j=d$, $n=2-\lambda$ and $i+j+d-1=n$ in (\ref{L10}). Then $(n+1-3j)x_{n-j\ j}=(n+1-3j)x_{j\ n-j}$. If $n+1\neq 3j$ for any $j$, then $x_{n-j\ j}=x_{j\ n-j}$. Thus $(2i-n)x_{n-1\ 1}=(n-i-1)x_{i\ n-i}-(i-1)x_{n-i\ i}=(n-2i)x_{i\ n-i}$ by (\ref{L10}). Hence $x_{i\ n-i}=-x_{n-1\ 1}$ for all $2\leq i\leq n-2$. Since $(2i-n-1)x_{n\ 0}=(n+1-i)x_{i\ n-i}-ix_{n-i+1\ i-1}$ for $2\leq i\leq n-1$,
\begin{eqnarray}\label{L12}\left\{\begin{array}{l}(3-n)x_{n\ 0}=(n-1)x_{2\ n-2}-2x_{n-1\ 1}=-(n+1)x_{n-1\ 1}\\
(5-n)x_{n\ 0}=(n-2)x_{3\ n-3}-3x_{n-2\ 2}=(5-n)x_{n-1\ 1}.\end{array}\right.\end{eqnarray}
Then $x_{n-1\ 1}=x_{n\ 0}=0$ and $x_{n\ 0}=x_{i\ n-i}=0$ for all $i$. Suppose $n+1=3i_0$ for some $i_0$.  Then $(2i-n-1)x_{n\ 0}=(n+1-i)x_{i\ n-i}-ix_{n-i+1\ i-1}=(n+1-2i)x_{1\ n-1}$  for $i\neq i_0$ and $i\neq i_0+1$. Since $n\geq 6$, $i_0\geq 3$. If $i_0 >3$, then $2\neq i_0$ and $3\neq i_0+1$. Thus (\ref{L12}) holds.
Hence $x_{n\ 0}=x_{ n-1\ 1}=0$.  Thus $0=(2i_0-n-1)x_{n\ 0}=(n+1-i)x_{i_0\ n-i_0}-ix_{n-i_0+1\ i_0-1}=(n+1-i)x_{i_0\ n-i_0}$. So $x_{i_0\ n-i_0}=0$. Since $(i-j)x_{i+j-1\ 1}=(j-1)x_{i\ j}-(i-1)x_{j\ i}$, $(n-1-i_0)x_{i_0\ n-i_0}-(i_0-1)x_{n-i_0\ i_0}=0$.
Then $x_{n-i_0\ i_0}=0$. Hence $x_{i\ n-i}=0$ for all $i$.
If $i_0=3$, then $n=8$. Let $i=d=2$ and $j=5$ in (\ref{L10}). Then $x_{62}=-x_{26}$ and $x_{26}=x_{6 2}=x_{71}=0$. Thus $x_{8 0}=0$ and $x_{35}=0$ by (\ref{L12}).
Let $d=1$, $i=5$ and $j=3$ in (\ref{L10}). Then $x_{71}=x_{53}-2x_{35}$ and $x_{53}=0$.
Since  $-x_{80}=5x_{44}-4x_{53}$, $x_{44}=0$. Thus $x_{8-i\ i}=0$ for $1\leq i\leq 8$. Therefore $\alpha'=(1\otimes A(s))\eta_1-\alpha\Delta(A(s))$ for all $\lambda\leq -4$.
\end{proof}

If $\eta_1=0$ in (\ref{aa4}), then $\eta_2=0$. Thus (\ref{aa1}) becomes $(\alpha\Delta\otimes 1)\alpha'=(1\otimes \alpha\Delta)\alpha'-(12)(1\otimes \alpha\Delta)\alpha'$.

\begin{lemma}\label{Aa3} Let $\alpha=s\otimes 1-1\otimes s$ and $\alpha'=\sum\limits_{i=0}^m\sum\limits_{j=0}^{p_i}x_{ij}s^{(i)}\otimes s^{(j)}\neq 0$ satisfying
\begin{eqnarray}\label{L15}(\alpha\Delta\otimes 1)\alpha'=(1\otimes \alpha\Delta)\alpha'-(12)(1\otimes \alpha\Delta)\alpha'.\end{eqnarray}
Then    $\alpha'=\sum\limits_{i=1}^m\sum\limits_{t=0}^i\frac{2t-i}{i}x_{i0}s^{(t)}\otimes s^{(i-t)}=\alpha\sum\limits_{i=1}^m\frac1ix_{i0}\Delta(s^{(i-1)})$ for some  $x_{i0}\in{\bf k}$.
\end{lemma}

\begin{proof} Applying $\varepsilon\otimes 1\otimes 1$ to (\ref{L15}), we obtain $\Delta(s)\alpha'=(1\otimes s-s\otimes 1)\sum\limits_{j=0}^{p_0}x_{0j}\Delta(s^{(j)})$ and $\sum\limits_{i=0}^mx_{i0}s^{(i)}=-\sum\limits_{j=0}^{p_0}x_{0j}s^{(j)}$ where $\alpha'=\sum\limits_{i=0}^m\sum\limits_{j=0}^{p_i}x_{ij}s^{(i)}\otimes s^{(j)}$. Then $x_{0i}=-x_{i0}$.
Moreover, $\Delta(s)\alpha'=(s\otimes 1-1\otimes s)\sum\limits_{i=0}^mx_{i0}\Delta(s^{(i)})$ means that
$$\sum\limits_{i=0}^m\sum\limits_{j=0}^{p_i}x_{ij}((i+1)s^{(i+1)}\otimes s^{(j)}+(j+1)s^{(i)}\otimes s^{(j+1)})=\sum\limits_{i=0}^m\sum\limits_{t=0}^i(t+1)x_{i0}(s^{(t+1)}\otimes s^{(i-t)}-s^{(i-t)}\otimes s^{(t+1)}).$$
Thus $i+j\leq m$ and $ix_{i-1\ j}+jx_{i\ j-1}=(i-j)x_{i+j-1\ 0}$. Let $k=i+j-1$. Then $ix_{i-1\ k+1-i}+(k+1-i)x_{i\ k-i}=(2i-k-1)x_{k\ 0}$. If $i=1$, then $x_{1\ k-1}=-\frac{k-2}{k}x_{k0}$. Suppose $x_{i\ k-i}=-\frac{k-2i}{k}x_{k0}$. Then $x_{i+1\ k-i-1}=-\frac{i+1}{k-i}x_{i\ k-i}+\frac{2i-k+1}{k-i}x_{k0}=\frac{i+1}{k-i}\frac{k-2i}{k}x_{k0}+\frac{2i-k+1}{k-i}x_{k0}=-\frac{k-2(i+1)}{k}x_{k0}.$ Consequently,
$\alpha'=\sum\limits_{i=1}^m\sum\limits_{t=0}^i\frac{2t-i}{i}x_{i0}s^{(t)}\otimes s^{(i-t)}.$  It is routine to check (\ref{L15}) holds for this $\alpha'$.
\end{proof}

Summing up all the above discussion, we get the following result.
\begin{theorem}\label{thm27}Let $L$ be a Leibniz $H$-pseudoalgebras of rank two, where $H={\bf k}[s]$. Then $L$ is one of the  following types

(1) $He_0\oplus He_0$ is a direct sum of two Virasoro Lie conformal algebras.

(2) Abelian Lie $H$-pseudoalgebra.

(3) $L$ has a basis $\{e_0,e_1\}$ such that $[e_0,e_0]=\alpha'\otimes_He_1$, $[e_0,e_1]=[e_1,e_0]=[e_0,e_0]=0$, where $\alpha'$ is a nonzero element in $H^{\otimes2}$. Moreover, $L$ is a Lie $H$-pseudoalgebra if and only if $\alpha'=-(12)\alpha'$.

(4) $L$ has a basis $\{e_0,e_1\}$ such that $[e_0,e_1]=-((12)\otimes_H1)[e_1,e_0]=(A(s)\otimes1)\otimes_He_1$, $[e_0,e_0]=[e_1,e_1]=0$, where $A(s)$ is a nonzero element in $H$. Moreover, $L$ is a Lie $H$-pseudoalgebra.

(5) $L$ has a basis $\{e_0,e_1\}$ such that $[e_0,e_1]=(A(s)\otimes1)\otimes_He_1$, $[e_0,e_0]=[e_1,e_0]=[e_1,e_1]=0$, where $A(s)$ is a nonzero element in $H$.

(6) $L$ is a direct sum of a Virasoro Lie conformal algebra $He_0$ and an abelian Lie confromal algebra $He_1$.

(7) $L$ has a basis $\{e_0,e_1\}$ such that $[e_0,e_1]=-((12)\otimes_H1)[e_1,e_0]=(\lambda s\otimes 1-1\otimes s+\kappa\otimes 1)\otimes _He_1$, $[e_1,e_1]=0$,
$[e_0,e_0]=\alpha\otimes_He_0$, where $\lambda,\kappa\in{\bf k}$ and $\kappa\neq 0$. Moreover, $L$ is a Lie $H$-pseudoalgebra.

(8) $L$ has a basis $\{e_0,e_1\}$ such that $[e_0,e_1]=-((12)\otimes_H1)[e_1,e_0]=(\lambda s\otimes 1-1\otimes s)\otimes _He_1$,
$[e_0,e_0]=\alpha\otimes_He_0$,  $[e_1,e_1]=0$,  where $\lambda\notin\{0,-1,-2,-5,-7\}$. Moreover, $L$ is a Lie $H$-pseudoalgebra.

(9) $L$ has a basis $\{e_0,e_1\}$ such that $[e_0,e_1]=-((12)\otimes_H1)[e_1, e_0]=-1\otimes s\otimes _He_1$,
$[e_0,e_0]=\alpha\otimes_He_0+x_{10}\alpha\otimes_He_1$, $[e_1,e_1]=0$,   where $x_{10}\in{\bf k}$. Moreover, $L$ is a Lie $H$-pseudoalgebra.

(10) $L$ has a basis $\{e_0,e_1\}$ such that $[e_0,e_1]=-((12)\otimes_H1)[e_1,e_0]=(-s\otimes 1-1\otimes s)\otimes _He_1$,
$[e_0,e_0]=\alpha\otimes_He_0+(x_{12}(s\otimes s^{(2)}-s^{(2)}\otimes s)+x_{20}(s^{(2)}\otimes 1-1\otimes s^{(2)}))\otimes_He_1$,  $[e_1,e_1]=0$,  where $x_{12},x_{20}\in{\bf k}$. Moreover, $L$ is a Lie $H$-pseudoalgebra.

(11) $L$ has a basis $\{e_0,e_1\}$ such that $[e_0,e_1]=-((12)\otimes_H1)[e_1,e_0]=(-2s\otimes 1-1\otimes s)\otimes _He_1$,
$[e_0,e_0]=\alpha\otimes_He_0+(x_{13}(s\otimes s^{(3)}-s^{(3)}\otimes s)+x_{30}(s^{(3)}\otimes 1-1\otimes s^{(3)}))\alpha\otimes_He_1$,  $[e_1,e_1]=0$,  where $x_{13}, x_{30}\in{\bf k}$. Moreover, $L$ is a Lie $H$-pseudoalgebra.

(12) $L$ has a basis $\{e_0,e_1\}$ such that $[e_0,e_1]=-((12)\otimes_H1)[e_1,e_0]=(-5s\otimes 1-1\otimes s)\otimes _He_1$,
$[e_0,e_0]=\alpha\otimes_He_0+x_{34}(s^{(3)}\otimes s^{(4)}-s^{(4)}\otimes s^{(3)})\otimes_He_1$,  $[e_1,e_1]=0$,  where $x_{34}\in{\bf k}$. Moreover, $L$ is a Lie $H$-pseudoalgebra.

(13) $L$ has a basis $\{e_0,e_1\}$ such that $[e_0,e_1]=-((12)\otimes_H1)[e_1,e_0]=(-7s\otimes 1-1\otimes s)\otimes _He_1$,
$[e_0,e_0]=\alpha\otimes_He_0+x_{36}((s^{(3)}\otimes s^{(6)}-s^{(6)}\otimes s^{(3)}-3(s^{(4)}\otimes s^{(5)}-s^{(5)}\otimes s^{(4)}))\otimes_He_1$,  $[e_1,e_1]=0$,  where $x_{36}\in{\bf k}$. Moreover, $L$ is a Lie $H$-pseudoalgebra.

(14) $L$ has a basis $\{e_0,e_1\}$ such that $[e_0,e_1]=(\lambda s\otimes 1-1\otimes s+\kappa\otimes 1)\otimes _He_1$, $[e_1,e_0]= [e_1,e_1]=0$,
$[e_0,e_0]=\alpha\otimes_He_0$, where $0\neq \kappa, \lambda \in{\bf k}$, or $\kappa=0, $ $\lambda\in{\bf k}$ and $\lambda\notin \{1,-1,-2,-3\}$.

(15) $L$ has a basis $\{e_0,e_1\}$ such that $[e_0,e_1]=(s\otimes 1-1\otimes s)\otimes _He_1$, $[e_1,e_0]= [e_1,e_1]=0$, $[e_0,e_0]=\alpha\otimes_He_0+(x_{00}\otimes 1)\otimes_He_1$ for some $x_{00}\in{\bf k}$.

(16) $L$ has a basis $\{e_0,e_1\}$ such that $[e_0,e_1]=(-s\otimes 1-1\otimes s)\otimes _He_1$, $[e_1,e_0]= [e_1,e_1]=0$,   $[e_0,e_0]=\alpha\otimes_He_0+(x_{30}s^{(3)}\otimes 1)\otimes_He_1$ for some $x_{30}\in{\bf k}$.

(17) $L$ has a basis $\{e_0,e_1\}$ such that $[e_0,e_1]=(-2s\otimes 1-1\otimes s)\otimes _He_1$, $[e_1,e_0]= [e_1,e_1]=0$,  $[e_0,e_0]=\alpha\otimes_He_0+x_{22}(s^{(2)}\otimes s^{(2)}+\frac32s^{(3)}\otimes s)\otimes_He_1$ for some $x_{22}\in{\bf k}$.

(18) $L$ has a basis $\{e_0,e_1\}$ such that $[e_0,e_1]=(-3s\otimes 1-1\otimes s)\otimes _He_1$, $[e_1,e_0]= [e_1,e_1]=0$,   $[e_0,e_0]=\alpha\otimes_He_0+x_{23}(s^{(2)}\otimes s^{(3)}+4s^{(3)}\otimes s^{(2)}+2s^{(4)}\otimes s)\otimes_He_1$ for some $x_{23}\in{\bf k}$.
\end{theorem}
\begin{proof}If the solvable radical of $L$ is zero, then $L$ is semisimple Lie $H$-pseudoalgebra. Thus $L$ is a direct sum of two Virasoro Lie conformal algebras by Theorem 13.3 of \cite{BDK}. If $L$ is solvable and is not abelian, then $L$ is a solvable with maximal derived series. Thus $L$ is isomorphic to
one of (3)-(5) by Lemma \ref{lem22}. Suppose the rank of the  solvable radical of $L$ is one and there is a basis $\{e_0,e_1\}$ such that $[e_1,e_1]=[e_0,e_1]=[e_1,e_0]=0$. Then $[e_0,e_0]=\alpha\otimes_He_0+\alpha\Delta(A(s))\otimes_He_1$ for some $A(s)\in H$  by Lemma \ref{Aa3}. Let $e_0'=e_0+A(s)e_1$. Then $[e_0', e_0']=\alpha\otimes_He_0'$, $[e_0',e_1]=[e_1,e_0']=0$. Thus $L$ is a direct sum of a Virasoro Lie conformal algebra and an abelian Lie conformal algebra of rank one.   In general, if the rank of the  solvable radical of $L$ is one, then $L$ has a basis $\{e_0,e_1\}$ with pseudoproduct given by (\ref{m1}). Moreover, $\alpha',\eta_1,\eta_2$ are subject to relations from (\ref{aa1}) to (\ref{aa4}). In the case when $\eta_2=-(12)\eta_1\neq 0$. Then $\eta_1=\lambda s\otimes 1-1\otimes s+\kappa\otimes 1$ for some $\lambda,\kappa\in{\bf k}$ by equations from (\ref{aa2}) to (\ref{aa4}). Thus, $\alpha'$ is determined by (\ref{aa1}). From Lemma \ref{Aaa1}, we know that $\alpha'=A(s)s\otimes 1-1\otimes A(s)s+\lambda(s\otimes A(s)-A(s)\otimes s)+\kappa(1\otimes A(s)-A(s)\otimes 1)-\alpha\Delta(A(s))$ for some $A(s)\in H$ when either $\kappa\neq 0$ or $\kappa=0$ and $\lambda\notin\{0,-1,-2,-5,-7\}$. Let $e_0'=e_0+A(s)e_1$. Then $[e_0',e_0']=\alpha\otimes_He_0'$, $[e_0',e_1]=-((12)\otimes_H1)[e_1,e_0']=\eta_1\otimes_He_1'$ and $[e_1,e_1]=0$. Thus $L$ is isomorphic to either the Lie $H$-pseudoalgebra in (7) if $\kappa\neq 0$, or the Lie $H$-pseudoalgebra in (8) if $\kappa=0$. If $\kappa=\lambda=0$, then $\alpha'=x_{10}\alpha+A(s)s\otimes 1-1\otimes A(s)s-\alpha\Delta(A(s))$ for some $A(s)\in H$. Let $e_0'=e_0+A(s)e_1$. Then $[e_0',e_0']=\alpha\otimes_He_0'+x_{10}\alpha\otimes_He_1$, $[e_0',e_1]=-((12)\otimes_H1)[e_1,e_0']=\eta_1\otimes_He_1'$ and $[e_1,e_1]=0$. Thus $L$ is isomorphic to the Lie $H$-pseudoalgebra in (9). Similarly, we can prove that $L$ is isomorphic to the Lie $H$-pseduoalgebras in (10), (11), (12) and (13) respectively if $\kappa=0$ and $\lambda=-1,-2,-5,-7$ respectively.
In the case when $\eta_1\neq 0$ and $\eta_2=0$. Then we can prove that $L$ is isomorphic to one of  the Leibniz $H$-pseudoalgebras described in from (14) to (18)  by using Lemma \ref{Aa1}.
 \end{proof}
Recall that a Lie conformal algebra $L$ introduced in \cite{Ka}  is a $\mathbb{C}[\partial]$-module endowed with $\mathbb{C}$-linear mappings $L\otimes L\to \mathbb{C}[\partial]\otimes L$, $a\otimes b\mapsto [a_{\lambda}b]$ satisfying the following axioms:
\begin{eqnarray*} & [\partial a_{\lambda}b]=-\lambda[a_{\lambda}b],\qquad [a_{\lambda}\partial b]=(\partial+\lambda)[a_{\lambda}b],\\ & [b_{\lambda}a]=-[a_{-\partial-\lambda}b],\ \ [a_{\lambda}[b_{\mu}c]]=[[a_{\lambda}]_{\lambda+\mu}c]+[b_{\mu}[a_{\lambda}c]].\end{eqnarray*}
For any Lie conformal algebra $L$, let $[a,b]=\sum\limits_iQ_i(-s\otimes1,s\otimes 1+1\otimes s)\otimes_{{\mathbb{C}[s]}}c_i$ for any
$[a_{\lambda}b]=\sum\limits_iQ_i(\lambda,\partial)c_i$. Then $L$ becomes a  Lie $H$-psuedoalgebra, where $H=\mathbb{C}[s]$. Conversely, if $L$ is a Lie $H$-pseudoalgebras, where $H={\bf k}[s]$, then $L$ is a Lie conformal algebra with the following mapping
$[a_{\lambda}b]=\sum\limits_iP_i(\lambda)c_i$, where $[a,b]=\sum\limits_{i}P_i(s)\otimes1\otimes_Hc_i$.

\begin{remark}The classification of rank two Lie conformal algebras were achieved by many authors (\cite{BCH},\cite{HL},\cite{Ka1}). Our classifications of Leibniz $H$-pseudoalgebras of rank two includes their results and our method is different.

The central extension of a Lie (Leibniz) $H$-pseudoalgebra $L$  is measured by a cohomological group $H^2(L,M)$ (respectively $Hl^2(L,M)$) (see \cite{BDK} and \cite{Wu2}), where $M$ is a representation of $L$. The dimension of $H^2(L,M)$ has been given in \cite{BKV} for the  Virasoro Lie conformal algebra $L$ and its representations $M$ of rank one. Theorem \ref{thm27} gives a basis of $H^2(L,M)$. In addition, from  Theorem \ref{thm27}, we get $Hl^2(L,H)=H^2(L,H)=H$ when $L$ is the Virasoro  Lie conformal algebra and $Hl^2(L,H)=H^{\otimes 2}\neq H^2(L,H)=\{a\in H^{\otimes 2}|a=-(12)a\}$ when $L$ is a trivial
$H$-pseudoalgebra.
\end{remark}

 \section{Definition of Schr\"odinger-Virasoro Lie $H$-pseudoalgebras}

In this section, we introduce the Schr\"odinger-Virasoro $H$-pseudoalgebras. First let us fix some notations used in remainder of this paper.  We always assume that $\beta_i=\lambda_is\otimes 1-1\otimes s+\kappa_i\otimes 1$,
where $\lambda_i,\kappa_i\in {\bf k}$ for $1\leq i\leq 3$ and $\alpha=s\otimes 1-1\otimes s$.
We set $\alpha'_0=0$. Further, we give the definition of the Schr\"odinger-Virasoro Lie $H$-pseudoalgebra as follows.
\begin{definition} \label{def25} Let $L=He_0\oplus He_1\oplus He_2\oplus He_3$ be a free $H$-module with basis $\{e_0, e_1, e_2, e_3\}$,
and  $\alpha'_m=\sum\limits_{0\leq i,j\leq m}c_{ij}s^{(i)}\otimes s^{(j)}$ for some $c_{ij}\in{\bf k}$, where $c_{im}, c_{mj}$ are not simultaneously  zero.
Suppose $L$ is a  Lie  $H$-pseudoalgebra with pseudobrackets  given  by $$(I*)\left\{\begin{array}{l}[e_0,e_0]=\alpha\otimes_He_0, \quad \ [e_0, e_i]=\beta_{i}\otimes_H e_i\quad 1\leq i\leq 3,\\

[e_1, e_1]=\alpha_m'\otimes_He_2,\quad
[e_1, e_2]=\eta\otimes_He_2,\\

 [e_1, e_3]=\eta_{11}\otimes_He_1+\eta_{12}\otimes_He_2,\\

[e_2, e_3]=\eta_{21}\otimes_He_1+\eta_{22}\otimes_He_2,\\

[e_2, e_2]=[e_3,e_3]=0.\end{array}\right.$$

\noindent Then $L$ is called a {\it Schr\"odinger-Virasoro Lie $H$-pseudoalgebra}.
Its subalgebra $L_s=He_0\oplus He_1\oplus He_2$ is called an {\it $m$-type Schr\"odinger-Virasoro Lie conformal algebra}, or  {\it $m$-type Schr\"o-dinger-Virasoro Lie $H$-pseudoalgebra}.\end{definition}

Since $[e_1,e_1]=-((12)\otimes_H1)[e_1,e_1]$,  we have $\alpha_m'=-(12)\alpha_m'$, that is,  $\alpha'_m$  is asymmetric.  So, if $\alpha_m'\neq 0$, then  we can always assume that
 $$\alpha'_m=\sum\limits_{0\leq i<j\leq m}w_{ij}(s^{(i)}\otimes s^{(j)}-s^{(j)}\otimes s^{(i)}),$$where $w_{im} \in{\bf k}$ are not all zero.

An $m$-type Schr\"odinger-Virasoro Lie conformal algebra is an extension of Virasoro conformal Lie algebra by a solvable $H$-pseudoalgebra $I_1=He_1\oplus He_2$. If $\alpha'_m\neq 0$ or $\eta\neq 0$,  then $I$ is  a solvable $H$-pseudoalgebra with maximal derived series.  For any $\alpha_m'$ and $\eta$, the ideal
$I=He_1+He_2+He_3$ of the Schr\"odinger-Virasoro Lie $H$-pseudoalgebra is also an extension of the abelian Lie $H$-pseudoalgebra $He_3$ by $I_1$  and
the Schr\"odinger-Virasoro Lie $H$-pseudoalgebra is an extension of $He_0$ by $I$.

\begin{example}Let $TSV(c)=HL\oplus HY\oplus HM$ be the Lie conformal algebra defined in \cite{HY}, where $H={\bf k}[\partial]$. Denote $e_0: =L, e_1: =Y$, $e_2: =M$ and $s: =\partial$. Then
\begin{eqnarray*} & [e_0, e_1]=(\frac12 s\otimes 1-1\otimes s-c\otimes1)\otimes _He_1,\\
& [e_0, e_2]=(-s\otimes 1-1\otimes s-2c\otimes 1)\otimes_He_2,\\
& [e_1,e_1]=2((s^{(2)}\otimes1 -1\otimes s^{(2)})+c(s\otimes1 -1\otimes s))\otimes_He_2.\end{eqnarray*} Thus $TSV(c)$ is a 2-type Schr\"odinger-Virasoro Lie conformal algebra with $\lambda_1=\frac12$, $\lambda_2=-1$, $\kappa_1=-c$, $\kappa_2=-2c$ and $c_{02}=-2$.

Similarly, $T(a,b)$ defined in \cite{HY} is a 1-type Schr\"odinger-Virasoro Lie conformal algebra, where $e_0=L$, $e_1=Y$, $e_2=M$, $s=\partial$, $\lambda_1=a-1,\lambda_2=2a-3$,
$\kappa_1=b$, $c_{01}=-1$ and $\kappa_2=2b$. In particular, the Schr\"odinger-Virasoro type Lie conformal algebra $DSV=T(0, 0)$ defined in \cite{WXX} is a 1-type Schr\"odinger-Virasoro Lie conformal algebra with $\kappa_1=\kappa_2=0$, $\lambda_1=-1$ and $\lambda_2=-3$.
\end{example}
\begin{example}\label{exaa27}Let $H=\mathbb{C}[\partial]$ and $\widetilde{SV}=HL\oplus HY\oplus HM\oplus HN$. Then $\widetilde{SV}$ is a Schr\"odinger-Virasoro Lie $H$-pseudoalgebra with
\begin{eqnarray*} & [L, L]=(1\otimes\partial-\partial\otimes1)\otimes_HL, \  \  \  [L, Y]=(\frac12\partial\otimes1- 1\otimes\partial)\otimes_HY,\\ & [L, M]=1\otimes\partial\otimes_HM,\ \ \ [L ,N]=1\otimes\partial\otimes_HN,\\ & [Y, Y]=(1\otimes\partial-\partial\otimes1)\otimes_HM,\ \ \  [Y, M]=0, \ \ \  [Y, N]=-1\otimes1\otimes_HY,
\\ & [M, N]=-2\otimes 1\otimes_HM,\ \ \  [M, M]=[N, N]=0,\end{eqnarray*} and $SV=HL\oplus HY\oplus HM$ is a subalgebra of $\widetilde{SV}$. In \cite{SY}, these two algebras $SV$ and $\widetilde{SV}$ are called the {\it Schr\"odinger-Virasoro Lie conformal algebra} and the {\it extended Sch\"odinger-Virasoro Lie conformal algebra} respectively.

Let $e_0=-L$, $e_1=-Y$, $e_2=M$, $e_3=N$ and $s=\partial$. Then $\widetilde{SV}$ is a Sch\"odinger-Virasoro Lie $H$-pseudoalgebra with
\begin{eqnarray*}& \beta_1=\frac12s\otimes1-1\otimes s, \ \beta_2=-1\otimes s, \beta_3=-1\otimes s, \\ & \alpha_1'=s\otimes 1-1\otimes s, \ \eta=\eta_{12}= \eta_{21}=0, \ \eta_{22}=2\eta_{11}=2\otimes 1,\end{eqnarray*}  where $H={\bf k}[s]$. It is a Sch\"odinger-Virasoro Lie $H$-pseudoalgebra, where $(\lambda_1, \kappa_1)=(\frac12,0)$, $(\lambda_2,  \kappa_2)=(0,0)$ and $(\lambda_3, \kappa_3)=(0, 0).$ Moreover, the Schr\"odinger-Virasoro Lie conformal algebra $SV$ is a 1-type
Lie conformal algebra.
\end{example}
To determine all Schr\"odinger-Virasoro Lie $H$-pseudoalgebras, we need to describe $\beta_i,\alpha_m'$, $\eta$ and $\eta_{jk}$ for $1\leq i\leq 3$ and $1\leq j,k\leq 2$. For this purpose, we need the following key lemma.
\begin{lemma} \label{lem28} Let $L$ be a free $H$-module with a basis $\{e_0,e_1,e_2,e_3\}$. Then $L$ is a Schr\"odinger-Virasoro Lie $H$-pseudoalgebra with the pseudobrackets determined by
 $(I*)$ if and only if these pseudobrackets are asymmetric and the following equations hold.

 \begin{eqnarray}(\alpha_m'\Delta\otimes1)\eta_{21}&=&(\eta\Delta\otimes 1)\eta_{21}=0,\label{eq25}\\
(\alpha_m'\Delta\otimes1)\eta_{22}&=&(1\otimes\eta_{11}\Delta)\alpha_m'-(12)(1\otimes\eta_{11}\Delta)\alpha_m'\label{eq26}\\
& & +(1\otimes\eta_{12}\Delta)\eta-(12)(1\otimes\eta_{12}\Delta)\eta, \nonumber \\
(\eta\Delta\otimes 1)\eta_{22}&=&(1\otimes \eta_{21}\Delta)\alpha_m'+(1\otimes\eta_{22}\Delta)\eta \label{eq27}\\
& &  +(12)(1\otimes\eta_{11}\Delta)((12)\eta),\nonumber \\
(\beta_{1}\Delta\otimes1)\eta_{11}
&=&(1\otimes\eta_{11}\Delta)\beta_{1}-(12)(1\otimes\beta_{3}\Delta)\eta_{11},\label{eq28}\\
(\beta_{1}\Delta\otimes1)\eta_{12}&=&
(1\otimes\eta_{12}\Delta)\beta_{2}-(12)(1\otimes\beta_{3}\Delta)\eta_{12},\label{eq29}\\
(\beta_{2}\Delta\otimes1)\eta_{21}
&=&(1\otimes\eta_{21}\Delta)\beta_{1}-(12)(1\otimes\beta_{3}\Delta)\eta_{21},\label{eq210}\\
(\beta_{2}\Delta\otimes1)\eta_{22}&=&
(1\otimes\eta_{22}\Delta)\beta_{2}-(12)(1\otimes\beta_{3}\Delta)\eta_{22},\label{eq211} \\ (\eta_{11}\Delta\otimes1)\eta_{11}+(\eta_{12}\Delta\otimes1)\eta_{21}&=&(12)(1\otimes\eta_{11}\Delta)((12)\eta_{11}) \label{eq212}\\ & & +(12)(1\otimes\eta_{12}\Delta)((12)\eta_{21}),\nonumber \\
(\eta_{11}\Delta\otimes1)\eta_{12}+(\eta_{12}\Delta\otimes1)\eta_{22}&=&(12)(1\otimes\eta_{11}\Delta)((12)\eta_{12})\label{eq213}\\
 & & +(12)(1\otimes\eta_{12}\Delta)((12)\eta_{22}),\nonumber\\
 (\eta_{21}\Delta\otimes1)\eta_{11}+(\eta_{22}\Delta\otimes1)\eta_{21}&=&(12)(1\otimes\eta_{21}\Delta)((12)\eta_{11})\label{eq214}\\
& & +(12)(1\otimes\eta_{22}\Delta)((12)\eta_{21}),\nonumber \\
(\eta_{21}\Delta\otimes1)\eta_{12}+(\eta_{22}\Delta\otimes1)\eta_{22}&=&(12)(1\otimes\eta_{21}\Delta)((12)\eta_{12})\label{eq215}\\
& & +(12)(1\otimes\eta_{22}\Delta)((12)\eta_{22}),\nonumber\\
 (1\otimes\eta_{21}\Delta)((12)\eta)&=& (12)(1\otimes\eta_{21}\Delta)((12)\eta),\label{eq216}\\
(\beta_1\Delta\otimes1)\eta &=&(1\otimes\eta\Delta)\beta_2-(12)(1\otimes\beta_2\Delta)\eta,\label{eq217}\\
(\beta_1\Delta\otimes 1)\alpha_m'&=&(1\otimes\alpha_m'\Delta)\beta_2-(12)(1\otimes\beta_1\Delta)\alpha_m',\label{eq218}\\
-(\alpha'_m\Delta\otimes1)((12)\eta)&=&(1\otimes\alpha_m'\Delta)\eta-(12)(1\otimes\alpha_m'\Delta)\eta,\label{eq219}\\
(1\otimes\eta\Delta)\eta &=&(12)(1\otimes\eta\Delta)\eta.\label{eq220}\end{eqnarray}
\end{lemma}
\begin{proof}It is easy to check (\ref{eq25})-(\ref{eq220}) by Definition \ref{def25} and using $$[[e_i  ,e_j], e_k]=[e_i, [e_j, e_k]]-((12)\otimes_H1)[e_j, [e_i, e_k]]$$ for all $0\leq i<j<k\leq 3$.\end{proof}
If (\ref{eq25}) holds, then (\ref{eq216}) holds and (\ref{eq27}) can be simplified as
\begin{eqnarray}\label{eq221}(\eta\Delta\otimes 1)\eta_{22}=(1\otimes\eta_{22}\Delta)\eta
  +(12)(1\otimes\eta_{11}\Delta)((12)\eta). \end{eqnarray}
\section{ $m$-type Schr\"odinger-Virasoro Lie conformal algebras }
In this section, we determine all $m$-type Schr\"odinger-Virasoro Lie conformal algebras. By Lemma \ref{lem28}, $L_s=He_0\oplus He_1\oplus He_2$ is an  $m$-type Schr\"odinger-Virasoro Lie conformal algebra with the pseudobrackets given by
\begin{eqnarray*} (I**)\left\{\begin{array}{l} [e_0, e_0]=\alpha\otimes_He_0,\qquad [e_0, e_i]=\beta_{i}\otimes_H e_i,\qquad i=1,2\\

[e_1, e_1]=\alpha_m'\otimes_He_2,\qquad
[e_1, e_2]=\eta\otimes_He_2,\qquad
[e_2, e_2]=0\end{array}\right.\end{eqnarray*}
if and only if (\ref{eq217})-(\ref{eq220}) hold and the pseudobrackets in $(I**)$ are asymmetric.

\begin{lemma} \label{lem24} Let  $L=He_0\oplus He_1\oplus He_2$ be an $m$-type  Lie conformal algebra. Then
$\eta=a\otimes 1$ for some $a\in{\bf k}$ with $a\lambda_1=a\kappa_1=0$ and  $(2\kappa_1-\kappa_2)\alpha'_m=0$.
If $\eta=a\otimes 1\neq 0$, then $\alpha'_m=\sum\limits_{j=1}^mw_{0j}(1\otimes s^{(j)}-s^{(j)}\otimes 1)$, where $w_{0j}\in \bf{k}$.\end{lemma}

\begin{proof} Assume that $\eta=\sum\limits_{i=0}^n b_is^{(i)}\otimes 1$ by Lemma \ref{lem22}. It follows from (\ref{eq217}) that  $\beta_1\Delta(\sum\limits_{i=0}^nb_is^{(i)})$ $=-1\otimes s \sum\limits_{i=0}^nb_is^{(i)}$.
Thus $\lambda_1s\sum\limits_{i=0}^nb_is^{(i)}=\kappa_1\sum\limits_{i=0}^nb_is^{(i)}=0$ and $\sum\limits_{i=0}^n\sum\limits_{0\leq t\leq i}b_i(i-t+1)s^{(t)}\otimes s^{(i-t+1)}=\sum\limits_{i=0}^n b_i(i+1)\otimes s^{(i+1)}$. If $n\geq 1$, then the term $b_ns^{(n)}\otimes s$ on the left hand of the previous equation can not be cancelled. Hence $n=0$. Thus $\eta=a\otimes 1$ for $a=b_0\in{\bf k}$ with $a\lambda_1=a\kappa_1=0$.

Applying $\varepsilon\otimes 1\otimes 1$ to equation {(\ref{eq218})}, we get $(\kappa_1\otimes 1-s\otimes 1)\alpha'_m=\alpha_m'(\kappa_2\otimes1-1\otimes s-s\otimes 1)-\alpha'_m(\kappa_1\otimes s-1\otimes s)$, which implies
 that $(2\kappa_1-\kappa_2)\alpha_m'=0$.

Suppose $a\neq 0$ and $\alpha_m'=\sum\limits_{i=0}^m\sum\limits_{j=0}^{q_i}y_{ij}s^{(i)}\otimes s^{(j)}\neq 0$.
Observe that $$[[e_1, e_1], e_1]=[e_1, [e_1, e_1]]-(12)[e_1, [e_1, e_1]].$$This means \begin{eqnarray} \qquad \qquad \sum\limits_{i=0}^m\sum\limits_{j=0}^{q_i}
y_{ij}\left((1\otimes s^{(i)}-s^{(i)}\otimes 1)\otimes s^{(j)}+
s^{(i)}\otimes s^{(j)}\otimes 1\right)=0.\label{eq021}\end{eqnarray}
From this, we get $\sum\limits_{i=1}^my_{i0}(1\otimes s^{(i)}-s^{(i)}\otimes 1)=-\sum\limits_{i=0}^m\sum\limits_{j=0}^{q_i}
y_{ij}s^{(i)}\otimes s^{(j)}.$ Hence  $\alpha_m'=\sum\limits_{i=1}^mw_{0i}(1\otimes s^{(i)}-s^{(i)}\otimes 1)$, where $w_{0i}=y_{i0}$.
\end{proof}

It follows from Lemma \ref{lem24} that (\ref{eq217}) and (\ref{eq220}) hold if and only if
$\eta=a\otimes 1$ for some $a\in{\bf k}$ with $a\lambda_1=a\kappa_1=(2\kappa_1-\kappa_2)\alpha_m'=0$.
If $\eta=a\otimes 1$ for a nonzero $a$, then $\alpha_m'=\sum\limits_{j=0}^mw_{0j}(1\otimes s^{(j)}-s^{(j)}\otimes 1)$ by Lemma \ref{lem24}.
It is easy to see that (\ref{eq219}) holds for this $\eta$ and $\alpha_m'$. (\ref{eq219}) is trivial if $\eta=0$.
To determine $m$ and $w_{0j}$ in $\alpha_m'$, we need the following lemma, which is also useful in the next section.

\begin{lemma}\label{lem24a}Suppose nonzero $\gamma=\sum\limits_{i=0}^m\sum\limits_{j=0}^{p_i}x_{ij}s^{(i)}\otimes s^{(j)}$ satisfying
\begin{eqnarray}(\beta_{1}\Delta\otimes1)\gamma=
(1\otimes\gamma\Delta)\beta_{2}-(12)(1\otimes\beta_{3}\Delta)\gamma.\label{eq21a}\end{eqnarray}
Then $(\kappa_1-\kappa_2+\kappa_3)\gamma=0$,  $(\lambda_1-\lambda_2+\lambda_3-i-j)x_{ij}+\kappa_1x_{i+1\ j}+\kappa_3x_{i\ j+1}=0,$
and $p_m\leq 1$. Furthermore, $p_{m-i}\leq i+1$ and $\lambda_3=0$ if $p_m=1$, and $p_{m-i}\leq i$  if $p_m=0$, where $0\leq i\leq m$.
\end{lemma}
\begin{proof}Comparing the terms $1\otimes s^{(i)}\otimes s^{(j)}$ and the terms $s\otimes s^{(i)}\otimes s^{(j)}$ in (\ref{eq21a}) respectively, we get
 $(\kappa_1-\kappa_2+\kappa_3)\gamma=0$ and
\begin{eqnarray}(\lambda_1-\lambda_2+\lambda_3-i-j)x_{ij}+\kappa_1x_{i+1\ j}+\kappa_3x_{i\ j+1}=0\label{eq24a}\end{eqnarray}
respectively.
It follows from (\ref{eq21a}) that
 \begin{eqnarray}
& &\sum\limits_{i=1}^{m}\sum\limits_{j=0}^{p_i}\sum\limits_{t=1}^ix_{ij}\lambda_1(t+1)s^{(t+1)}\otimes s^{(i-t)}\otimes s^{(j)}\nonumber\\ & &+\sum\limits_{i=0}^{m}\sum\limits_{j=1}^{p_i}\sum\limits_{l=1}^j\lambda_3x_{ij}(l+1)s^{(l+1)}\otimes s^{(i)}\otimes s^{(j-l)}\nonumber \\
& &+\sum\limits_{i=2}^{m}\sum\limits_{j=0}^{p_i}\sum\limits_{t=2}^i\kappa_1 x_{ij}s^{(t)}\otimes s^{(i-t)}\otimes s^{(j)}+\sum\limits_{i=0}^{m}\sum\limits_{j=2}^{p_i}\sum\limits_{l=2}^j\kappa_3 x_{ij}s^{(l)}\otimes s^{(i)}\otimes s^{(j-l)}\label{eq25a}\\
& &=\sum\limits_{i=2}^{m}\sum\limits_{j=0}^{p_i}\sum\limits_{t=2}^ix_{ij}(i-t+1)s^{(t)}\otimes s^{(i-t+1)}\otimes s^{(j)}\nonumber \\
& & +\sum\limits_{i=0}^{m}\sum\limits_{j=2}^{p_i}\sum\limits_{l=2}^j(j-l+1)x_{ij}s^{(l)}\otimes s^{(i)}\otimes s^{(j-l+1)}
.\nonumber\end{eqnarray}
If $q:=p_{m}\geq 2$, then \begin{eqnarray}
& &\sum\limits_{l=1}^{q}\lambda_3(l+1)s^{(l+1)}\otimes s^{(m)}\otimes s^{(q-l)}
 =\sum\limits_{l=2}^{q}({q}-l+1)s^{(l)}\otimes s^{(m)}\otimes s^{(q-l+1)},
\label{eq22abc}\end{eqnarray}
which is impossible for any $\lambda_3$. Thus $p_m\leq 1$.

 If $p_m=1$, then it holds $p_{m-i}\leq i+1$ for $i=0$. Suppose $p_{m-i}\leq i+1$ for some $i$. Then $(m+1-p_{m-i-1}-(m-i-1))x_{m-i-1\
p_{m-i-1}}+\kappa_1x_{m-i \  p_{m-i-1}}=0$ by (\ref{eq24a}). If $p_{m-i-1}\geq p_{m-i}+1$, then $(i+2-p_{m-i-1})x_{m-i-1\ p_{m-i-1}}=0$ and $p_{m-i-1}=i+2$. If $p_{m-i-1}\leq p_{m-i}$, then $p_{m-i-1}\leq p_{m-i}\leq i+1<i+2$. Hence $p_{m-i}\leq i+1$ for all $0\leq i\leq m$. Moreover, it follows $2\lambda_3s^{(2)}\otimes s^{(m)}\otimes 1=0$ from (\ref{eq22abc}), which implies that $\lambda_3=0$.

If $p_m=0$, then  $p_{m-i}\leq i$ for $i=0$. Suppose $p_{m-i}\leq i$ for some $i$. Then $(m-p_{m-i-1}-(m-i-1))x_{m-i-1\
p_{m-i-1}}+\kappa_1x_{m-i \  p_{m-i-1}}=0$ by (\ref{eq24a}). If $p_{m-i-1}\geq p_{m-i}+1$, then $(i+1-p_{m-i-1})x_{m-i-1\ p_{m-i-1}}=0$ and $p_{m-i-1}=i+1$. If $p_{m-i-1}\leq p_{m-i}$, then $p_{m-i-1}\leq p_{m-i}\leq i<i+1$. Hence $p_{m-i}\leq i$ for all $0\leq i\leq m$.
\end{proof}
Using  Lemma \ref{lem24a}, we determine all nonzero $\gamma$ for $m\leq 1$ as follows.

(T1) In the case when $m=0$ and $p_0=1$, we have $\lambda_3=0$, $\lambda_1-\lambda_2=1$ and $x_{00}+\kappa_3x_{01}=0$ by (\ref{eq24a}).  So $\gamma=x_{01}(-\kappa_3\otimes 1+1\otimes s)$ for some nonzero $x_{01}\in{\bf k}$.

(T2) In the case when $m=0$ and $p_0=0$, we have $\lambda_1-\lambda_2+\lambda_3=0$, $\kappa_1-\kappa_2+\kappa_3=0$ and hence $\gamma=x_{00}\otimes 1$ for some nonzero $x_{00}\in{\bf k}$.

(T3) In the case when $m=1$, $p_1=1$ and $p_0=2$, we have $\lambda_3=0$ and $\lambda_1-\lambda_2=2$. From (\ref{eq25a}), we get $x_{02}=2\lambda_1x_{11}$ and $2\lambda_1x_{10}+\kappa_3x_{02}=0$. Since $x_{01}=-(\kappa_1x_{11}+\kappa_3x_{02})$ and $x_{00}=-\frac12(\kappa_1x_{10}+\kappa_3x_{01})$ by (\ref{eq24a}), we have $x_{10}=-\kappa_3x_{11}$, $x_{02}=2\lambda_1x_{11}$, $x_{01}=-(\kappa_1+2\lambda_1\kappa_3)x_{11}$ and $x_{00}=(\kappa_1\kappa_3+\lambda_1\kappa_3^2)x_{11}$. Thus $\gamma=x_{11}((\kappa_1\kappa_3+\lambda_1\kappa_3^2)\otimes1-(\kappa_1+2\lambda_1\kappa_3)\otimes s+2\lambda_1\otimes s^{(2)}-\kappa_3s\otimes1+s\otimes s)$ for some nonzero $x_{11}\in{\bf k}$.

(T4) In the case when $m=1$, $p_1=1$ and $p_0=1$, we have $\lambda_3=0$ and $\lambda_1-\lambda_2=2$ by (\ref{eq24a}). From (\ref{eq25a}), we get $\lambda_1=0$. Thus $\lambda_2=-2$. Since $x_{10}=-\kappa_3x_{11}$, $x_{01}=-\kappa_1x_{11}$ and $x_{00}=\kappa_1\kappa_3x_{11}$ by (\ref{eq24a}),
$\gamma=x_{11}(\kappa_1\kappa_3\otimes1-\kappa_3s\otimes1-\kappa_1\otimes s+s\otimes s)$ for some nonzero $x_{11}\in{\bf k}$.

(T5) In the case when $m=1$, $p_1=1$ and $p_0=0$, we have $\lambda_3=0$ and $\lambda_1-\lambda_2=2$. From (\ref{eq25a}), we get $\lambda_1=0$. Thus $\lambda_2=-2$, $x_{10}=-\kappa_3x_{11}$ and $x_{01}=-\kappa_1x_{11}=0$. So $\kappa_1=0$ and
$\gamma=x_{11}(-\kappa_3s\otimes 1+s\otimes s)$ for some nonzero $x_{11}\in{\bf k}$.

(T6) In the case when $m=1$, $p_1=0$ and $p_0=1$, we have $\lambda_1-\lambda_2+\lambda_3=1$ and $\lambda_1x_{10}+\lambda_3x_{01}=0$ by (\ref{eq25a}).
So, $\gamma=-(\kappa_1x_{10}+\kappa_3x_{01})\otimes 1+x_{10}s\otimes 1+x_{01}\otimes s$ for some  $x_{10}, x_{01} \in{\bf k}$, where $x_{10}\neq 0$.

(T7) In the case when $m=1$, $p_1=0$ and $p_0=0$, we have $\lambda_1-\lambda_2+\lambda_3=1$ and $\lambda_1=0$ by (\ref{eq25a}).
Hence $\gamma=x_{10}(-\kappa_1\otimes 1+s\otimes 1)$ for some nonzero $x_{10}\in{\bf k}$.
\begin{lemma}\label{lem25a}Suppose nonzero $\gamma=\sum\limits_{i=0}^m\sum\limits_{j=0}^{p_i}x_{ij}s^{(i)}\otimes s^{(j)}$ satisfying
(\ref{eq21a}), $p_m=0$ and $m\geq 3$. Then $\lambda_3\neq 0$ and $m=3$.
\end{lemma}
\begin{proof} From (\ref{eq25a}), it is easy to see the following equations hold.
 \begin{eqnarray}\left\{\begin{array}{l}
\lambda_1x_{m0}+\lambda_3x_{0m}=0 \\
((m\lambda_3-1)x_{0m}+m\lambda_1x_{m-1\ 1})s^{(m)}\otimes 1\otimes s=0\\
((2\lambda_1-m+1)x_{m0}+2\lambda_3x_{m-1\ 1})s^{(2)}\otimes s^{(m-1)}\otimes 1=0\\
((3\lambda_1-m+2)x_{m0}+3\lambda_3x_{m-2\ 2})s^{(3)}\otimes s^{(m-2)}\otimes 1=0\\
((m-1)\lambda_3-2)x_{0m}+(m-1)\lambda_1x_{m-2\ 2})s^{(m-1)}\otimes1\otimes s^{(2)}=0.\end{array}\right.\label{eq29c}
\end{eqnarray}
If $\lambda_3=0$, then $\lambda_1x_{m0}=0$ and $\lambda_1=0$. Hence $x_{0m}=m\lambda_1x_{m-1\ 1}=0$ by the second equation in (\ref{eq29c}). Thus $(1-m)x_{m0}=0$ and $m=1$, which is impossible by the assumption.
Since $\lambda_3\neq 0$, the first three equations in (\ref{eq29c}) imply $\lambda_1+\lambda_3=\frac{m^2-m+2}{2m}$. Similarly, $\lambda_1+\lambda_3=\frac{m^2-3m+8}{3(m-1)}$ by the first and last two equations in (\ref{eq29c}). Hence $\frac{m^2-m+2}{2m}=\frac{m^2-3m+8}{3(m-1)}$ and $m=3$.\end{proof}
 From Lemma \ref{lem25a}, we have other two cases.

 (T8) In the case when $m=2$ and  $p_2=0$,  it follows
 \begin{eqnarray}\left\{\begin{array}{ll}
(2\lambda_1-1)x_{20}+2\lambda_3x_{11}=0 &
(2\lambda_3-1)x_{02}+2\lambda_1x_{11}=0\\
\lambda_1x_{20}+\lambda_3x_{02}=0&
x_{10}+\kappa_1x_{20}+\kappa_3x_{11}=0\\
x_{01}+\kappa_1x_{11}+\kappa_3x_{02}=0&
2x_{00}+\kappa_1x_{10}+\kappa_3x_{01}=0\end{array}\right.\label{eq291c}
\end{eqnarray} from (\ref{eq25a}). If $\lambda_3=0$, then $x_{20}=x_{02}=0$, which contradicts with the fact that $m=2$. Hence $\lambda_3\neq 0$. Moreover, $x_{02}=-\frac{\lambda_1}{\lambda_3}x_{20}$. Thus
$\left|\begin{array}{cc}2\lambda_1-1&2\lambda_3\\ -\frac{(2\lambda_3-1)\lambda_1}{\lambda_3}&2\lambda_1\end{array}\right|=4\lambda_1(\lambda_1+\lambda_3-1)=0$. So either $\lambda_1=0$ or $\lambda_1+\lambda_3=1$. If $\lambda_1=0$, then $$\gamma=\frac{x_{20}}{2\lambda_3}(\kappa_1(\lambda_3\kappa_1+\kappa_3)\otimes1-\kappa_1\otimes s
-(2\lambda_3\kappa_1+\kappa_3)s\otimes1+s\otimes s+2\lambda_3s^{(2)}\otimes1)$$ by (\ref{eq24a}) and (\ref{eq29c}).
If  $\lambda_1+\lambda_3=1$, then \begin{eqnarray*}\begin{array}{lll}\gamma&=&x_{20}((\frac{\lambda_3-1}{2\lambda_3}\kappa_3^2+\frac{2\lambda_3-1}{2\lambda_3}\kappa_1\kappa_3+\frac12\kappa_1^2)\otimes1
+(\frac{(1-2\lambda_3)\kappa_3}{2\lambda_3}
-\kappa_1)s\otimes1\\
&&+\frac{(2\lambda_1-1)\kappa_1+2\lambda_1\kappa_3}{2\lambda_3}\otimes s-\frac{\lambda_1}{\lambda_3}\otimes s^{(2)}
+\frac{2\lambda_3-1}{2\lambda_3}s\otimes s+s^{(2)}\otimes1)\end{array}\end{eqnarray*}
 by (\ref{eq24a}) and (\ref{eq29c}).

(T9) In the case when $m=3$ and  $p_3=0$, it follows from (\ref{eq25a}) that
 \begin{eqnarray}\left\{\begin{array}{ll}
2\lambda_1x_{10}+2\lambda_3x_{01}+\kappa_1x_{20}+\kappa_3x_{02}=0, &
2\lambda_1x_{12}+2\lambda_3x_{03}=2x_{03},\\
2\lambda_1x_{11}+2\lambda_3x_{02}+\kappa_1x_{21}+\kappa_3x_{03}=x_{02},&
\lambda_1x_{30}+\lambda_3x_{03}=0,\\
2\lambda_1x_{20}+2\lambda_3x_{11}+\kappa_1x_{30}+\kappa_3x_{12}=x_{20},&
3\lambda_1x_{30}+3\lambda_3x_{12}=x_{30},\\
3\lambda_1x_{20}+3\lambda_3x_{02}+\kappa_1x_{30}+\kappa_3x_{03}=0, &2(\lambda_1-1)x_{30}+2\lambda_3x_{21}=0,\\
2\lambda_1x_{21}+2\lambda_3x_{12}=x_{21}+x_{12}, &3\lambda_1x_{21}+3\lambda_3x_{03}=x_{03}.\end{array}\right.\label{eq292c}
\end{eqnarray}  Thus  $x_{03}=-\frac{\lambda_1}{\lambda_3}x_{30}$ and $\left|\begin{array}{cc}3\lambda_1-1&3\lambda_3\\ -\frac{(2\lambda_3-2)\lambda_1}{\lambda_3}&2\lambda_1\end{array}\right|=6\lambda_1(\lambda_1+\lambda_3-\frac43)=0$. If $\lambda_1=0$, then $x_{03}=0$, $x_{12}=\frac1{3\lambda_3}x_{30}$ and $x_{21}=\frac1{\lambda_3}x_{30}$. Thus $(2\lambda_3-1)\frac1{3\lambda_3}x_{30}=\frac1{\lambda_3}x_{30}$, which implies $\lambda_3=2$, $x_{12}=\frac16x_{30}$ and $x_{21}=\frac12x_{30}$.
By (\ref{eq24a}) and (\ref{eq292c}), one obtains
$$\begin{array}{lll}\gamma&=&x_{30}(s^{(3)}\otimes 1-(\kappa_1+\frac12\kappa_3)s^{(2)}\otimes1+\frac12s^{(2)}\otimes s+\frac12(\kappa_1^2+\frac16\kappa_3^2+\kappa_1\kappa_3)s\otimes1\\
&&-\frac16(3\kappa_1+\kappa_3)s\otimes s+\frac16s\otimes s^{(2)}
-\frac16(\kappa_1^3+\frac12\kappa_1\kappa_3^2+\frac32\kappa_1^2\kappa_3)\otimes1\\&&+\frac1{12}(3\kappa_1^2+2\kappa_1\kappa_3)\otimes s
-\frac16\kappa_1\otimes s^{(2)}).\end{array}$$ If $\lambda_1+\lambda_3=\frac43$, then $x_{12}=\frac{\lambda_3-1}{\lambda_3}x_{30}=-x_{12}$ and
$(2\lambda_1-1)x_{21}+(2\lambda_3-1)x_{12}=2(\lambda_1-\lambda_3)x_{12}=0$ by (\ref{eq292c}). Therefore $\lambda_1=\lambda_3=\frac23$ and $\lambda_2=-\frac53$.
  From (\ref{eq24a}), we get $$\begin{array}{lll}\gamma&=&x_{30}((s^{(3)}\otimes1-1\otimes s^{(3)})+\frac12(s^{(2)}\otimes s-s\otimes s^{(2)})-(\kappa_1+\frac12\kappa_3)s^{(2)}\otimes 1\\&&+(\frac12\kappa_1+\kappa_3)\otimes s^{(2)}+\frac12(\kappa_1^2+\kappa_1\kappa_3-\frac12\kappa_3^2)s\otimes1+\frac12(\frac12\kappa_1^2-\kappa_1\kappa_3-\kappa_3^2)\otimes s\\ &&+ \frac12(\kappa_3-\kappa_1)s\otimes s +\frac16(\kappa_3-\kappa_1)(\kappa_1^2+\frac52\kappa_1\kappa_3+\kappa_3^2)\otimes1).\end{array}$$

\begin{lemma}\label{lem26a}Suppose nonzero $\gamma=\sum\limits_{i=0}^m\sum\limits_{j=0}^{p_i}x_{ij}s^{(i)}\otimes s^{(j)}$ satisfying
(\ref{eq21a}) and $p_m=1$. Then $\lambda_3=0$ and $m\leq 2$.
\end{lemma}
 \begin{proof}Since $p_m=1$, $\lambda_3=0$ by Lemma \ref{lem24a}. Suppose $m\geq 3$. It is easy to see the following equations hold from (\ref{eq25a}).
\begin{eqnarray}\left\{\begin{array}{l}
(m+1)\lambda_1x_{m1}=x_{0\ m+1}, \\
((m\lambda_1-1)x_{m1}-x_{1m})s^{(m)}\otimes s\otimes s=0,\\
(m\lambda_1x_{m-1\ 2}-2x_{0\ m+1})s^{(m)}\otimes 1\otimes s^{(2)}=0,\\
(((m-1)\lambda_1-2)x_{m1}-x_{2\ m-1})s^{(m-1)}\otimes s^{(2)}\otimes s=0,\\
((2\lambda_1-m+1)x_{m1}-x_{m-1\ 2})s^{(2)}\otimes s^{(m-1)}\otimes s=0,\\
(3\lambda_1-m+2)x_{m 1}-x_{m-2\ 3})s^{(3)}\otimes s^{(m-2)}\otimes s=0,\\
((2\lambda_1-m+2)x_{m-1\ 2}-2x_{m-2\ 3})s^{(2)}\otimes s^{(m-2)}\otimes s^{(2)}=0
, \end{array}\right.\label{eq293c}
\end{eqnarray}
which implies that \begin{eqnarray*}\left\{\begin{array}{l}(m+1)\lambda_1x_{m1}=x_{0\ m+1},\\
(m\lambda_1-1)x_{m1}=x_{1m},\\
m\lambda_1x_{m-1\ 2}=2x_{0\ m+1},\\
((m-1)\lambda_1-1)x_{m-1\ 2}=2x_{1m}
\end{array}\right. \ \ \mbox{and}\ \ \ \left\{\begin{array}{l}(2\lambda_1-m+1)x_{m1}=x_{m-1\ 2},\\
(3\lambda_1-m+2)x_{m1}=x_{m-2\ 3},\\
(2\lambda_1-m+2)x_{m-1\ 2}=2x_{m-2\ 3}.
\end{array}\right.
\end{eqnarray*}
Since $x_{m1}\neq 0$ by the assumption, one obtains
$$\left|\begin{array}{cc}m\lambda_1-1&\frac12((m-1)\lambda_1-1)\\ 2(m+1)\lambda_1-1&m\lambda_1\end{array}\right|=\left|\begin{array}{cc}2\lambda_1-m+1&1\\ 3\lambda_1-m+2&\frac12({2\lambda_1-m+2})\end{array}\right|=0.$$
Solving the above equations yield $(\lambda_1, m)=(0, -1)$, or $(0, 2)$, or $(-1, -1)$, or $(-1, -2)$. This is a contradiction since  $m\geq 3$. So $m\leq 2$.
\end{proof}

From Lemma \ref{lem26a}, we get the last two cases.

(T10) In the case when $m=2$ and  $p_2=1$, it follows from (\ref{eq25a}) that
 \begin{eqnarray}\left\{\begin{array}{ll}
(2\lambda_1-1)x_{20}+\kappa_3x_{12}=0, &
3\lambda_1x_{20}+\kappa_3x_{03}=0,\\
2\lambda_1x_{21}=x_{12}+x_{21},&
3\lambda_1x_{21}=x_{03},\\
2\lambda_1x_{10}+\kappa_1x_{20}+\kappa_3x_{02}=0,& 2\lambda_1x_{12}=2x_{03},\\
2\lambda_1x_{11}+\kappa_1x_{21}+\kappa_3x_{03}=x_{02}.&\end{array}\right.\label{eq293d}
\end{eqnarray}
If $\lambda_1=0$, then $$\begin{array}{lll}\gamma&=&x_{21}((s^{(2)}\otimes s-s\otimes s^{(2)})+(\kappa_1\otimes s^{(2)}-\kappa_3s^{(2)}\otimes1)
+(\frac12\kappa_3(2\kappa_1-\kappa_3)s\otimes 1\\ &&+\frac12\kappa_1(\kappa_1-2\kappa_3)\otimes s)+(\kappa_3-\kappa_1)s\otimes s+\frac12\kappa_1\kappa_3(\kappa_3-\kappa_1)\otimes 1)\end{array}$$ by (\ref{eq24a}) and (\ref{eq293d}).
If $\lambda_1\neq 0$, then  $\lambda_1=2$, $\lambda_2=-1$ and $\lambda_3=0$ by (\ref{eq293c}). Therefore,  $$\begin{array}{lll}\gamma&=&x_{21}(s^{(2)}\otimes s+3s\otimes s^{(2)}-\kappa_3s^{(2)}\otimes 1-3(\kappa_1+2\kappa_3)\otimes s^{(2)}-(\kappa_1+3\kappa_3)s\otimes s\\ &&+(\kappa_1\kappa_3+\frac32\kappa_3^2)s\otimes 1+\frac12(\kappa_1^2+6\kappa_1\kappa_3+6\kappa_3^2)\otimes s\\ &&-\frac12(\kappa_1^2\kappa_3+3\kappa_1\kappa_2^3+2\kappa_3^3)\otimes 1+6\otimes s^{(3)}).\end{array}$$

\bigskip
Now we start to classify all $m$-type Schr\"odinger-Virasoro Lie conformal algebras based on the cases (T1)-(T10). Note that $\eta$ has been completely determined by Lemma \ref{lem24}. We only need to find all $\alpha_m'$ for this purpose.

In the case of $\alpha_m'=0$,  (\ref{eq218}) is trivial. In this case,  we get the $0$-type Schr\"odinger-Virasoro Lie conformal algebra (A),with
$$\alpha'_m=\alpha_0'= 0\quad \beta_1=\lambda_1s\otimes 1 -1\otimes s+\kappa_1\otimes1,$$   $$\beta_2=\lambda_2s\otimes 1 -1\otimes s+\kappa_2\otimes 1,\qquad \eta=a\otimes 1,$$ where $\lambda_1,\lambda_2, \kappa_1,\kappa_2, a\in{\bf k}$ satisfying $a\lambda_1=a\kappa_1=0$.

In the case of $\alpha_m'\neq 0$,  we replace $\beta_3$ with $\beta_1$ and $\gamma$ with $\alpha_m'$ in (\ref{eq21a})  to determine $\alpha_m'$ satisfying (\ref{eq218}). We also replace $x_{ij}$ with $w_{ij} (=-w_{ji})$ in (T1)-(T10).

If $m=1$, then $\alpha'_m=\alpha_1'=w_{01}(1\otimes s-s\otimes 1)\neq 0$ by (T6). Moreover, $2\lambda_1-\lambda_2=1$, \ $\kappa_2=2\kappa_1$ and $\eta=a\otimes 1$, where $w_{01}, a\in{\bf k}$ satisfying $a\lambda_1=a\kappa_1=0$ by Lemma \ref{lem24}. Thus we get a $1$-type Schr\"odinger-Virasoro Lie conformal algebra (B) with
$$\alpha'_m=\alpha_1'=w_{01}(1\otimes s-s\otimes 1)\neq 0\quad \beta_1=\lambda_1s\otimes 1 -1\otimes s+\kappa_1\otimes1,$$   $$\beta_2=(2\lambda_1-1)s\otimes 1 -1\otimes s+2\kappa_1\otimes 1,\qquad \eta=a\otimes 1,$$ where $\lambda_1, \kappa_1, w_{01}, a\in{\bf k}$ satisfying $a\lambda_1=a\kappa_1=0$.

If $m=2$ and $w_{12}=0$, then $\alpha'_m=\alpha_2''=w_{02}((1\otimes s^{(2)}-s^{(2)}\otimes1)-\kappa_1(1\otimes s-s\otimes1))$, $\lambda_1=\frac12$ by (T8). Moreover, $\lambda_2=-1$ and $\kappa_2=2\kappa_1$ by Lemma \ref{lem24a}. Thus one obtains a $2$-type Schr\"odinger-Virasoro Lie conformal algebra (C) with

$$\alpha'_m=\alpha_2''=w_{02}((1\otimes s^{(2)}-s^{(2)}\otimes1)-\kappa_1(1\otimes s-s\otimes1))\neq0,\quad \eta=0, $$ $$\beta_1=\frac12s\otimes1-1\otimes s+\kappa_1\otimes1,\quad
\beta_2=-s\otimes1-1\otimes s+2\kappa_1\otimes1,$$  where $ w_{02},\kappa_1\in{\bf k}$.

If $m=2$ and $w_{12}\neq 0$, then $\lambda_1=0$, $\lambda_2=-3$, $w_{02}=-\kappa_1w_{12}$ and $w_{01}=-\frac12\kappa_1w_{02}=\frac 12\kappa_1^2w_{12}$ by  (\ref{eq24a}).
Hence $\alpha_m'=\alpha_2'=-w_{12}(\kappa_1(1\otimes s^{(2)}-s^{(2)}\otimes1)-(s\otimes s^{(2)}-s^{(2)}\otimes s)-\frac{\kappa_1^2}2(1\otimes s-s\otimes1))$ from (T10).  Note that $w_{12}\neq 0$ and $\eta=0$ by Lemma \ref{lem24}. Thus one obtains a $2$-type Schr\"odinger-Virasoro Lie conformal algebra (D) with
$$\alpha_m'=-w_{12}(\kappa_1(1\otimes s^{(2)}-s^{(2)}\otimes1)-(s\otimes s^{(2)}-s^{(2)}\otimes s)-\frac{\kappa_1^2}2(1\otimes s-s\otimes1))\neq0,$$  $$\eta=0,\quad \beta_1=-1\otimes s+\kappa_1\otimes1\quad \beta_2=-3s\otimes 1-1\otimes s+2\kappa_1\otimes1$$ for any $ w_{12}, \kappa_1\in{\bf k}$.

If $m=3$,  then $\alpha_m'=\alpha_3'=w_{03}[(1\otimes s^{(3)}- s^{(3)}\otimes 1)+\frac12(s\otimes s^{(2)}-s^{(2)}\otimes s)-\frac{3\kappa_1}2(1\otimes s^{(2)}-s^{(2)}\otimes 1)+\frac34\kappa_1^2(1\otimes s-s\otimes 1)]$ for some nonzero $w_{03}\in{\bf k}$ from (T9). We get a  $3$-type Schr\"odinger-Virasoro Lie conformal algebra (E) with
$$\alpha_m'=w_{03}((1\otimes s^{(3)}- s^{(3)}\otimes 1)+\frac12(s\otimes s^{(2)}-s^{(2)}\otimes s)-\frac{3\kappa_1}2(1\otimes s^{(2)}-s^{(2)}\otimes 1)+\frac34\kappa_1^2(1\otimes s-s\otimes 1)),$$ $$\eta=0,\quad \beta_1=\frac 23s\otimes 1-1\otimes s+\kappa_1\otimes1,\quad \beta_2=-\frac53s\otimes1-1\otimes s+2\kappa_1\otimes 1$$ for any $0\neq w_{03}, \kappa_1\in{\bf k}$.

Summing up, we have the following theorem.
\begin{theorem} \label{thm31} There are only five kinds of $m$-type Schr\"odinger-Virasoro Lie conformal algebras from (A) to (E) described as above.\end{theorem}
\begin{remark} If ${\bf k}$ is algebracially closed, then we can assume that
$\alpha_1'=1\otimes s-s\otimes 1$ in  (B), $\alpha_2''=(1\otimes s^{(2)}-s^{(2)}\otimes1)-\kappa_1(1\otimes s-s\otimes 1)$ in (C), $\alpha_2'=\kappa_1(1\otimes s^{(2)}-s^{(2)}\otimes1)-(s\otimes s^{(2)}-s^{(2)}\otimes s)-\frac{\kappa_1^2}2(1\otimes s-s\otimes 1)$ in (D) and  $\alpha_3'=1\otimes s^{(3)}-s^{(3)}\otimes1 -\frac32\kappa_1(1\otimes s^{(2)}-s^{(2)}\otimes1)+\frac12(s\otimes s^{(2)}-s^{(2)}\otimes1)+\frac43\kappa_1^2(1\otimes s-s\otimes 1)$ in (E) by letting $e_1'=\frac1{\sqrt{w_{01}}}e_1$, $e_1'=\frac1{\sqrt{-w_{12}}}e_1$, $e_1'=\frac1{\sqrt{w_{02}}}e_1$, and $e_1'=\frac1{\sqrt{w_{03}}}e_1$ respectively.
\end{remark}
\begin{example}\label{ex33} Let $A={\bf k}[[t, t^{-1}]]$ be an $H$-bimodule given by $sf=fs=\frac{df}{dt}$ for any $f\in A$. Then $A$ is an $H$-differential algebra both for the left and the right action of $H$. Suppose $\mathscr{A}_A(L_s)=A\otimes_HL_s$. Then $\mathscr{A}_A(L_s)=A\otimes_HL_s$ is a Lie algebra with bracket $[f\otimes_Ha, g\otimes _Hb]=\sum\limits_{i=0}^3(fh_{1i})(gh_{2i})\otimes_He_i,$ where $[a,b]=\sum\limits_{i=0}^3h_{1i}\otimes h_{2i}\otimes_He_i\in H^{\otimes 2}\otimes_HL_s$. Let $L_n=t^{n+1}\otimes_He_0$, $Y_{p+\rho}=t^{p+1}\otimes_H e_1$ and $M_{k+2\ \rho}=t^{k+1}\otimes_He_2$ for any $n, p, k\in \mathbb{Z}$ and $\rho\in {\bf k}$. Suppose that $SV_{\rho}$ is a vector space with a basis $\{L_n, Y_p, M_k|n\in\mathbb{Z}, p\in \rho+\mathbb{Z}, k\in 2\rho+ \mathbb{Z}\}$.

If  $L_s$ is of type (B) with $\eta=0$ and $w_{01}=1$, then $SV_{\rho}$ is a Lie algebra with nonzero brackets given by
 \begin{eqnarray*} & [L_n, L_{n'}]=(n-n')L_{n+n'},\qquad [Y_p, Y_{p'}]=(p-p')M_{p+p'},\\
& [L_n, Y_p]=\left(\lambda_1(n+1)-p+\rho-1\right)Y_{n+p}+\kappa_1Y_{n+p+1}, \\
 & [L_n, M_k]=\left[(2\lambda_1-1)(n+1)-k+2\rho-1\right]M_{n+k}+2\kappa_1M_{n+k+1}.\end{eqnarray*}
From these,  the Schr\"odinger-Virasoro Lie algebra defined in \cite{U} is the algebra $SV_{\rho}$ in the case when $\lambda_1=\rho=\frac12$ and $\kappa_1=0$. The algebra $W(\varrho)[0]$ introduced in \cite{L} is isomorphic to the subalgebra of $SV_{\rho}$ generated by $\{ L_n,Y_p|n,p\in\mathbb{Z}\}$, where $\kappa_1=0$, $\rho=\frac12$ and $\varrho=2\lambda_1-1$.

If  $L_s$ is of type (C) with $w_{02}=1$, then $SV_{\rho}$ is a Lie algebra with  nonzero brackets
\begin{eqnarray*} & [L_n,L_{n'}]=(n-n')L_{n+n'},\qquad [L_n, Y_p]=(\frac{n-1}2-p+\rho)Y_{n+p}+\kappa_1Y_{n+p+1},\\
& [L_n, M_k]=(2\rho-2-n-k)M_{n+k}+2\kappa_1M_{n+k+1}, \ \  \ \mbox{and} \\
& [Y_p, Y_{p'}]=(p-p')\left[\kappa_1M_{p+p'}-\frac12(p+p'-2\rho+1)M_{p+p'-1}\right].\end{eqnarray*}

If  $L_s$ is of type (D) with $\eta=0$ and $w_{12}=1$, then $SV_{\rho}$ is a Lie algebra with  nonzero brackets given by
\begin{eqnarray*} & [L_n, L_{n'}]=(n-n')L_{n+n'},\qquad [L_n, Y_p]=(\rho-1-p)Y_{n+p}+\kappa_1Y_{n+p+1},\\
& [L_n, M_k]=(2\rho-4-3n-k)M_{n+k}+2\kappa_1M_{n+k+1},\end{eqnarray*} and
\begin{eqnarray*}[Y_p,Y_{p'}]&=&(p-p')\frac{\kappa_1}2(p+p'+1-2\rho)M_{p+p'-1}\\
& &-(p-p')\left[\frac12(p+1-\rho)(p'+1-\rho)M_{p+p'-2}+\frac{\kappa_1^2}2M_{p+p'}\right].\end{eqnarray*}

If  $L_s$ is of type (E) with $w_{03}=1$, then $SV_{\rho}$ is a Lie algebra with  nonzero brackets \begin{eqnarray*} & [L_n, L_{n'}]=(n-n')L_{n+n'},\qquad [L_n, Y_p]=(\frac{2n-1}3-p+\rho)Y_{n+p}+\kappa_1Y_{n+p+1},\\
& [L_n, M_k]=(2\rho-\frac{5n+8}3-k)M_{n+k}+2\kappa_1M_{n+k+1},\end{eqnarray*}  and
\begin{eqnarray*}[Y_p, Y_{p'}]&=&(p'-p)\left[\frac34\kappa_1^2M_{p+p'}-\frac32\kappa_1(p+p'-2\rho+1)M_{p+p'-1}\right]\\
& &+\frac{p'-p}{2}\left[2p^2+3pp'+2p'^2+(1-7\rho)(p+p')+11\rho^2-2\rho+1\right]M_{p+p'-2}.\end{eqnarray*}
\end{example}

\section{Schr\"odinger-Virasoro Lie $H$-pseudoalgebras}

In this section, we determine all Schr\"odinger-Virasoro Lie $H$-pseudoalgebras.
First of all, we describe the Schr\"odinger-Virasoro Lie $H$-pseudoalgebras satisfying  $\eta_{ij}=0$ for all $1\leq i,j\leq 2$.
\begin{proposition} \label{prop41} Let $L$ be a Schr\"odinger-Virasoro Lie $H$-pseudoalgebra with pseudobrackets given by (I*) satisfying $\eta_{ij}=0$ for all $1\leq i, j\leq 2$. Then $L$ must be one of the following algebras (Z1)-(Z5).

(Z1) $L$ is a Schr\"odinger-Virasoro Lie $H$-pseudoalgebra with $\alpha_m'=0$, $\eta=a\otimes 1$, $\beta_i=\lambda_i s\otimes1-1\otimes s+\kappa_i\otimes 1$ for $1\leq i\leq 3$, where $a,\lambda_1,\lambda_2,\lambda_3,\kappa_1,\kappa_2,\kappa_3,
\in{\bf k}$ satisfying $a\lambda_1=a\kappa_1=0$.

(Z2) $L$ is a Schr\"odinger-Virasoro Lie $H$-pseudoalgebra with   $\alpha_m'=w_{01}(1\otimes s-s\otimes 1)\neq 0$, $\eta=a\otimes 1$, $\beta_1=\lambda_1 s\otimes1-1\otimes s+\kappa_1\otimes 1$,
$\beta_2=(2\lambda_1-1) s\otimes1-1\otimes s+2\kappa_1\otimes 1$, $\beta_3=\lambda_3 s\otimes1-1\otimes s+\kappa_3\otimes 1$, where $ w_{01}, a,\lambda_1, \lambda_3,\kappa_1,\kappa_3
\in{\bf k}$ satisfying $a\lambda_1=a\kappa_1=0$.

(Z3) $L$ is a Schr\"odinger-Virasoro Lie $H$-pseudoalgebra with   $\alpha_m'=-w_{12}(\kappa_1(1\otimes s^{(2)}-s^{(2)}\otimes1)-(s\otimes s^{(2)}-s^{(2)}\otimes s)-\frac{\kappa_1^2}2(1\otimes s-s\otimes1))\neq0,$ $\eta=0$,
$\beta_1=-1\otimes s+\kappa_1\otimes 1$, $\beta_2=-3s\otimes 1-1\otimes s+2\kappa_1\otimes1$, $\beta_3=\lambda_3s\otimes1-1\otimes s+\kappa_3\otimes 1$, where $w_{12},\lambda_3, \kappa_1,\kappa_3\in{\bf k}$.

(Z4) $L$ is a Schr\"odinger-Virasoro Lie $H$-pseudoalgebra with   $\alpha_m'=\alpha_2''=w_{02}((1\otimes s^{(2)}-s^{(2)}\otimes1)-\kappa_1(1\otimes s-s\otimes1))\neq0$, $\eta=0$, $\beta_1=\frac12s\otimes 1-1\otimes s+\kappa_1\otimes1$, $\beta_2=-s\otimes 1-1\otimes s+2\kappa_1\otimes1$, $\beta_3=\lambda_3s\otimes 1-1\otimes s+\kappa_3\otimes1$, where $w_{02},\lambda_3,\kappa_1,\kappa_3\in{\bf k}$.

(Z5)  $L$ is a Schr\"odinger-Virasoro Lie $H$-pseudoalgebra with   $\alpha_m'=w_{03}((1\otimes s^{(3)}- s^{(3)}\otimes 1)+\frac12(s\otimes s^{(2)}-s^{(2)}\otimes s)-\frac{3\kappa_1}2(1\otimes s^{(2)}-s^{(2)}\otimes 1)+\frac34\kappa_1^2(1\otimes s-s\otimes 1))\neq 0$, $\eta=0$,
$\beta_1=\frac23s\otimes1-1\otimes s+\kappa_1\otimes 1$, $\beta_2=-\frac53s\otimes 1-1\otimes s+2\kappa_1\otimes 1$, $\beta_3=\frac23s\otimes1-1\otimes s+\kappa_3\otimes 1$, where  $w_{03},\kappa_1, \kappa_3\in{\bf k}$.
\end{proposition}

\begin{proof} Since $\eta_{ij}=0$ for all $1\leq i, j\leq n$, $L$ is a Schr\"odinger-Virasoro Lie $H$-pseudoalgebra if and only if (\ref{eq217})-(\ref{eq220}) hold by Lemma \ref{lem28}.  Lemma \ref{lem24} tells us  that $\eta=a\otimes 1$ for some $a\in {\bf k}$ with $a\lambda_1=a\kappa_1=0$ and $(\kappa_2-2\kappa_1)\alpha_m'=0$ if  (\ref{eq217}), (\ref{eq219}) and (\ref{eq220}) hold. Thus, $L$ is a Schr\"odinger-Virasoro Lie $H$-pseudoalgebra if and only if  (\ref{eq218}) is true. By Theorem \ref{thm31}, $L$ is the algebra described by one of the  cases (Z1)-(Z5).\end{proof}

 Let us denote $e_0: =-L$, $e_1: =T(X)$, $e_2: =T^-(X)$ and $e_3: =U$, where $L, $ $T(X)$, $T^-(X)$ are given in \cite{CP}. Then the subalgebra $He_0\oplus He_1\oplus He_2\oplus He_3$ of the large $N=4$ conformal superalgebras in \cite{CP} is the Schr\"odinger-Virasoro Lie algebra described by (Z1) in Proposition \ref{prop41}, where $a=0$ and $\beta_i =-1\otimes s $ for $1\leq i\leq 3$.

In the remainder of this section, we always assume that  $\eta_{ij}$, for $1\leq i, j\leq 2$,  are not all zero.
\begin{lemma}\label{lem41a}Suppose $\eta_{11}=\sum\limits_{i=0}^{m_1}\sum\limits_{j=0}^{p_i}a_{ij}s^{(i)}\otimes s^{(j)}\neq 0$ for some $a_{ij}\in{\bf k}$ satisfying   (\ref{eq28}). Then $\kappa_3=0$ and  $\lambda_3\in \{0, 1, 2, 3\}$.  Further, $\eta_{11}$ must be one of the following types.

(a1) $\eta_{11}=a_{00}\otimes 1$  if $\lambda_3=0$.

(a2) $\eta_{11}=a_{10}(-\kappa_1\otimes 1+s\otimes 1-\lambda_1\otimes s)$ for any $\kappa_1$, $\lambda_1$ if $\lambda_3=1$.

(a3) $\eta_{11}=\frac{a_{20}}{4}(2\kappa_1^2\otimes1-\kappa_1\otimes s
-4\kappa_1s\otimes1+s\otimes s+4s^{(2)}\otimes1)$ if $(\lambda_1,\lambda_3)=(0,2)$, or $\eta_{11}=a_{20}(\frac{\kappa_1^2}2\otimes1-\kappa_1s\otimes1-\frac{3\kappa_1}{4}\otimes s+\frac{1}{2}\otimes s^{(2)} +\frac{3}{4}s\otimes s+s^{(2)}\otimes1)$ if  $(\lambda_1,\lambda_3)=(-1, 2)$.

(a4) $\eta_{11}=a_{21}((s^{(2)}\otimes s-s\otimes s^{(2)})+\kappa_1\otimes s^{(2)}
+\frac12\kappa_1^2 \otimes s-\kappa_1s\otimes s)$  if $\lambda_3=3$. In this case, $\lambda_1=0$.
\end{lemma}
\begin{proof} Note that $\kappa_3=0$ by Lemma \ref{lem25a} since $\eta_{11}\neq 0$.
We replace  $\gamma$ with $\eta_{11}$ and $\beta_2$ with $\beta_1$ in (\ref{eq21a}) as before. Then (\ref{eq21a}) becomes (\ref{eq28}). Since $\eta_{11}$ satisfies (\ref{eq28}), it is easy to obtain (a1)-(a4) from (T1)-(T10) in Section 3 by using $a_{ij}$ instead of $x_{ij}$.
\end{proof}

In the same way, one can obtain the following result by (\ref{eq211}).
\begin{lemma}\label{lem42a}Suppose $\eta_{22}=\sum\limits_{i=0}^{m_4}\sum\limits_{j=0}^{v_i}d_{ij}s^{(i)}\otimes s^{(j)}\neq 0$ for some $d_{ij}\in{\bf k}$ satisfying (\ref{eq211}). Then $\kappa_3=0$ and  $\lambda_3\in \{0, 1, 2, 3\}$. Further, $\eta_{22}$ must be one of the following types.

(d1)  $\eta_{22}=d_{00}\otimes 1$  if $\lambda_3=0$.

(d2)  $\eta_{22}=d_{10}(-\kappa_2\otimes 1+s\otimes 1-\lambda_2\otimes s)$ for any $\kappa_2, \lambda_2\in{\bf k}$ if $\lambda_3=1$.

(d3)  $\eta_{22}=\frac{d_{20}}{4}(2\kappa_2^2\otimes1-\kappa_2\otimes s
-4\kappa_2s\otimes1+s\otimes s+4s^{(2)}\otimes1)$ if $(\lambda_2,\lambda_3)=(0,2)$, or $\eta_{22}=d_{20}(\frac{\kappa_2^2}2\otimes1-\kappa_2s\otimes1-\frac{3\kappa_2}{4}\otimes s+\frac{1}{2}\otimes s^{(2)}
+\frac{3}{4}s\otimes s+s^{(2)}\otimes1)$ for any $\kappa_2\in{\bf k}$ if $(\lambda_2,\lambda_3)=(-1,2)$.

(d4) $\eta_{22}=d_{21}((s^{(2)}\otimes s-s\otimes s^{(2)})+\kappa_2\otimes s^{(2)}
+\frac12\kappa_2^2s\otimes1-\kappa_2s\otimes s)$ for  any $\kappa_2$ if $\lambda_3=3$. In this case, $\lambda_2=0$.
\end{lemma}

In the next four lemmas, we describe all  $\eta_{ij}$ satisfying (\ref{eq212})-(\ref{eq215}) for  $1\leq i, j\leq 2$.
With notations in Lemmas \ref{lem41a}-\ref{lem42a}, we assume in addition that $\eta_{12}=\sum\limits_{i=0}^{m_2}\sum\limits_{j=0}^{q_i}b_{ij}s^{(i)}\otimes s^{(j)}$ and  $\eta_{21}=\sum\limits_{i=0}^{m_3}\sum\limits_{j=0}^{u_i}c_{ij}s^{(i)}\otimes s^{(j)}$ for some $b_{ij},c_{ij}\in {\bf k}$.
\begin{lemma}\label{lem43a} Suppose $\eta_{ij}$ ($1\leq i,j\leq 2$) satisfy (\ref{eq28})-(\ref{eq215}) and $\eta_{12}=0$. Then $\eta_{11}=a_{00}\otimes 1$, $\eta_{22}=d_{00}\otimes 1$.
When $\eta_{21}=0$,  we have either $\eta_{11}\neq 0$ or $\eta_{22}\neq 0$,  and $\lambda_3=\kappa_3=0$.
When $\eta_{21}\neq 0$, one of the following cases occurs.

(i) If $\eta_{11}=\eta_{22}=0$, then $\eta_{21}$ is described by (T1)-(T10) in Section 3, where $x_{ij}$ are replaced by $c_{ij}$,  and $\lambda_1$, $\lambda_2$ are exchanged. Moreover, $\kappa_2-\kappa_1+\kappa_3=0$.

(ii) If $\eta_{11}=\eta_{22}\neq 0$,  then $\eta_{21}=c_{00}\otimes 1$ for some nonzero $c_{00}$. Moreover, $\kappa_1=\kappa_2$ and $\lambda_1-\lambda_2=\lambda_3=0$.

Suppose $\eta_{11}\neq \eta_{22}$. Then $m_3\leq 1$. Further,

(iii) if $m_3=0$,  then $\eta_{21}=c_{00}\otimes 1$ for some nonzero $c_{00}\in{\bf k}$. Moreover,  $\kappa_1=\kappa_2$ and $\lambda_1=\lambda_2$.

(iv) if $m_3=1$,  then  $$\eta_{21}=c_{10}\left(-\kappa_1\otimes 1+s\otimes 1+\frac{d_{00}}{d_{00}-a_{00}}\otimes s\right)$$ for some nonzero $ c_{10}\in{\bf k}$. Moreover, $\kappa_1=\kappa_2$ and $\lambda_1+1=\lambda_2=\lambda_3=0$.
\end{lemma}

\begin{proof}Since $\eta_{12}=0$, $(\eta_{11}\Delta\otimes 1)\eta_{11}=(12)(1\otimes \eta_{11}\Delta)((12)\eta_{11})$ by (\ref{eq212}). Note that $\eta_{11}=\sum\limits_{i=0}^{m_1}\sum\limits_{j=0}^{p_i}a_{ij}s^{(i)}\otimes s^{(j)}$ by the assumption. The above equation is translated into
\begin{eqnarray}&&\sum\limits_{i=0,u=0}^{m_1,m_1}\sum\limits_{j=0,v=0}^{p_i,p_u}\sum\limits_{t=0}^ua_{ij}a_{uv}s^{(i)}s^{(t)}\otimes s^{(j)}s^{(u-t)}\otimes s^{(v)}\label{eqa41}\\
&&=\sum\limits_{i=0,u=0}^{m_1,m_1}\sum\limits_{j=0,v=0}^{p_i,p_u}\sum\limits_{t=0}^ua_{ij}a_{uv}s^{(i)}s^{(t)}\otimes s^{(v)}\otimes s^{(j)}s^{(u-t)}. \nonumber\end{eqnarray} From this, we get $\max\{p_0,p_1, \cdots, p_{m_1}\}=m_1+\max\{p_0,p_1, \cdots, p_{m_1}\}$. Thus $m_1=0$ and $\eta_{11}=a_{00}\otimes 1$  by Lemma \ref{lem41a}. Similarly, we get $\eta_{22}=d_{00}\otimes 1$ for some
$d_{00}\in{\bf k}$ by (\ref{eq215}).

If $a_{00}=d_{00}=0$, then  (\ref{eq28})-(\ref{eq211}) hold for any $\eta_{21}$. By the assumption that $\eta_{ij}$ ($1\leq i,j\leq 2$) are not all zero, the nonzero $\eta_{21}$ is given by $\gamma $ in (T1)-(T10) in Section 4, where $x_{ij}$ are replaced by $c_{ij}$ and $\lambda_1$, $\lambda_2$ are exchanged. Since $\eta_{21}\neq 0$, we have $\kappa_2-\kappa_1+\kappa_3=0$ by Lemma \ref{lem24a}.

If $\eta_{21}=\eta_{12}=0$, then  (\ref{eq28})-(\ref{eq211}) hold for any $\eta_{11}=a_{00}\otimes 1$ and $\eta_{22}=d_{00}\otimes 1$. By the assumption that  $\eta_{ij}$ are not all zero,  we have either  $a_{00}\neq 0$ or $d_{00}\neq 0$.

Next, we assume that $\eta_{21}\neq 0$. Then, by (\ref{eq214}),  $$a_{00}(\eta_{21}\otimes 1)+d_{00}(\Delta\otimes 1)\eta_{21}=a_{00}(12)(1\otimes\eta_{21})+d_{00}(12)(1\otimes \Delta)((12)\eta_{21}).$$  Applying the functor $1\otimes1\otimes \varepsilon$ to this equation yields $$(a_{00}-d_{00})\eta_{21}=a_{00}\sum\limits_{i=0}^{m_3}c_{i0}s^{(i)}\otimes 1-d_{00}\sum\limits_{i=0}^{m_3}\sum\limits_{t=0}^ic_{i0}s^{(t)}\otimes s^{(i-t)}.$$
If  $a_{00}\neq d_{00}$,   $\eta_{21}=\frac{d_{00}}{d_{00}-a_{00}}\sum\limits_{i=1}^{m_3}\sum\limits_{t=0}^{i-1}c_{i0}s^{(t)}\otimes s^{(i-t)}+\sum\limits_{i=0}^{m_3}c_{i0}s^{(i)}\otimes 1$. From (T1)-(T10) in Section 3, we get $m_3\leq 1$. If $m_3=0$, then $\eta_{21}=c_{00}\otimes 1$ and $\lambda_2-\lambda_1+\lambda_3=0$. Since $a_{00}\neq d_{00}$, $\lambda_3=0$ by Lemma \ref{lem41a} and Lemma \ref{lem42a}. Thus $\lambda_1=\lambda_2$. Similar to the previous case, we can prove that $\kappa_1=\kappa_2$. If $m_3=1$, then $\eta_{21}=c_{10}(\frac{\kappa_1a_{00}-(\kappa_1+\kappa_3)d_{00}}{d_{00}-a_{00}}\otimes 1+s\otimes 1+\frac{d_{00}}{d_{00}-a_{00}}\otimes s)$ with $\lambda_2a_{00}=(\lambda_2+\lambda_3)d_{00}$ and $\lambda_2-\lambda_1+\lambda_3=1$, where $c_{10}\neq 0$. Similar to the case when $m=0$, we can prove that $\lambda_3=0$ and $\kappa_1=\kappa_2$. Then $\lambda_2a_{00}=\lambda_2d_{00}$. Since $a_{00}\neq d_{00}$,  $\lambda_2=0$. Thus  $\lambda_1=-1$ by (\ref{eq24a}).
\end{proof}
From Lemma \ref{lem41a}-\ref{lem43a}, we can get all the Schr\"odinger-Virasoro Lie $H$-pseudoalgebras with $\eta=\alpha'_m=\eta_{12}=0$.
\begin{theorem} \label{thm51} Let $L$ be a Sch\"odinger-Virasoro Lie $H$-pseudoalgebra with pseudobracket given by
$$\left\{\begin{array}{l}[e_0,e_0]=\alpha\otimes_He_0, \quad \ [e_0, e_i]=\beta_{i}\otimes_H e_i\quad 1\leq i\leq 3,\\

[e_1, e_1]=[e_1, e_2]=[e_2, e_2]=[e_3,e_3]=0,\\

 [e_1, e_3]=\eta_{11}\otimes_He_1,\\

[e_2, e_3]=\eta_{21}\otimes_He_1+\eta_{22}\otimes_He_2,\end{array}\right.$$
Suppose $\eta_{ij}$ are not all zero. Then $\eta_{ij}$ and $\beta_i$ are  described by one of the following cases:

(A1)\  $\eta_{11}=\eta_{22}=0$, $\eta_{21}=c_{01}(-\kappa_3\otimes 1+1\otimes s)\neq 0$, $\beta_1=\lambda_1s\otimes 1-1\otimes s+\kappa_1\otimes 1$, $\beta_2=(\lambda_1+1)s\otimes 1-1\otimes s+(\kappa_1-\kappa_3)\otimes 1$, $\beta_3=-1\otimes s+\kappa_3\otimes 1$,
where $c_{01},\lambda_1,\kappa_1,\kappa_3\in{\bf k}$.

(A2)\  $\eta_{11}=\eta_{22}=0$, $\eta_{21}=c_{00}\otimes 1\neq 0$, $\beta_1=\lambda_1s\otimes 1-1\otimes s+\kappa_1\otimes 1$, $\beta_2=\lambda_2s\otimes 1-1\otimes s+\kappa_2\otimes 1$, $\beta_3=(\lambda_1-\lambda_2)s\otimes 1-1\otimes s+(\kappa_1-\kappa_2)\otimes 1$,
where $c_{00},\lambda_1,\lambda_2,\kappa_1,\kappa_2\in{\bf k}$.

(A3)\  $\eta_{11}=\eta_{22}=0$, $\eta_{21}=c_{11}((\kappa_2(\kappa_1-\kappa_2)+(\lambda_1+2)(\kappa_1-\kappa_2)^2)\otimes1-(\kappa_2+2(\lambda_1+2)(\kappa_1-\kappa_3))\otimes s+2(\lambda_1+2)\otimes s^{(2)}-(\kappa_1-\kappa_2)s\otimes1+s\otimes s)\neq 0$, $\beta_1=\lambda_1s\otimes 1-1\otimes s+\kappa_1\otimes 1$, $\beta_2=(\lambda_1+2)s\otimes 1-1\otimes s+\kappa_2\otimes 1$, $\beta_3=-1\otimes s+(\kappa_1-\kappa_2)\otimes 1$,
where $c_{11},\lambda_1,\kappa_1,\kappa_2\in{\bf k}$.

(A4)\  $\eta_{11}=\eta_{22}=0$, $\eta_{21}=c_{11}(\kappa_2(\kappa_1-\kappa_2)\otimes1-(\kappa_1-\kappa_2)s\otimes1-\kappa_2\otimes s+s\otimes s)\neq 0$, $\beta_1=\lambda_1s\otimes 1-1\otimes s+\kappa_1\otimes 1$, $\beta_2=-2s\otimes 1-1\otimes s+\kappa_2\otimes 1$, $\beta_3=-1\otimes s+(\kappa_1-\kappa_2)\otimes 1$,
where $c_{11},\lambda_1,\kappa_1,\kappa_2\in{\bf k}$.

(A5)\  $\eta_{11}=\eta_{22}=0$, $\eta_{21}=c_{11}(-(\kappa_1-\kappa_2)s\otimes 1+s\otimes s)\neq 0$, $\beta_1=\lambda_1s\otimes 1-1\otimes s+\kappa_1\otimes 1$, $\beta_2=(\lambda_1+2)s\otimes 1-1\otimes s+\kappa_2\otimes 1$, $\beta_3=-1\otimes s+(\kappa_1-\kappa_2)\otimes 1$,
where $c_{11},\lambda_1,\kappa_1,\kappa_2\in{\bf k}$.

(A6)\  $\eta_{11}=\eta_{22}=0$, $\eta_{21}=-(\kappa_2c_{10}+(\kappa_1-\kappa_2)c_{01})\otimes 1+c_{10}s\otimes 1+c_{01}\otimes s$, $\beta_1=\lambda_1s\otimes 1-1\otimes s+\kappa_1\otimes 1$, $\beta_2=\lambda_2s\otimes 1-1\otimes s+\kappa_2\otimes 1$, $\beta_3=(\lambda_1-\lambda_2+1)s\otimes 1-1\otimes s+(\kappa_1-\kappa_2)\otimes 1$,
where $c_{11},\lambda_1,\kappa_1,\kappa_2\in{\bf k}$ satisfying $\lambda_2c_{10}+(\lambda_1-\lambda_2+1)c_{01}=0$ and $c_{10}\neq 0$.

(A7) \  $\eta_{11}=\eta_{22}=0$, $\eta_{21}=c_{10}(-\kappa_2\otimes 1+s\otimes 1)\neq 0$, $\beta_1=\lambda_1s\otimes 1-1\otimes s+\kappa_1\otimes 1$, $\beta_2=\lambda_2s\otimes 1-1\otimes s+\kappa_2\otimes 1$, $\beta_3=(\lambda_1+1)s\otimes 1-1\otimes s+(\kappa_1-\kappa_2)\otimes 1$,
where $c_{10},\lambda_1,\lambda_2, \kappa_1,\kappa_2\in{\bf k}$.

(A8)  \  $\eta_{11}=\eta_{22}=0$, $\eta_{21}=\frac{c_{20}}{2(\lambda_1+2)}(\kappa_2((\lambda_1+2)\kappa_2+(\kappa_1-\kappa_2))\otimes1-\kappa_2\otimes s
-(2(\lambda_1+2)\kappa_2+(\kappa_1-\kappa_2))s\otimes1+s\otimes s+2(\lambda_1+2)s^{(2)}\otimes1)\neq 0$, $\beta_1=\lambda_1s\otimes 1-1\otimes s+\kappa_1\otimes 1$, $\beta_2=-1\otimes s+\kappa_2\otimes 1$, $\beta_3=(\lambda_1+2)s\otimes 1-1\otimes s+(\kappa_1-\kappa_2)\otimes 1$,
where $c_{20},\lambda_1, \kappa_1,\kappa_2\in{\bf k}$.

(A9) \  $\beta_1=-s\otimes 1-1\otimes s+\kappa_1\otimes 1$, $\beta_2=\lambda_2s\otimes 1-1\otimes s+\kappa_2\otimes 1$, $\beta_3=(1-\lambda_2)s\otimes 1-1\otimes s+(\kappa_1-\kappa_2)\otimes 1$,
$\eta_{11}=\eta_{22}=0$,

\begin{eqnarray*}\begin{array}{lll}\eta_{21}&=&b_{20}((\frac{\lambda_2}{2\lambda_2-2}(\kappa_1-\kappa_2)^2
 +\frac{1-2\lambda_2}{2-2\lambda_2}\kappa_2(\kappa_1-\kappa_2)+\frac12\kappa_2^2)\otimes1\\
&&+(\frac{(2\lambda_2-1)(\kappa_1-\kappa_2)}{2-2\lambda_2}
-\kappa_2)s\otimes1
+\frac{(2\lambda_2-1)\kappa_2+2\lambda_2(\kappa_1-\kappa_2)}{2-2\lambda_2}\otimes s-\frac{\lambda_2}{1-\lambda_2}\otimes s^{(2)}\\  &&
+\frac{1-2\lambda_2}{2-2\lambda_2}s\otimes s+s^{(2)}\otimes1)\neq 0,\end{array}\end{eqnarray*} where $c_{20},\lambda_2, \kappa_1,\kappa_2\in{\bf k}$.

(A10)\  $\beta_1=s\otimes 1-1\otimes s+\kappa_1\otimes 1$, $\beta_2=-1\otimes s+\kappa_2\otimes 1$, $\beta_3=2s\otimes 1-1\otimes s+(\kappa_1-\kappa_2)\otimes 1$,
 $\eta_{11}=\eta_{22}=0$,\begin{eqnarray*}\begin{array}{lll}\eta_{21}&=&c_{30}(s^{(3)}\otimes 1-\frac12(\kappa_1+\kappa_2))s^{(2)}\otimes1+\frac12s^{(2)}\otimes s+\frac1{12}(\kappa_1^2+4\kappa_1\kappa_2+\kappa_2^2)s\otimes1\\
&&-\frac16(2\kappa_2+\kappa_1)s\otimes s+\frac16s\otimes s^{(2)}
-\frac1{12}(\kappa_1\kappa_2^2+\kappa_1^2\kappa_2)\otimes1\\&&
+\frac1{12}(\kappa_2^2+2\kappa_1\kappa_2)\otimes s
-\frac16\kappa_2\otimes s^{(2)})\neq 0,\end{array}\end{eqnarray*} where $c_{20},\lambda_2, \kappa_1,\kappa_2\in{\bf k}$.

(A11) \   $\eta_{11}=\eta_{22}=0$,$$\begin{array}{lll}\eta_{21}&=&c_{30}((s^{(3)}\otimes1-1\otimes s^{(3)})+\frac12(s^{(2)}\otimes s-s\otimes s^{(2)})-\frac12(\kappa_1+\kappa_2)s^{(2)}\otimes 1\\&&+\frac12(2\kappa_1-\kappa_2)\otimes s^{(2)}-\frac14(\kappa_1^2-4\kappa_1\kappa_2+\kappa_2^2)s\otimes1+\frac14(\kappa_2^2+2\kappa_1\kappa_2-2\kappa_1^2)\otimes s\\ &&+ \frac12(\kappa_1-2\kappa_2)s\otimes s +\frac1{12}(2\kappa_2-\kappa_1)(\kappa_2^2-\kappa_1\kappa_2-2\kappa_1^2)\otimes1),\end{array}$$ $\beta_1=\frac23s\otimes 1-1\otimes s+\kappa_1\otimes 1$, $\beta_2=-\frac53s\otimes 1-1\otimes s+\kappa_2\otimes 1$, $\beta_3=\frac23s\otimes 1-1\otimes s+(\kappa_2-\kappa_1)\otimes 1$,
where $c_{30}, \kappa_1,\kappa_2\in{\bf k}$.

(A12) \  $\beta_1=\lambda_1s\otimes 1-1\otimes s+\kappa_1\otimes 1$, $\beta_2=-1\otimes s+\kappa_2\otimes 1$, $\beta_3=(\lambda_1+3)s\otimes 1-1\otimes s+(\kappa_1-\kappa_2)\otimes 1$,  $\eta_{11}=\eta_{22}=0$, $$\begin{array}{lll}\eta_{21}&=&c_{21}((s^{(2)}\otimes s-s\otimes s^{(2)})+(\kappa_2\otimes s^{(2)}-(\kappa_1-\kappa_2)s^{(2)}\otimes1)\\ &&+(\frac12(\kappa_1-\kappa_2)(3\kappa_2-\kappa_1)s\otimes 1+\frac12\kappa_2(3\kappa_1-2\kappa_1)\otimes s)+(\kappa_1-2\kappa_2)s\otimes s\\  &&+\frac12\kappa_2(\kappa_1-\kappa_2)(\kappa_1-2\kappa_2)\otimes 1)\neq 0,\end{array}$$
where $c_{21}, \lambda_1, \kappa_1,\kappa_2\in{\bf k}$.

(A13) \   $\beta_1=-s\otimes 1-1\otimes s+\kappa_1\otimes 1$, $\beta_2=2s\otimes1 -1\otimes s+\kappa_2\otimes 1$, $\beta_3=-1\otimes s+(\kappa_1-\kappa_2)\otimes 1$, $\eta_{11}=\eta_{22}=0$, $$\begin{array}{lll}\eta_{21}&=&c_{21}(s^{(2)}\otimes s+3s\otimes s^{(2)}-(\kappa_1-\kappa_2)s^{(2)}\otimes 1-3(2\kappa_1-\kappa_2)\otimes s^{(2)}\\ &&-(3\kappa_1-2\kappa_2)s\otimes s+\frac12(\kappa_2^2-4\kappa_1\kappa_2+3\kappa_1^2)s\otimes 1+\frac12(\kappa_2^2-6\kappa_1\kappa_2+6\kappa_1^2)\otimes s\\ &&+\frac12(3\kappa_2^3-6\kappa_2^2\kappa_1-3\kappa_2\kappa_1^2-2\kappa_1^3)\otimes 1+6\otimes s^{(3)})\neq 0,\end{array}$$
where $c_{21}, \kappa_1,\kappa_2\in{\bf k}$.

(A14)\  $\eta_{11}=\eta_{22}=a_{00}\otimes 1\neq 0$,   $\eta_{21}=c_{00}\otimes 1\neq 0$,  $\beta_1=\lambda_1s\otimes 1-1\otimes s+\kappa_1\otimes 1=\beta_2$, $\beta_3=-1\otimes s$ for some  $a_{00}, c_{00},\lambda_1,\kappa_1\in{\bf k}$.

(A15) $\eta_{11}=a_{00}\otimes 1\neq \eta_{22}=d_{00}\otimes1$,   $\eta_{21}=c_{00}\otimes 1\neq 0$,  $\beta_1=\beta_2=\lambda_1s\otimes 1-1\otimes s+\kappa_1\otimes 1$, $\beta_3=-1\otimes s$  for some $a_{00}, c_{00}, d_{00}, \lambda_1,\kappa_1\in{\bf k}$.

(A16)\  $\eta_{11}=a_{00}\otimes 1, \eta_{22}=d_{00}\otimes1$,   $$\eta_{21}=c_{10}\left(-\kappa_1\otimes 1+s\otimes 1+\frac{d_{00}}{d_{00}-a_{00}}\otimes s\right)\neq0,$$ $\beta_1=-s\otimes 1-1\otimes s+\kappa_1\otimes 1$, $\beta_2=-1\otimes s+ \kappa_1\otimes 1$ and $\beta_3=-1\otimes s$ for some  $ a_{00},d_{00}, c_{10},\kappa_1\in{\bf k}$.
\end{theorem}
\begin{proof}Since $\eta=\alpha'_m=0$, $L$ is a Scr\"odinger-Virasoro Lie $H$-pseudoalgebra with the pseudobracket given by (I*) if and only if (\ref{eq28})-(\ref{eq215}) hold by Lemma \ref{lem28}.  Thus this theorem follows from Lemma \ref{lem43a}.\end{proof}

Similar to Lemma \ref{lem43a}, we can prove the following result.
\begin{lemma}\label{lem44a}  Suppose $\eta_{ij}$ ($1\leq i,j\leq 2$) satisfy (\ref{eq28})-(\ref{eq215}) and $\eta_{21}=0$. Then $\eta_{11}=a_{00}\otimes 1$, $\eta_{22}=d_{00}\otimes 1$ for some $a_{00},d_{00}\in{\bf k}$.
When $\eta_{12}=0$,  we have either $\eta_{11}\neq 0$ or $\eta_{22}\neq 0$,  and $\lambda_3=\kappa_3=0$.
When $\eta_{12}\neq 0$, one of the following cases occurs.

(i) If $\eta_{11}=\eta_{22}=0$, then $\eta_{12}$  is described by (T1)-(T10) in Section 3, where $x_{ij}$ are replaced by $b_{ij}$. Moreover, $\kappa_1-\kappa_2+\kappa_3=0$.

(ii) If $\eta_{11}=\eta_{22}\neq 0$,  then $\eta_{12}=b_{00}\otimes 1$ for some nonzero $b_{00}\in{\bf k}$. Moreover, $\lambda_1=\lambda_2$ and  $\kappa_1=\kappa_2$.

Suppose $\eta_{11}\neq \eta_{22}$. Then $m_2\leq 1$. Further,

(iii) if  $m_2=0$,  then $\eta_{12}=b_{00}\otimes 1$ for some $b_{00}\in{\bf k}$. Moreover, $\lambda_1=\lambda_2$ and  $\kappa_1=\kappa_2$.

(iv) if $m_2=1$, then  $$\eta_{12}=b_{10}\left(-\kappa_1\otimes 1+s\otimes 1+\frac{d_{00}}{d_{00}-a_{00}}\otimes s\right)$$ for some nonzero $b_{10}\in{\bf k}$. Moreover, $\lambda_2+1=\lambda_1=\lambda_3=0$ and $\kappa_1=\kappa_2$. \end{lemma}
Using Lemma 5.5, we can prove the following result.
\begin{theorem} \label{thm52} Let $L$ be a Sch\"odinger-Virasoro Lie $H$-pseudoalgebra with pseudobracket given by
$$\left\{\begin{array}{l}[e_0,e_0]=\alpha\otimes_He_0, \quad \ [e_0, e_i]=\beta_{i}\otimes_H e_i\quad 1\leq i\leq 3,\\

[e_1, e_1]=[e_1, e_2]=[e_2, e_2]=[e_3,e_3]=0,\\

 [e_1, e_3]=\eta_{11}\otimes_He_1+\eta_{12}\otimes _He_2,\\

[e_2, e_3]=\eta_{22}\otimes_He_2.\end{array}\right.$$
Suppose $\eta_{ij}$ are not all zero.  Then $\eta_{ij}$ and $\beta_i$ are  described by one of the following types:

(B1)\  $\eta_{11}=\eta_{22}=0$, $\eta_{12}=b_{01}(-\kappa_3\otimes 1+1\otimes s)\neq 0$, $\beta_1=\lambda_1s\otimes 1-1\otimes s+\kappa_1\otimes 1$, $\beta_2=(\lambda_1-1)s\otimes 1-1\otimes s+(\kappa_1+\kappa_3)\otimes 1$, $\beta_3=-1\otimes s+\kappa_3\otimes 1$,
where $b_{01},\lambda_1,\kappa_1,\kappa_3\in{\bf k}$.

(B2)\  $\eta_{11}=\eta_{22}=0$, $\eta_{12}=b_{00}\otimes 1\neq 0$, $\beta_1=\lambda_1s\otimes 1-1\otimes s+\kappa_1\otimes 1$, $\beta_2=\lambda_2s\otimes 1-1\otimes s+\kappa_2\otimes 1$, $\beta_3=(\lambda_2-\lambda_1)s\otimes 1-1\otimes s+(\kappa_2-\kappa_1)\otimes 1$,
where $b_{00},\lambda_1,\lambda_2, \kappa_1,\kappa_2\in{\bf k}$.

(B3)\   $\eta_{11}=\eta_{22}=0$, $\eta_{12}=b_{11}((\kappa_1(\kappa_2-\kappa_1)+\lambda_1(\kappa_2-\kappa_1)^2)\otimes1-(\kappa_1+2\lambda_1(\kappa_2-\kappa_1))\otimes s+2\lambda_1\otimes s^{(2)}-(\kappa_2-\kappa_1)s\otimes1+s\otimes s)\neq 0$, $\beta_1=\lambda_1s\otimes 1-1\otimes s+\kappa_1\otimes 1$, $\beta_2=(\lambda_1-2)s\otimes 1-1\otimes s+\kappa_2\otimes 1$, $\beta_3=-1\otimes s+(\kappa_2-\kappa_1)\otimes 1$,
where $b_{11},\lambda_1, \kappa_1,\kappa_2\in{\bf k}$.

(B4)\   $\eta_{11}=\eta_{22}=0$, $\eta_{12}=b_{11}(\kappa_1(\kappa_2-\kappa_1)\otimes1-(\kappa_2-\kappa_1)s\otimes1-\kappa_1\otimes s+s\otimes s)\neq 0$, $\beta_1=-1\otimes s+\kappa_1\otimes 1$, $\beta_2=-2s\otimes 1-1\otimes s+\kappa_2\otimes 1$, $\beta_3=-1\otimes s+(\kappa_2-\kappa_1)\otimes 1$,
where $b_{11}, \kappa_1,\kappa_2\in{\bf k}$.

(B5)\   $\eta_{11}=\eta_{22}=0$, $\eta_{12}=b_{11}(-\kappa_3s\otimes 1+s\otimes s)\neq 0$, $\beta_1=\lambda_1s\otimes 1-1\otimes s+\kappa_1\otimes 1$, $\beta_2=(\lambda_1-2)s\otimes 1-1\otimes s+\kappa_2\otimes 1$, $\beta_3=-1\otimes s+(\kappa_2-\kappa_1)\otimes 1$,
where $b_{11}, \lambda_1, \kappa_1,\kappa_2\in{\bf k}$.

(B6)\  $\eta_{11}=\eta_{22}=0$, $\eta_{12}=-(\kappa_1b_{10}+(\kappa_2-\kappa_1)b_{01})\otimes 1+b_{10}s\otimes 1+b_{01}\otimes s$, $\beta_1=\lambda_1s\otimes 1-1\otimes s+\kappa_1\otimes 1$, $\beta_2=\lambda_2s\otimes 1-1\otimes s+\kappa_2\otimes 1$, $\beta_3=(\lambda_2-\lambda_1+1)s\otimes 1-1\otimes s+(\kappa_2-\kappa_1)\otimes 1$,
where $b_{11}, \lambda_1, \kappa_1,\kappa_2\in{\bf k}$ satisfying $\lambda_1b_{10}+(\lambda_2-\lambda_1+1)b_{01}=0$ and $b_{10}\neq 0$.

(B7)\  $\eta_{11}=\eta_{22}=0$, $\eta_{12}=b_{10}(-\kappa_1\otimes 1+s\otimes 1)\neq 0$, $\beta_1=-1\otimes s+\kappa_1\otimes 1$, $\beta_2=\lambda_2s\otimes 1-1\otimes s+\kappa_2\otimes 1$, $\beta_3=(\lambda_2+1)s\otimes 1-1\otimes s+(\kappa_2-\kappa_1)\otimes 1$,
where $b_{10}, \lambda_1, \lambda_2, \kappa_1,\kappa_2\in{\bf k}$.

(B8)\   $\beta_1=-1\otimes s+\kappa_1\otimes 1$, $\beta_2=\lambda_2s\otimes 1-1\otimes s+\kappa_2\otimes 1$, $\beta_3=(\lambda_2+2)s\otimes 1-1\otimes s+(\kappa_2-\kappa_1)\otimes 1$,
$\eta_{11}=\eta_{22}=0$, $\eta_{12}=\frac{b_{20}}{2(\lambda_2+2)}(\kappa_1((\lambda_2+2)\kappa_1+(\kappa_2-\kappa_1))\otimes1-\kappa_1\otimes s
-(2(\lambda_2+2)\kappa_1+(\kappa_2-\kappa_1))s\otimes1+s\otimes s+2(\lambda_2+2)s^{(2)}\otimes1)\neq 0$,where $b_{20}, \lambda_2, \kappa_1,\kappa_2\in{\bf k}$.

(B9)\  $\beta_1=\lambda_1s\otimes 1-1\otimes s+\kappa_1\otimes 1$, $\beta_2=-s\otimes 1-1\otimes s+\kappa_2\otimes 1$, $\beta_3=(1-\lambda_1)s\otimes 1-1\otimes s+(\kappa_2-\kappa_1)\otimes 1$,
 $\eta_{11}=\eta_{22}=0$, \begin{eqnarray*}\begin{array}{lll}\eta_{12}&=&b_{20}((\frac{\lambda_1}{2\lambda_1-2}(\kappa_2-\kappa_1)^2+\frac{1-2\lambda_1}{2-2\lambda_1}
\kappa_1(\kappa_2-\kappa_1)
+\frac12\kappa_1^2)\otimes1
\\ && +(\frac{(2\lambda_1-1)(\kappa_2-\kappa_1)}{2-2\lambda_1}
-\kappa_1)s\otimes1
+\frac{(2\lambda_1-1)\kappa_1+2\lambda_1(\kappa_2-\kappa_1)}{2-2\lambda_1}\otimes s-\frac{\lambda_1}{1-\lambda_1}\otimes s^{(2)}\\ &&
+\frac{1-2\lambda_1}{2-2\lambda_1}s\otimes s+s^{(2)}\otimes1)\neq 0,\end{array}\end{eqnarray*} where $b_{20}, \lambda_1, \kappa_1,\kappa_2\in{\bf k}$.

(B10)\    $\beta_1=-1\otimes s+\kappa_1\otimes 1$, $\beta_2=s\otimes 1-1\otimes s+\kappa_2\otimes 1$, $\beta_3=2s\otimes 1-1\otimes s+(\kappa_2-\kappa_1)\otimes 1$,
 $\eta_{11}=\eta_{22}=0$, \begin{eqnarray*}\begin{array}{lll}\eta_{12}&=&b_{30}(s^{(3)}\otimes 1-\frac12(\kappa_1+\kappa_2))s^{(2)}\otimes1+\frac12s^{(2)}\otimes s+\frac1{12}(\kappa_1^2+4\kappa_1\kappa_2+\kappa_2^2)s\otimes1\\
&&-\frac16(2\kappa_1+\kappa_2)s\otimes s+\frac16s\otimes s^{(2)}
-\frac1{12}(\kappa_1\kappa_2^2+\kappa_1^2\kappa_2)\otimes1\\&&
+\frac1{12}(\kappa_1^2+2\kappa_1\kappa_2)\otimes s
-\frac16\kappa_1\otimes s^{(2)})\neq 0,\end{array}\end{eqnarray*}where $b_{30}, \lambda_1, \kappa_1,\kappa_2\in{\bf k}$.

(B11)\  $\beta_1=\frac23s\otimes 1-1\otimes s+\kappa_1\otimes 1$, $\beta_2=-\frac53s\otimes 1-1\otimes s+\kappa_2\otimes 1$, $\beta_3=\frac23s\otimes 1-1\otimes s+(\kappa_2-\kappa_1)\otimes 1$,
$\eta_{11}=\eta_{22}=0$, $$\begin{array}{lll}\eta_{12}&=&b_{30}((s^{(3)}\otimes1-1\otimes s^{(3)})+\frac12(s^{(2)}\otimes s-s\otimes s^{(2)})-\frac12(\kappa_1+\kappa_2)s^{(2)}\otimes 1\\&&+\frac12(2\kappa_2-\kappa_1)\otimes s^{(2)}-\frac14(\kappa_1^2-4\kappa_1\kappa_2+\kappa_2^2)s\otimes1+\frac14(\kappa_1^2+2\kappa_1\kappa_2-2\kappa_2^2)\otimes s\\ &&+ \frac12(\kappa_2-2\kappa_1)s\otimes s +\frac1{12}(2\kappa_1-\kappa_2)(\kappa_1^2-\kappa_1\kappa_2-2\kappa_2^2)\otimes1),\end{array}$$ where $b_{30}, \kappa_1,\kappa_2\in{\bf k}$.

(B12)\
 $\beta_1=-1\otimes s+\kappa_1\otimes 1$, $\beta_2=\lambda_2s\otimes 1-1\otimes s+\kappa_2\otimes 1$, $\beta_3=(\lambda_2+3)s\otimes 1-1\otimes s+(\kappa_2-\kappa_1)\otimes 1$, $\eta_{11}=\eta_{22}=0$,
 $$\begin{array}{lll}\eta_{12}&=&b_{21}((s^{(2)}\otimes s-s\otimes s^{(2)})+(\kappa_1\otimes s^{(2)}-(\kappa_2-\kappa_1)s^{(2)}\otimes1)
\\ && +(\frac12(\kappa_2-\kappa_1)(3\kappa_1-\kappa_2)s\otimes 1+\frac12\kappa_1(3\kappa_1-2\kappa_2)\otimes s)+(\kappa_2-2\kappa_1)s\otimes s\\ &&+\frac12\kappa_1(\kappa_2-\kappa_1)(\kappa_2-2\kappa_1)\otimes 1)\neq 0,\end{array}$$
where $b_{21}, \lambda_2, \kappa_1,\kappa_2\in{\bf k}$.

(B13)\  $\beta_1=2\lambda_1s\otimes 1-1\otimes s+\kappa_1\otimes 1$, $\beta_2=-s\otimes 1-1\otimes s+\kappa_2\otimes 1$, $\beta_3=-1\otimes s+(\kappa_2-\kappa_1)\otimes 1$,
$\eta_{11}=\eta_{22}=0$, $$\begin{array}{lll}\eta_{12}&=&b_{21}(s^{(2)}\otimes s+3s\otimes s^{(2)}-(\kappa_2-\kappa_1)s^{(2)}\otimes 1-3(2\kappa_2-\kappa_1)\otimes s^{(2)}\\ && -(3\kappa_2-2\kappa_1)s\otimes s+\frac12(\kappa_1^2-4\kappa_1\kappa_2+3\kappa_2^2)s\otimes 1+\frac12(\kappa_1^2-6\kappa_1\kappa_2+6\kappa_2^2)\otimes s\\ &&+\frac12(3\kappa_1^3-6\kappa_1^2\kappa_2-3\kappa_1\kappa_2^2-2\kappa_2^3)\otimes 1+6\otimes s^{(3)})\neq 0,\end{array}$$ where $b_{21}, \kappa_1,\kappa_2\in{\bf k}$.

(B14)\  $\eta_{11}=\eta_{22}=a_{00}\otimes 1\neq 0$,   $\eta_{12}=b_{00}\otimes 1\neq 0$, $\beta_1=\beta_2= \lambda_1s\otimes 1-1\otimes s+\kappa_1\otimes 1$, $\beta_3=-1\otimes s$ for some $a_{00}, b_{00},\lambda_1,\kappa_1\in{\bf k}$.

(B15)\    $\eta_{11}=a_{00}\otimes 1\neq \eta_{22}=d_{00}\otimes 1$,  $\eta_{12}=b_{00}\otimes 1$, $\beta_1=\beta_2=\lambda_1\otimes 1-1\otimes s+\kappa_1\otimes 1$, $\beta_3=-1\otimes s$ for some $a_{00},b_{00}, d_{00}, \lambda_1,\kappa_1\in{\bf k}$.

(B16)\    $\eta_{11}=a_{00}\otimes 1$, $ \eta_{22}=d_{00}\otimes 1$,   $$\eta_{12}=b_{10}\left(-\kappa_1\otimes 1+s\otimes 1+\frac{d_{00}}{d_{00}-a_{00}}\otimes s\right),$$ $\beta_1=-\otimes s+\kappa_1\otimes 1$, $\beta_2=-s\otimes 1-1\otimes s+\kappa_1\otimes 1$, $\beta_3=-1\otimes s$
for some  $a_{00},d_{00},b_{10}, \lambda_1,\kappa_1\in{\bf k}$.
\end{theorem}
\begin{proof} Similar to the proof of Theorem \ref{thm51}.\end{proof}

\begin{lemma}\label{lem45a} Suppose $\eta_{ij}$ ($1\leq i,j\leq 2$) satisfy (\ref{eq28})-(\ref{eq215}),  $\eta_{12}\eta_{21}\neq 0$ and $\eta_{11}\eta_{22}=0$. Then $\eta_{11}=a_{00}\otimes 1$ and $\eta_{22}=d_{00}\otimes 1$ for some $a_{00},d_{00}\in{\bf k}$, $\kappa_1-\kappa_2=\kappa_3=0$. Moreover, $\lambda_1=\lambda_2$ and $\lambda_3\in \{0,1\}.$ Further,

(i)  if $\eta_{11}=0$, $\lambda_3=0$, then  $\eta_{12}=b_{00}\otimes 1$ and $\eta_{21}=c_{00}\otimes 1$, $\eta_{22}=d_{00}\otimes 1$ for some $b_{00},c_{00}, d_{00}\in{\bf k}$, where $b_{00}c_{00}\neq 0$.

(ii) if $\eta_{22}=0$, $\lambda_3=0$, then  $\eta_{11}=a_{00}\otimes 1$, $\eta_{12}=b_{00}\otimes 1$ and $\eta_{21}=c_{00}\otimes 1$ for some $a_{00},b_{00}, c_{00}\in{\bf k}$, where $b_{00}c_{00}\neq 0$.

(iii) if $\lambda_3=1$, then $\eta_{12}=b_{01}\otimes s$, $\eta_{21}=c_{01}\otimes s$ and $\eta_{11}=\eta_{22}=0$.
\end{lemma}
\begin{proof} Since $\eta_{12}\eta_{21}\neq 0$, $\kappa_1-\kappa_2+\kappa_3=\kappa_2-\kappa_1+\kappa_3=0$. Hence $\kappa_1-\kappa_2=\kappa_3=0.$
Assume that $\eta_{11}=0$. Then one obtains \begin{eqnarray}\label{eq42}&&\sum\limits_{i=0,k=0}^{m_2,m_3}\sum\limits_{j=0,v=0}^{q_i,u_k}\sum\limits_{t=0}^kb_{ij}c_{kv}s^{(i)}s^{(t)}\otimes s^{(j)}s^{(k-t)}\otimes s^{(v)}\label{eqa42}\\
&&=\sum\limits_{i=0,k=0}^{m_2,m_3}\sum\limits_{j=0,v=0}^{q_i,u_k}\sum\limits_{t=0}^kb_{ij}c_{kv}s^{(i)}s^{(t)}\otimes s^{(v)}\otimes s^{(j)}s^{(k-t)}\nonumber\end{eqnarray} from (\ref{eq212}). By (\ref{eqa42}), we get $\max\{q_1,q_2,\cdots,q_{m_2}\}+m_3=\max\{u_1, u_2, \cdots, u_{m_3}\}$. Similarly, we have $\max\{q_1,q_2,\cdots,q_{m_2}\}+m_4=\max\{v_1,v_2,\cdots,v_{m_4}\}$ and $\max\{v_1,v_2,\cdots,v_{m_4}\}+m_3=\max\{u_1,u_2,\cdots,u_{m_3}\}$ by (\ref{eq213}) and (\ref{eq214}) respectively. Therefore either $\eta_{22}=0$ or $m_4=0$. If $m_4=0$ and $\eta_{22}\neq 0$, then $\lambda_3=\kappa_3=0$ and  $\eta_{22}=d_{00}\otimes 1$ for some nonzero $d_{00}\in {\bf k}$ by Lemma \ref{lem42a}. Thus $\max\{u_1,\cdots,u_{m_3}\}+m_2=\max\{q_1,q_2,\cdots,q_{m_2}\}$ by (\ref{eq215}). Hence $m_2=m_3=0$, which implies that $\lambda_1-\lambda_2+\lambda_3, \lambda_2-\lambda_1+\lambda_3\in \{0,1\}$ by Lemma \ref{lem24a}.

If  $\lambda_1-\lambda_2+\lambda_3=\lambda_2-\lambda_1+\lambda_3=0$, then $\lambda_1=\lambda_2$ and $\lambda_3=0$. So $\eta_{12}=b_{00}\otimes 1$ and $\eta_{21}=c_{00}\otimes 1$ for some nonzero $b_{00},c_{00}\in {\bf k}$ by (T2). This finishes the proof of (i) by now. The proof of (ii) is similar.

If  $\lambda_1-\lambda_2+\lambda_3=\lambda_2-\lambda_1+\lambda_3=1$, then $\lambda_1=\lambda_2$ and $\lambda_3=1$. So $\eta_{12}=b_{01}\otimes s$ and $\eta_{21}=c_{01}\otimes s$ by (T1). If $\eta_{11}=0$, then $\eta_{22}=0$ by  (\ref{eq213}). Conversely, if $\eta_{22}=0$, then $\eta_{11}=0$ by (\ref{eq214}). Thus  $\eta_{22}=\eta_{11}=0$ if $\eta_{11}\eta_{22}=0$.

If $\lambda_1-\lambda_2+\lambda_3=0$ and $\lambda_2-\lambda_1+\lambda_3=1$, then  $\eta_{12}=b_{00}\otimes 1$ and $\eta_{21}=c_{01}\otimes s$, which make (\ref{eq42}) fails. This is a contradiction.

Similarly, the case of $\lambda_1-\lambda_2+\lambda_3=1$ and $\lambda_2-\lambda_1+\lambda_3=0$ does not happen.
\end{proof}

\begin{lemma}\label{lem46a} Suppose $\eta_{ij}$ ($1\leq i,j\leq 2$) satisfy (\ref{eq28})-(\ref{eq215}) and  $\eta_{ij}\neq 0$ for all $i, j$. Then $\kappa_1-\kappa_2=\kappa_3=0$ and $\lambda_3\in \{0,1,2,3\}$. Further,

(i) if  $\lambda_3=0$, then $\lambda_1=\lambda_2$,  $\eta_{11}=a_{00}\otimes 1$, $\eta_{12}=b_{00}\otimes 1$, $\eta_{21}=c_{00}\otimes 1$ and  $\eta_{22}=d_{00}\otimes 1$ for some nonzero $a_{00},b_{00}, c_{00}, d_{00}\in{\bf k}$.

(ii)  if $\lambda_3=1$, then $\lambda_1=\lambda_2$, $\eta_{11}=a_{10}(-\kappa_1\otimes 1+s\otimes 1-\lambda_1\otimes s)$, $\eta_{12}=\frac{b_{10}}{a_{10}}\eta_{11}$, $\eta_{21}=-\frac{a_{10}}{b_{10}}\eta_{11}$, and $\eta_{22}=-\eta_{11}$ for some nonzero $a_{10},b_{10}\in {\bf k}.$

(iii)  if $\lambda_3=2$ and $\lambda_1=0$,  then $\lambda_2=0$,  $\eta_{11}=\frac{a_{20}}4(2\kappa_1^2\otimes 1-\kappa_1\otimes s-4\kappa_1s\otimes 1+s\otimes s+4s^{(2)}\otimes1)$, $\eta_{12}=\frac{b_{20}}{a_{20}}\eta_{11}$, $\eta_{21}=-\frac{a_{20}}{b_{20}}\eta_{11}$ and $\eta_{22}=-\eta_{11}$ for some nonzero  $a_{20},b_{20}\in{\bf k}.$

(iv)  if $\lambda_3=2$ and $\lambda_1\neq 0$, then $\lambda_1=\lambda_2=-1$,   $\eta_{11}={a_{20}}(\frac{\kappa_1^2}2\otimes 1-\kappa_1s\otimes 1-\frac34\kappa_1\otimes s+\frac34s\otimes s+\frac12\otimes s^{(2)}+s^{(2)}\otimes 1)$, $\eta_{12}=\frac{b_{20}}{a_{20}}\eta_{11}$, $\eta_{21}=-\frac{a_{20}}{b_{20}}\eta_{11}$ and $\eta_{22}=-\eta_{11}$ for some nonzero  $a_{20},b_{20}\in{\bf k}.$

(v)  If $\lambda_3=3$,  then $\lambda_1=\lambda_2=0$, $\eta_{11}=a_{21}((s^{(2)}\otimes s-s\otimes s^{(2)})+\kappa_1\otimes s^{(2)}
+\frac12\kappa_1^2s\otimes1-\kappa_1s\otimes s)$, $\eta_{12}=\frac{b_{21}}{a_{21}}\eta_{11}$, $\eta_{21}=-\frac{a_{21}}{b_{21}}\eta_{11}$ and $\eta_{22}=-\eta_{11}$ for some nonzero  $a_{21},b_{21}\in{\bf k}.$
\end{lemma}
\begin{proof} Since $\eta_{12}\eta_{21}\neq 0$, we have $\kappa_1-\kappa_2+\kappa_3=\kappa_2-\kappa_1+\kappa_3=0$ by Lemma {\ref{lem24a}}, that is, $\kappa_1-\kappa_2=\kappa_3=0$. Moreover  $\eta_{11}\eta_{22}\neq 0$ implies that $\lambda_3\in\{0, 1 ,2, 3\}$ by Lemmas \ref{lem41a}-\ref{lem42a}.

{\it Case (i)}: $\lambda_3=0$.  In this case, we have $\eta_{11}=a_{00}\otimes 1$ and $\eta_{22}=d_{00}\otimes1$ by Lemmas \ref{lem41a}-\ref{lem42a}. From
(T1)-(T10), we get $(\lambda_1-\lambda_2)(\lambda_2-\lambda_1)\geq0$ since $\eta_{12}\eta_{21}\neq 0$. Thus $\lambda_1= \lambda_2$, $\eta_{12}=b_{00}\otimes 1$ and $\eta_{21}=c_{00}\otimes 1$.

{\it Case (ii)}: $\lambda_3=1$. By Lemmas \ref{lem41a}-\ref{lem42a},  we get $\eta_{11}=a_{10}(-\kappa_1\otimes 1+s\otimes 1-\lambda_1\otimes s)$ and $\eta_{22}=d_{10}(-\kappa_1\otimes 1+s\otimes 1-\lambda_2\otimes s)$. Since $\lambda_1-\lambda_2+\lambda_3\in\{0, 1, 2, 3\}$ from (T1)-(T10),  we have  $\lambda_1-\lambda_2\in\{-1, 0, 1, 2\}$. Similarly, $\lambda_2-\lambda_1\in\{-1, 0, 1, 2\}$. Thus $\lambda_1-\lambda_2\in\{-1, 0, 1\}$. Suppose $\lambda_1-\lambda_2=-1$. Then $\eta_{12}=b_{00}\otimes 1$
and $\lambda_2-\lambda_1+\lambda_3=2$. Thus $\lambda_2+\lambda_3=1$ and $\lambda_2=0$ by (T8). So
$\eta_{21}=\frac{c_{20}}2(\kappa_1^2\otimes 1-\kappa_1\otimes s-2\kappa_1s\otimes 1+s\otimes s+2s^{(2)}\otimes1),\ \lambda_2=0$ and $\lambda_3=-\lambda_1=1.
$ Hence  $a_{10}^2=-\frac12b_{00}c_{20}$ and $a_{10}b_{00}=\frac12b_{00}c_{20}=0$ by (\ref{eq212}) and (\ref{eq213}) respectively. Thus $a_{10}=0$, that is, $\eta_{11}=0$, which is impossible.  Similarly, we can prove that there are not nonzero $\eta_{ij}$ satisfying (\ref{eq212})-(\ref{eq215}) if $\lambda_1-\lambda_2=1$.
Suppose $\lambda_1-\lambda_2=0$. Then $\eta_{12}=b_{10}(-\kappa_1\otimes 1+s\otimes 1-\lambda_1\otimes s)$ for some nonzero $b_{10}$ and $\eta_{21}=c_{10}(-\kappa_1\otimes 1+s\otimes 1-\lambda_2\otimes s)$ for some nonzero $c_{10}$ by (T6). From (\ref{eq212})-(\ref{eq215}),  we get $a^2_{10}+b_{10}c_{10}=b_{10}(a_{10}+d_{10})=c_{10}(a_{10}+d_{10})=d_{10}^2+b_{10}c_{10}=0$. Thus $d_{10}=-a_{10}$ and $a_{10}^2=-b_{10}c_{10}$.

{\it Case (iii)}: $\lambda_3=2$ and $\lambda_1=0$. In this case, we have $\eta_{11}=\frac{a_{20}}{4}(2\kappa_1^2\otimes 1-\kappa_1\otimes s-4\kappa_1s\otimes 1+s\otimes s+4s^{(2)}\otimes 1)$. Since $\lambda_1-\lambda_2+\lambda_3=2-\lambda_2\in\{0,1,2,3\}$, we have $\lambda_2\in\{-1,0,1,2\}$. Consequently, $\lambda_2-\lambda_1+\lambda_3=\lambda_2+2\in\{1,2,3,4\}\cap \{0,1,2,3\}=\{1,2,3\}.$ If $\lambda_2=1$, then $\eta_{22}=0$ by Lemma \ref{lem42a}. Hence $\lambda_2\in\{0,-1\}$.  If $\lambda_2=0$, then $\eta_{22}=\frac{d_{20}}{4}(2\kappa_1^2\otimes 1-\kappa_1\otimes s-4\kappa_1s\otimes 1+s\otimes s+4s^{(2)}\otimes 1)$, $\eta_{12}=\frac{b_{20}}{4}(2\kappa_1^2\otimes 1-\kappa_1\otimes s-4\kappa_1s\otimes 1+s\otimes s+4s^{(2)}\otimes 1)$ and $\eta_{21}=\frac{c_{20}}{4}(2\kappa_1^2\otimes 1-\kappa_1\otimes s-4\kappa_1s\otimes 1+s\otimes s+4s^{(2)}\otimes 1)$. From (\ref{eq212})-(\ref{eq215}), we get $a^2_{20}+b_{20}c_{20}=b_{20}(a_{20}+d_{20})=c_{20}(a_{20}+d_{20})=d_{20}^2+b_{20}c_{20}=0$. Thus $d_{20}=-a_{20}.$
  If $\lambda_2=-1$, then $\eta_{22}=d_{20}(\frac{\kappa_2^2}2\otimes1-\kappa_2s\otimes1-\frac{3\kappa_2}{4}\otimes s+\frac{1}{2}\otimes s^{(2)}
+\frac{3}{4}s\otimes s+s^{(2)}\otimes1)$ by Lemma \ref{lem42a}. Moreover, since $\lambda_1-\lambda_2+\lambda_3=3$, we have $\eta_{12}=\vartheta$ or $\eta_{12}=\vartheta'$, where
$\vartheta=b_{30}(s^{(3)}\otimes 1-\kappa_1s^{(2)}\otimes 1+\frac12s^{(2)}\otimes s+\frac12\kappa_1^2s\otimes 1-\frac12s\otimes s+\frac16s\otimes s^{(2)}+\frac16\kappa_1^3\otimes1 +\frac14\kappa_1^2\otimes s-\frac16\kappa_1\otimes s^{(2)})$ and
$\vartheta'=b_{21}(s^{(2)}\otimes s-s\otimes s^{(2)}+\kappa_1\otimes s^{(2)}+\frac12\kappa_1^2s\otimes 1-\kappa_1s\otimes s)$.
Similarly, we have $\eta_{21}=c_{10}(-\kappa_1\otimes 1+s\otimes 1+\frac12\otimes s)$.  If $\eta_{12}=\vartheta$ or $ \vartheta'$, and (\ref{eq212}) holds, then $\eta_{11}=0$. This is impossible.

{\it Case (iv)}: $\lambda_3=2$ and $\lambda_1\neq 0$. In this case, we have $\lambda_1=-1$ and $\eta_{11}={a_{20}}(\frac{\kappa_1^2}2\otimes 1-\kappa_1s\otimes 1-\frac34\kappa_1\otimes s+\frac34s\otimes s+\frac12\otimes s^{(2)}+s^{(2)}\otimes 1)$. Since $\lambda_1-\lambda_2+\lambda_3\in\{0,1,2,3\}$,  $\lambda_1-\lambda_2\in\{-2,-1,0,1\}$. Similarly, $\lambda_2-\lambda_1\in\{-2,-1,0,1\}$. Thus $\lambda_1-\lambda_2\in\{-1,0,1\}$ and hence $\lambda_2\in\{0,-1,-2\}$. If $\lambda_2=-2$, then $\eta_{22}=0$ by Lemma \ref{lem42a}, which is impossible. Hence $\lambda_2\in \{0,-1\}$.
 If $\lambda_2=0$, then $\lambda_1-\lambda_2+\lambda_3=1$ and $\eta_{12}=b_{10}(-\kappa_1\otimes 1+s\otimes 1+\frac{1}2\otimes s)$. Since $\lambda_2-\lambda_1+\lambda_3=3$ and $\lambda_2=0$,  $\eta_{21}\in\{\frac{c_{30}}{b_{30}}\vartheta,\frac{c_{21}}{b_{21}}\vartheta'\}$.
It is easy to check that (\ref{eq212}) holds only if $\eta_{11}=0$.
 If $\lambda_2=-1$, then $\lambda_1-\lambda_2+\lambda_3=\lambda_2-\lambda_1+\lambda_3=2$, $\eta_{12}=\frac{b_{20}}{a_{20}}\eta_{11}$,
$\eta_{21}=\frac{c_{20}}{a_{20}}\eta_{11}$ and $\eta_{22}=\frac{d_{20}}{a_{20}}\eta_{11}$. Thus $a_{20}^2+b_{20}c_{20}=b_{20}(a_{20}+d_{20})=c_{20}(a_{20}+d_{20})=d_{20}^2+b_{20}c_{20}=0$. So $d_{20}=-a_{20}$ and $a_{20}^2=-b_{20}c_{20}$.

{\it Case (v)}:  $\lambda_3=3$.  Note that $\lambda_1-\lambda_2\in\{0,-1,-2,-3\}\cap\{0,1,2,3\}$. Then $\lambda_1=\lambda_2=0$ by Lemma \ref{lem41a} and Lemma \ref{lem42a}. Moreover,  $\eta_{11}=a_{21}((s^{(2)}\otimes s-s\otimes s^{(2)})+\kappa_1\otimes s^{(2)}
+\frac12\kappa_1^2s\otimes1-\kappa_1s\otimes s)$ and $\eta_{22}=\frac{d_{21}}{a_{21}}\eta_{11}$, $\eta_{12}=\frac{b_{21}}{a_{21}}\eta_{11}$ and $\eta_{21}=\frac{c_{21}}{a_{21}}\eta_{11}$. It is easy to check that (\ref{eq212})-(\ref{eq215}) hold if and only if $d_{21}=-a_{21}$ and $a_{21}^2=-b_{21}c_{21}.$
\end{proof}

Using Lemma \ref{lem45a} and Lemma \ref{lem46a}, we can determine all Sch\"odinger-Virasoro Lie $H$-pseudoalgebras  with $\eta=0$, $\alpha_m=0$ and $\eta_{12}\eta_{21}\neq 0$ in the following theorem.

\begin{theorem}  \label{thm53} Let $L$ be a Sch\"odinger-Virasoro Lie $H$-pseudoalgebra with pseudobracket given by
$$\left\{\begin{array}{l}[e_0,e_0]=\alpha\otimes_He_0, \quad \ [e_0, e_i]=\beta_{i}\otimes_H e_i\quad 1\leq i\leq 3,\\

[e_1, e_1]=[e_1, e_2]=[e_2, e_2]=[e_3,e_3]=0,\\

 [e_1, e_3]=\eta_{11}\otimes_He_1+\eta_{12}\otimes _He_2,\\

[e_2, e_3]=\eta_{21}\otimes _He_1+ \eta_{22}\otimes_He_2,\end{array}\right.$$
where $\eta_{12}\eta_{21}\neq 0$. Then  $L$ is  one of the following types:

(C1)\  $\eta_{11}=0$,  $\eta_{12}=b_{00}\otimes 1$ and $\eta_{21}=c_{00}\otimes 1$, $\eta_{22}=d_{00}\otimes 1$, $\beta_1=\beta_2=\lambda_1s\otimes 1-1\otimes s+\kappa_1\otimes 1$, $\beta_3=-1\otimes s$, where $b_{00},c_{00}, d_{00}, \lambda_1,\kappa_1\in{\bf k}$ satisfying $b_{00}c_{00}\neq 0$.

(C2) \    $\eta_{11}=a_{00}\otimes 1$, $\eta_{12}=b_{00}\otimes 1$, $\eta_{21}=c_{00}\otimes 1$, $\eta_{22}=0$, $\beta_1=\beta_2=\lambda_1s\otimes 1-1\otimes s+\kappa_1\otimes 1$, $\beta_3=-1\otimes s$, where $b_{00},c_{00}, d_{00}, \lambda_1,\kappa_1\in{\bf k}$ satisfying $b_{00}c_{00}\neq 0$.

(C3)\ $\eta_{12}=b_{01}\otimes s$, $\eta_{21}=c_{01}\otimes s$ and $\eta_{11}=\eta_{22}=0$, $\beta_1=\beta_2=\lambda_1s\otimes 1-1\otimes s+\kappa_1\otimes 1$, $\beta_3=s\otimes 1-1\otimes s$, where $b_{01},c_{01}, d_{00}, \lambda_1,\kappa_1\in{\bf k}$ satisfying $b_{01}c_{01}\neq 0$.

(C4) \  $\eta_{11}=a_{00}\otimes 1$, $\eta_{12}=b_{00}\otimes 1$, $\eta_{21}=c_{00}\otimes 1$ and  $\eta_{22}=d_{00}\otimes 1$ for some nonzero $a_{00},b_{00}, c_{00}, d_{00}\in{\bf k}$, $\beta_1=\beta_2=\lambda_1s\otimes 1-1\otimes s+\kappa_1\otimes 1$, $\beta_3=-1\otimes s$ for some $\lambda_1,\kappa_1\in{\bf k}$.

(C5) \  $\eta_{11}=a_{10}(-\kappa_1\otimes 1+s\otimes 1-\lambda_1\otimes s)$, $\eta_{12}=\frac{b_{10}}{a_{10}}\eta_{11}$, $\eta_{21}=-\frac{a_{10}}{b_{10}}\eta_{11}$, and $\eta_{22}=-\eta_{11}$ for some nonzero $a_{10},b_{10}\in {\bf k},$
$\beta_1=\beta_2=\lambda_1s\otimes 1-1\otimes s+\kappa_1\otimes 1$, $\beta_3=s\otimes1-1\otimes s$ for some $\lambda_1,\kappa_1\in{\bf k}$.

(C6) \  $\eta_{11}=\frac{a_{20}}4(2\kappa_1^2\otimes 1-\kappa_1\otimes s-4\kappa_1s\otimes 1+s\otimes s+4s^{(2)}\otimes1)$, $\eta_{12}=\frac{b_{20}}{a_{20}}\eta_{11}$, $\eta_{21}=-\frac{a_{20}}{b_{20}}\eta_{11}$ and $\eta_{22}=-\eta_{11}$ for some nonzero  $a_{20},b_{20}\in{\bf k},$ $\beta_1=\beta_2=-1\otimes s+\kappa_1\otimes 1$, $\beta_3=2s\otimes 1-1\otimes s$ for some $\lambda_1,\kappa_1\in {\bf k}$.

(C7)\  $\eta_{11}={a_{20}}(\frac{\kappa_1^2}2\otimes 1-\kappa_1s\otimes 1-\frac34\kappa_1\otimes s+\frac34s\otimes s+\frac12\otimes s^{(2)}+s^{(2)}\otimes 1)$, $\eta_{12}=\frac{b_{20}}{a_{20}}\eta_{11}$, $\eta_{21}=-\frac{a_{20}}{b_{20}}\eta_{11}$ and $\eta_{22}=-\eta_{11}$ for some nonzero  $a_{20},b_{20}\in{\bf k},$ $\beta_1=\beta_2=-s\otimes 1-1\otimes s+\kappa_1\otimes 1$, $\beta_3=2s\otimes 1-1\otimes s$ for some $\kappa_1\in {\bf k}$.

(C8)\  $\eta_{11}=a_{21}((s^{(2)}\otimes s-s\otimes s^{(2)})+\kappa_1\otimes s^{(2)}
+\frac12\kappa_1^2s\otimes1-\kappa_1s\otimes s)$, $\eta_{12}=\frac{b_{21}}{a_{21}}\eta_{11}$, $\eta_{21}=-\frac{a_{21}}{b_{21}}\eta_{11}$ and $\eta_{22}=-\eta_{11}$ for some nonzero  $a_{21},b_{21}\in{\bf k},$ $\beta_1=\beta_2=-1\otimes s+\kappa_1\otimes 1$, $\beta_3=3s\otimes 1-1\otimes s$ for some $\kappa_1\in{\bf k}$.
\end{theorem}
\begin{proof}Similar to the proof of Theorem \ref{thm52}.\end{proof}
Next, we assume that $\eta\neq 0$ or $\alpha_m'\neq 0$. First, let us determine $\eta_{ij}$ for $1\leq i,j\leq 2$ in the case of $\eta\neq 0$.

\begin{lemma} \label{lem48} Let $L$ be a Schr\"odinger-Virasoro Lie $H$-pseudoalgebra with pseudobrackets given by (I*). Suppose $\eta\neq 0$. Then $\beta_{1}=-1\otimes s$, $\eta_{11}=\eta_{21}=0$, $\eta_{22}=d_{00}\otimes 1$ and
$\alpha_m'=w_{01}(1\otimes s-s\otimes 1)$ for some $d_{00},  w_{01}\in{\bf k}$. Moreover,

(1)  If $\eta_{12}\eta_{22}\neq 0$, then $\lambda_1=\lambda_2=\lambda_3=0$, $\kappa_1=\kappa_2=\kappa_3=w_{01}=0$, $\eta_{12}=b_{00}\otimes 1\neq 0$.

(2)  If $\eta_{22}=0$, then $\lambda_1=\lambda_2-\lambda_3=0$, $\kappa_1=\kappa_2-\kappa_3=0$, $\eta_{12}=b_{00}\otimes 1\neq 0$ and $\alpha_m'=w_{01}(1\otimes s-s\otimes 1)$, where $\lambda_2w_{01}=\kappa_2w_{01}=0$.

(3) If $\eta_{22}\neq 0$ and $\eta_{12}=0$, then $\lambda_i=\kappa_i=0$ for $i=1, 3$ and $\alpha_m'=0$
\end{lemma}
\begin{proof} Since $\eta\neq 0$, $\eta_{21}=0$ by (\ref{eq25}) and $\alpha_m'=w_{01}(1\otimes s-s\otimes 1)$ by Theorem \ref{thm31}. Moreover, $\eta_{11}=a_{00}\otimes 1$ and $\eta_{22}=d_{00}\otimes 1$ by Lemma \ref{lem44a}. From (\ref{eq27}),  we get
$(\eta\Delta\otimes 1)\eta_{22}=(1\otimes \eta_{22}\Delta)\eta+(12)(1\otimes \eta_{11}\Delta)\eta$.
 Thus  $\eta_{11}=0$. Similarly, one obtains that $\lambda_1=\kappa_1=0$ from (\ref{eq217}).
 (i) If $d_{00}\eta_{12}\neq 0$, then $\lambda_3=0$ by Lemma \ref{lem42a}. In addition,  either $\eta_{12}=b_{00}\otimes 1$, or $\eta_{12}=b_{10}(s\otimes 1+1\otimes s)$ by Lemma \ref{lem44a}. From (\ref{eq26}), we get either $\alpha_m'=0$, or $\alpha'_m=\frac{b_{10}}{d_{00}}(1\otimes s-s\otimes 1)$. If $\eta_{12}=b_{10}(1\otimes s-s\otimes1)\neq 0$, then $\alpha'_m=\frac{b_{10}}{d_{00}}(1\otimes s-s\otimes 1)$, $\lambda_2b_{10}=\kappa_2b_{10}=0$ by (\ref{eq218}), and $\lambda_1-\lambda_2+\lambda_3=1$ by Lemma \ref{lem24a}. Since $b_{10}\neq 0$, we have $\lambda_2=0$ and $\lambda_1-\lambda_2+\lambda_3=0$. This is impossible.  Hence $\eta_{12}=b_{00}\otimes 1$ for some nonzero $b_{00}$ and $\lambda_2=0$. Furthermore, we get $\alpha_m'=0$ by (\ref{eq26}).
(ii) If $d_{00}=0$, then $1\otimes \eta_{12}=(12)(1\otimes \eta_{12})$ by (\ref{eq26}). Thus $\eta_{12}=b_{00}\otimes 1\neq 0$ by Lemma \ref{lem44a} and the assumption that $\eta_{ij}$ $ (1\leq i,j\leq 2)$ are not all zero. Moreover, $0=\lambda_1-\lambda_2+\lambda_3=\lambda_3-\lambda_2$ by (T2).
If $\alpha_m'=w_{01}(1\otimes s-s\otimes 1),$ then $\lambda_2w_{01}=\kappa_2w_{01}=0$ by (\ref{eq218}). (iii) If $\eta_{12}=0$, then $d_{00}\neq 0$,
$\alpha_m'=0$ and $\lambda_3=\kappa_3=0$.
\end{proof}
From the above lemma, we get the following theorem.
\begin{theorem} \label{thm54} Let $L$ be a Schr\"odinger-Virasoro Lie $H$-pseudoalgebra with pseudobrackets given by (I*). Suppose $\eta\neq 0$. Then $L$ is one of  the following three types.

(D1)\  $\eta_{11}=\eta_{21}=0$,  $\eta_{12}=b_{00}\otimes 1\neq 0$, $\eta_{22}=d_{00}\otimes 1\neq 0$,  $\alpha_m'=0$, $\eta=a\otimes1$, $\beta_1=\beta_2=\beta_3=-1\otimes s$ for some $a,b_{00},d_{00}\in {\bf k}$.

(D2)\  $\eta_{11}=\eta_{21}=\eta_{22}=0$, $\eta_{12}=b_{00}\otimes 1\neq 0$, $\alpha_m'=w_{01}(1\otimes s-s\otimes 1)$,  $\eta=a\otimes1$, $\beta_1=-1\otimes s$,  $\beta_2=\beta_3=\lambda_2s\otimes 1-1\otimes s+\kappa_2\otimes 1$ for some $a,b_{00},w_{01},\lambda_2,\kappa_2\in {\bf k}$, where $\lambda_2w_{01}=\kappa_2w_{01}=0$.

(D3)\   $\eta_{11}=\eta_{21}=\eta_{12}=0$,  $\eta_{22}=d_{00}\otimes 1\neq 0$, $\alpha_m'=0$, $\eta=a\otimes1$, $\beta_3=\beta_1=-1\otimes s$, $\beta_2=\lambda_2s\otimes 1-1\otimes s+\kappa_2\otimes 1$ for some $a,d_{00}, \lambda_2,\kappa_2\in{\bf k}$.
\end{theorem}

By now, we have determined the Sch\"odinger-Virasoro Lie $H$-pseudoalgebras with $\eta=a\otimes 1\neq 0$ by Theorem {\ref{thm54} and
the Sch\"odinger-Virasoro Lie $H$-pseudoalgebras with $\eta=\alpha_m'= 0$ by Theorems \ref{thm51}, Theorem \ref{thm52}, Theorem \ref{thm53}.
Next, we assume that $\eta=0$ and $\alpha_m'\neq 0$. In this case, (\ref{eq25}) and (\ref{eq27}) are equivalent. From (\ref{eq25}), we get that $\eta_{21}=0$. If $\alpha_m'\neq 0$, then $\alpha'_m$ must be one of $\alpha_1'$, $\alpha_2'$, $\alpha_2''$ and $\alpha'_3$ by Theorem \ref{thm31}.

\begin{lemma} \label{lem410} Let $L$ be a Sch\"odinger-Virasoro Lie $H$-pseudoalgebra with $\eta=0$. Suppose $\alpha_m'\neq 0$.  Then $\eta_{11}=a_{00}\otimes 1$, $\eta_{21}=0$ and $\eta_{22}=d_{00}\otimes 1$.

Suppose  $\alpha_m'\in \{\alpha_2',\alpha_2'',\alpha_3'\}$.  Then $a_{00}=d_{00}=0$. Under the assumption  that $a_{00}=d_{00}=0$, we have the following

(i) If $\alpha_m'\in \{\alpha_1',\alpha_2',\alpha_2'',\alpha_3'\}$, then $\eta_{12}\neq 0$ described by (T1)-(T10), where $x_{ij}$ is replaced by $b_{ij}$.

Suppose $a_{00}\neq d_{00}$. Then either $\eta_{12}=0$, or $\eta_{12}=b_{00}\otimes 1$ for some nonzero $b_{00}$, or $\eta_{12}=b_{10}(s\otimes 1+2\otimes s)$ for some nonzero $b_{10}$. Moreover, we have the following

(ii) If  $\eta_{12}=0$, then $d_{00}=2a_{00}$, $\alpha_m'=\alpha_1',$ $\kappa_2=2\kappa_1$, $2\lambda_1-\lambda_2=1$, $\kappa_3=\lambda_3=0$.

(iii) If  $\eta_{12}=b_{00}\otimes 1\neq 0$, then $d_{00}=2a_{00}$, $\alpha_m'=\alpha_1',$ $\kappa_1=\kappa_2=\kappa_3=0$, $\lambda_1=\lambda_2=1$, $\lambda_3=0$.

(iv) If  $\eta_{12}=b_{10}(s\otimes 1+2\otimes s)\neq 0$, then $d_{00}=2a_{00}$, $\alpha_m'=\alpha_1',$ $\kappa_1=\kappa_2=\kappa_3=0$, $\lambda_1=\lambda_3=0$, $\lambda_2=-1$.
\end{lemma}

\begin{proof}From Lemma \ref{lem44a}, one obtains that $\eta_{11}=a_{00}\otimes 1$, $\eta_{22}=d_{00}\otimes 1$ for some $a_{00},d_{00}\in{\bf k}$.
Suppose $\alpha_m'=\alpha_1'\neq 0$. Then $d_{00}=2a_{00}$ by (\ref{eq26}). In addition, $\kappa_2=2\kappa_1$ and $2\lambda_1-\lambda_2=1$ by (\ref{eq218}).
If $\alpha_m'\in\{\alpha_2',\alpha_2'',\alpha_3'\}$, then $a_{00}=d_{00}=0$ by (\ref{eq26}). Let us assume that $\alpha_m'\in\{\alpha_1',\alpha_2',\alpha_2'',\alpha_3'\}$ and $a_{00}=d_{00}=0$.
Since we assume that  $\eta_{ij}$  ($1\leq i,j\leq 2$) are not all zero,  $\eta_{12}\neq 0$, which is described by (T1)-(T10), where $x_{ij}$ is replaced by $b_{ij}$.

Next, we assume that $a_{00}\neq d_{00}$. Then either $\eta_{12}=0$, or $\eta_{12}=b_{00}\otimes 1\neq 0$, or $\eta_{12}=b_{10}(-\kappa_1\otimes 1+s\otimes 1+2\otimes s)\neq 0$. In the case when $\eta_{12}=b_{00}\otimes 1\neq 0$, we have $\lambda_1-\lambda_2+\lambda_3=0$.
Since $\lambda_3=0$, $\lambda_1=\lambda_2=1.$ In the case when $\eta_{12}=b_{10}(-\kappa_1\otimes 1+s\otimes 1+2\otimes s)\neq 0$, we have $\lambda_1-\lambda_2+\lambda_3=1$. Thus $\lambda_1=0$ and $\lambda_2=-1$.
\end{proof}.

From the above lemma, we obtain the following theorem.
\begin{theorem} \label{thm410} Let $L$ be a Sch\"odinger-Virasoro Lie $H$-pseudoalgebra with the pseudobrackets given by (I*). Suppose $\eta=0$ and $\alpha_m'\neq 0$.  Then $L$ must be one of the following types.

(E1)\ $\eta_{11}=\eta_{21}=\eta_{22}=0$, $\alpha_m'\in \{\alpha_1',\alpha_2',\alpha_2'',\alpha_3'\}$, $\eta=0$, $\eta_{12}=b_{01}(-\kappa_3\otimes 1+1\otimes s)\neq 0$, $\beta_1=\lambda_1s\otimes 1-1\otimes s+\kappa_1\otimes 1$, $\beta_2=(\lambda_1+1)s\otimes 1-1\otimes s+(\kappa_1+\kappa_3)\otimes 1$, $\beta_3=-1\otimes s+\kappa_3\otimes 1$,
where $b_{01},\lambda_1,\kappa_1,\kappa_3\in{\bf k}$.

(E2)\  $\eta_{11}=\eta_{21}=\eta_{22}=0$, $\eta_{12}=b_{00}\otimes 1\neq 0$,  $\alpha_m'\in \{\alpha_1',\alpha_2',\alpha_2'',\alpha_3'\}$, $\eta=0$, $\beta_1=\lambda_1s\otimes 1-1\otimes s+\kappa_1\otimes 1$, $\beta_2=\lambda_2s\otimes 1-1\otimes s+\kappa_2\otimes 1$, $\beta_3=(\lambda_2-\lambda_1)s\otimes 1-1\otimes s+(\kappa_2-\kappa_1)\otimes 1$,
where $b_{00},\lambda_1,\lambda_2, \kappa_1,\kappa_2\in{\bf k}$.

(E3)\   $\eta_{11}=\eta_{21}=\eta_{22}=0$, $\eta_{12}=b_{11}((\kappa_1(\kappa_2-\kappa_1)+\lambda_1(\kappa_2-\kappa_1)^2)\otimes1-(\kappa_1+2\lambda_1(\kappa_2-\kappa_1))\otimes s+2\lambda_1\otimes s^{(2)}-(\kappa_2-\kappa_1)s\otimes1+s\otimes s)\neq 0$, $\alpha_m'\in \{\alpha_1',\alpha_2',\alpha_2'',\alpha_3'\}$,   $\beta_1=\lambda_1s\otimes 1-1\otimes s+\kappa_1\otimes 1$, $\beta_2=(\lambda_1-2)s\otimes 1-1\otimes s+\kappa_2\otimes 1$, $\beta_3=-1\otimes s+(\kappa_2-\kappa_1)\otimes 1$,
where $b_{11},\lambda_1, \kappa_1,\kappa_2\in{\bf k}$.

(E4)\   $\eta_{11}=\eta_{21}=\eta_{22}=0$, $\eta_{12}=b_{11}(\kappa_1(\kappa_2-\kappa_1)\otimes1-(\kappa_2-\kappa_1)s\otimes1-\kappa_1\otimes s+s\otimes s)\neq 0$,  $\alpha_m'\in \{\alpha_1',\alpha_2',\alpha_2'',\alpha_3'\}$, $\beta_1=-1\otimes s+\kappa_1\otimes 1$, $\beta_2=-2s\otimes 1-1\otimes s+\kappa_2\otimes 1$, $\beta_3=-1\otimes s+(\kappa_2-\kappa_1)\otimes 1$,
where $b_{11}, \kappa_1,\kappa_2\in{\bf k}$.

(E5)\   $\eta_{11}=\eta_{21}=\eta_{22}=0$, $\eta_{12}=b_{11}(-\kappa_3s\otimes 1+s\otimes s)\neq 0$,  $\alpha_m'\in \{\alpha_1',\alpha_2',\alpha_2'',\alpha_3'\}$,  $\beta_1=\lambda_1s\otimes 1-1\otimes s+\kappa_1\otimes 1$, $\beta_2=(\lambda_1-2)s\otimes 1-1\otimes s+\kappa_2\otimes 1$, $\beta_3=-1\otimes s+(\kappa_2-\kappa_1)\otimes 1$,
where $b_{11}, \lambda_1, \kappa_1,\kappa_2\in{\bf k}$.

(E6)\  $\eta_{11}=\eta_{21}=\eta_{22}=0$, $\eta_{12}=-(\kappa_1b_{10}+(\kappa_2-\kappa_1)b_{01})\otimes 1+b_{10}s\otimes 1+b_{01}\otimes s$,  $\alpha_m'\in \{\alpha_1',\alpha_2',\alpha_2'',\alpha_3'\}$,  $\beta_1=\lambda_1s\otimes 1-1\otimes s+\kappa_1\otimes 1$, $\beta_2=\lambda_2s\otimes 1-1\otimes s+\kappa_2\otimes 1$, $\beta_3=(\lambda_2-\lambda_1+1)s\otimes 1-1\otimes s+(\kappa_2-\kappa_1)\otimes 1$,
where $b_{11}, \lambda_1, \kappa_1,\kappa_2\in{\bf k}$ satisfying $\lambda_1b_{10}+(\lambda_2-\lambda_1+1)b_{01}=0$ and $b_{10}\neq 0$.

(E7)\  $\eta_{11}=\eta_{21}=\eta_{22}=0$, $\eta_{12}=b_{10}(-\kappa_1\otimes 1+s\otimes 1)\neq 0$,  $\alpha_m'\in \{\alpha_1',\alpha_2',\alpha_2'',\alpha_3'\}$,  $\beta_1=-1\otimes s+\kappa_1\otimes 1$, $\beta_2=\lambda_2s\otimes 1-1\otimes s+\kappa_2\otimes 1$, $\beta_3=(\lambda_2+1)s\otimes 1-1\otimes s+(\kappa_2-\kappa_1)\otimes 1$,
where $b_{10}, \lambda_1, \lambda_2, \kappa_1,\kappa_2\in{\bf k}$.

(E8)\  $\eta_{11}=\eta_{12}=\eta_{22}=0$, $\eta_{12}=\frac{b_{20}}{2(\lambda_2+2)}(\kappa_1((\lambda_2+2)\kappa_1+(\kappa_2-\kappa_1))\otimes1-\kappa_1\otimes s
-(2(\lambda_2+2)\kappa_1+(\kappa_2-\kappa_1))s\otimes1+s\otimes s+2(\lambda_2+2)s^{(2)}\otimes1)\neq 0$,  $\alpha_m'\in \{\alpha_1',\alpha_2',\alpha_2'',\alpha_3'\}$,  $\beta_1=-1\otimes s+\kappa_1\otimes 1$, $\beta_2=\lambda_2s\otimes 1-1\otimes s+\kappa_2\otimes 1$, $\beta_3=(\lambda_2+2)s\otimes 1-1\otimes s+(\kappa_2-\kappa_1)\otimes 1$,
where $b_{20}, \lambda_2, \kappa_1,\kappa_2\in{\bf k}$.

(E9)\  $\eta_{11}=\eta_{12}=\eta_{22}=0$, \begin{eqnarray*}\begin{array}{lll}\eta_{12}&=&b_{20}((\frac{\lambda_1}{2\lambda_1-2}(\kappa_2-\kappa_1)^2+\frac{1-2\lambda_1}{2-2\lambda_1}
\kappa_1(\kappa_2-\kappa_1)
+\frac12\kappa_1^2)\otimes1\\ &&
+(\frac{(2\lambda_1-1)(\kappa_2-\kappa_1)}{2-2\lambda_1}
-\kappa_1)s\otimes1
+\frac{(2\lambda_1-1)\kappa_1+2\lambda_1(\kappa_2-\kappa_1)}{2-2\lambda_1}\otimes s-\frac{\lambda_1}{1-\lambda_1}\otimes s^{(2)}\\ &&
+\frac{1-2\lambda_1}{2-2\lambda_1}s\otimes s+s^{(2)}\otimes1)\neq 0,\end{array}\end{eqnarray*}  $\alpha_m'\in \{\alpha_1',\alpha_2',\alpha_2'',\alpha_3'\}$,  $\beta_1=\lambda_1s\otimes 1-1\otimes s+\kappa_1\otimes 1$, $\beta_2=-s\otimes 1-1\otimes s+\kappa_2\otimes 1$, $\beta_3=(1-\lambda_1)s\otimes 1-1\otimes s+(\kappa_2-\kappa_1)\otimes 1$,
where $b_{20}, \lambda_1, \kappa_1,\kappa_2\in{\bf k}$.

(E10)\  $\eta_{11}=\eta_{12}=\eta_{22}=0$, \begin{eqnarray*}\begin{array}{lll}\eta_{12}&=&b_{30}(s^{(3)}\otimes 1-\frac12(\kappa_1+\kappa_2))s^{(2)}\otimes1+\frac12s^{(2)}\otimes s+\frac1{12}(\kappa_1^2+4\kappa_1\kappa_2+\kappa_2^2)s\otimes1\\
&&-\frac16(2\kappa_1+\kappa_2)s\otimes s+\frac16s\otimes s^{(2)}
-\frac1{12}(\kappa_1\kappa_2^2+\kappa_1^2\kappa_2)\otimes1\\&&
+\frac1{12}(\kappa_1^2+2\kappa_1\kappa_2)\otimes s
-\frac16\kappa_1\otimes s^{(2)})\neq 0,\end{array}\end{eqnarray*}  $\alpha_m'\in \{\alpha_1',\alpha_2',\alpha_2'',\alpha_3'\}$,  $\beta_1=-1\otimes s+\kappa_1\otimes 1$, $\beta_2=s\otimes 1-1\otimes s+\kappa_2\otimes 1$, $\beta_3=2s\otimes 1-1\otimes s+(\kappa_2-\kappa_1)\otimes 1$,
where $b_{30}, \lambda_1, \kappa_1,\kappa_2\in{\bf k}$.

(E11)\  $\eta_{11}=\eta_{12}=\eta_{22}=0$, $$\begin{array}{lll}\eta_{12}&=&b_{30}((s^{(3)}\otimes1-1\otimes s^{(3)})+\frac12(s^{(2)}\otimes s-s\otimes s^{(2)})-\frac12(\kappa_1+\kappa_2)s^{(2)}\otimes 1\\&&+\frac12(2\kappa_2-\kappa_1)\otimes s^{(2)}-\frac14(\kappa_1^2-4\kappa_1\kappa_2+\kappa_2^2)s\otimes1+\frac14(\kappa_1^2+2\kappa_1\kappa_2-2\kappa_2^2)\otimes s\\ &&+ \frac12(\kappa_2-2\kappa_1)s\otimes s +\frac1{12}(2\kappa_1-\kappa_2)(\kappa_1^2-\kappa_1\kappa_2-2\kappa_2^2)\otimes1),\end{array}$$ $\alpha_m'\in \{\alpha_1',\alpha_2',\alpha_2'',\alpha_3'\}$, $\beta_1=\frac23s\otimes 1-1\otimes s+\kappa_1\otimes 1$, $\beta_2=-\frac53s\otimes 1-1\otimes s+\kappa_2\otimes 1$, $\beta_3=\frac23s\otimes 1-1\otimes s+(\kappa_2-\kappa_1)\otimes 1$,
where $b_{30}, \kappa_1,\kappa_2\in{\bf k}$.

(E12)\  $\eta_{11}=\eta_{12}=\eta_{22}=0$, $$\begin{array}{lll}\eta_{12}&=&b_{21}((s^{(2)}\otimes s-s\otimes s^{(2)})+(\kappa_1\otimes s^{(2)}-(\kappa_2-\kappa_1)s^{(2)}\otimes1)
\\ && +(\frac12(\kappa_2-\kappa_1)(3\kappa_1-\kappa_2)s\otimes 1+\frac12\kappa_1(3\kappa_1-2\kappa_2)\otimes s)+(\kappa_2-2\kappa_1)s\otimes s\\  && +\frac12\kappa_1(\kappa_2-\kappa_1)(\kappa_2-2\kappa_1)\otimes 1)\neq 0,\end{array}$$  $\alpha_m'\in \{\alpha_1',\alpha_2',\alpha_2'',\alpha_3'\}$,  $\beta_1=-1\otimes s+\kappa_1\otimes 1$, $\beta_2=\lambda_2s\otimes 1-1\otimes s+\kappa_2\otimes 1$, $\beta_3=(\lambda_2+3)s\otimes 1-1\otimes s+(\kappa_2-\kappa_1)\otimes 1$,
where $b_{21}, \lambda_2, \kappa_1,\kappa_2\in{\bf k}$.

(E13)\  $\eta_{11}=\eta_{12}=\eta_{22}=0$, $$\begin{array}{lll}\eta_{12}&=&b_{21}(s^{(2)}\otimes s+3s\otimes s^{(2)}-(\kappa_2-\kappa_1)s^{(2)}\otimes 1-3(2\kappa_2-\kappa_1)\otimes s^{(2)}\\ &&-(3\kappa_2-2\kappa_1)s\otimes s+\frac12(\kappa_1^2-4\kappa_1\kappa_2+3\kappa_2^2)s\otimes 1+\frac12(\kappa_1^2-6\kappa_1\kappa_2+6\kappa_2^2)\otimes s\\ &&+\frac12(3\kappa_1^3-6\kappa_1^2\kappa_2-3\kappa_1\kappa_2^2-2\kappa_2^3)\otimes 1+6\otimes s^{(3)})\neq 0,\end{array}$$  $\alpha_m'\in \{\alpha_1',\alpha_2',\alpha_2'',\alpha_3'\}$, $\beta_1=2\lambda_1s\otimes 1-1\otimes s+\kappa_1\otimes 1$, $\beta_2=-s\otimes 1-1\otimes s+\kappa_2\otimes 1$, $\beta_3=-1\otimes s+(\kappa_2-\kappa_1)\otimes 1$,
where $b_{21}, \kappa_1,\kappa_2\in{\bf k}$.

(E14)\ $\eta_{11}=a_{00}\otimes 1\neq 0$, $\eta_{12}=\eta_{21}=0$, $\eta_{22}=2a_{00}\otimes 1$, $\alpha_m'=c_{01}(1\otimes s-s\otimes 1)\neq 0$,
$\beta_1=\lambda_1s\otimes 1-1\otimes s+\kappa_1\otimes 1$, $\beta_2=(2\lambda_1-1)s\otimes 1-1\otimes s+2\kappa_1\otimes 1$, $\beta_3=-1\otimes s$
for some $a_{00},w_{01},\lambda_1,\kappa_1\in{\bf k}$.

(E15)\ $\eta_{11}=a_{00}\otimes 1\neq 0$, $\eta_{12}=b_{00}\otimes 1\neq 0$, $\eta_{21}=0$, $\eta_{22}=2a_{00}\otimes 1$, $\alpha_m'=c_{01}(1\otimes s-s\otimes 1)\neq 0$,
$\beta_1=\beta_2=s\otimes 1-1\otimes s$,  $\beta_3=-1\otimes s$
for some $a_{00},b_{00},w_{01}\in{\bf k}$.

(E16)\ $\eta_{11}=a_{00}\otimes 1\neq 0$, $\eta_{12}=b_{10}(s\otimes 1+2\otimes s)\neq 0$, $\eta_{21}=0$, $\eta_{22}=2a_{00}\otimes 1$, $\alpha_m'=c_{01}(1\otimes s-s\otimes 1)\neq 0$,
$\beta_1=\beta_3=-1\otimes s$,  $\beta_2=-s\otimes 1-1\otimes s$
for some $a_{00},b_{10}, w_{01}\in{\bf k}$.
\end{theorem}

The extended Schr\"odinger Virasoro Lie conformal algebra defined in \cite{SY} (see Example \ref{exaa27}) is a pseudoalgebra described by (E14) of Theorem \ref{thm410}, where $w_{01}=-1$, $\lambda_1=\frac12$, $\kappa_1=0$ and $a_{00}=1$. Moreover, similar to Example \ref{ex33}, one can obtain a series of infinite-dimensional Lie algebras, which contain Virasoro algebra as a subalgebra and Schr\"odinger Lie algebra as a subalgebra.

\begin{example} Let $A={\bf k}[[t,t^{-1}]]$ be an $H$-bimodule given by $sf=fs=\frac{df}{dt}$ for any $f\in Y$. Then $A$ is an $H$-differential algebra both for the left and the right action of $H$. Let $\mathscr{A}_A(L_s)=A\otimes_HL$ and $\rho\in {\bf k}$. Set $L_n=t^{n+1}\otimes_He_0$, $Y_{p+\rho}=t^{p+1}\otimes_H e_1$, $M_{k+2\rho}=t^{k+1}\otimes_He_2$ and $N_m=t^{m+1}\otimes _He_3$ for any $n,m, p, k\in \mathbb{Z}$. Suppose that $\mathfrak{S}V_{\rho}$ is a vector space with a basis $\{L_n,Y_p,M_k,N_m|m, n\in\mathbb{Z}, p\in \rho+\mathbb{Z}, k\in 2\rho+ \mathbb{Z}\}$.

If  $L$ is of the type (E14) of Theorem \ref{thm410} with  $w_{01}=1$, then $\mathfrak{S}V_{\rho}$ is a Lie algebra with nonzero brackets given by
$$ [L_n,L_{n'}]=(n-n')L_{n+n'},\qquad [Y_p,  Y_{p'}]=(p-p')M_{p+p'},$$
\begin{eqnarray*} & [L_n, Y_p]=\left(\lambda_1(n+1)-p+\rho-1\right)Y_{n+p},\\
& [L_n, M_k]=\left((2\lambda_1-1)(n+1)-k+2\rho-1\right)M_{n+k},\\
&  [L_n, N_m]=-mM_{n+m},\quad [Y_p, N_m]=a_{00}Y_{p+m+1},\quad [M_k, N_m]=2a_{00}M_{k+m+1},\\
& [Y_p, M_k]=[M_k, M_{k'}]=[N_m, N_{m'}]=0.\end{eqnarray*}
From these,   the extended Schr\"odinger-Virasoro Lie algebra defined in \cite{U} is exactly the algebra $\mathfrak{S}V_{\rho}$ in the case when $\lambda_1=\rho=\frac12$ and $a_{00}=1$.

If  $L$ is of the type  (E15) of Theorem \ref{thm410} with $\eta=0$, then $\mathfrak{S}V_{\rho}$ is a Lie algebra with nonzero brackets  given by
\begin{eqnarray*} & [L_n, L_{n'}]=(n-n')L_{n+n'},\qquad [L_n, Y_p]=(\rho-1-p)Y_{n+p},\\
& [L_n, M_k]=(2\rho-2-n-k)M_{n+k},\qquad [M_k,N_m]=2a_{00}M_{k+m+1},\\
& [L_n, N_m]=(\rho-m-1)N_{n+m},\qquad
[Y_p, Y_{p'}]=(p-p')M_{p+p'},\\
& [Y_p, N_m]=a_{00}\kappa_1Y_{p+m+1+\rho}+(p+2m+3-\rho)M_{p+m+\rho},\\
& [Y_p, M_k]=[M_k, M_{k'}]=[N_m, N_{m'}]=0.\end{eqnarray*}
\end{example}

\end{document}